\documentclass{article}
\usepackage{graphicx, fullpage} 

\usepackage[T1]{fontenc}
\usepackage{xcolor}
\usepackage{amssymb}
\usepackage{amsthm}
\usepackage{amsmath}
\usepackage{bm}
\usepackage{tikz}
\usepackage{lineno}
\usepackage{multirow}
\usepackage{empheq}

\usepackage{svg}

\usepackage{hyperref,bookmark}
\usetikzlibrary{arrows.meta}
\usetikzlibrary{arrows}
\newtheorem{theorem}{Theorem}
\newtheorem{lemma}[theorem]{Lemma}
\newtheorem{remark}{Remark}
\newtheorem{definition}[theorem]{Definition}

\DeclareMathOperator{\Id}{Id}

\newif
\ifmarkedup

\markedupfalse

\newcommand{\blue}[1]{{\color{black} #1}}

\ifmarkedup
    
    \renewcommand{\blue}[1]{{\color{blue} #1}}
    
\fi

\DeclareMathOperator{\diver}{div}
\newcommand{\grad}{\textbf{\textup{grad}}}

\newcommand{\curldeuxD}{\textup{curl}_{2D}}
\newcommand{\curl}{\textbf{\textup{curl}}}

\renewcommand{\d}{\textup{d}}
\newcommand{\dd}{\d}
\newcommand{\dsp}{\displaystyle}

\title{Structure-preserving space discretization of differential and nonlocal constitutive relations for port-Hamiltonian systems}

\author{Antoine Bendimerad-Hohl \and Ghislain Haine \and Laurent Lef\`{e}vre \and Denis Matignon}

\begin{document}

\maketitle

\begin{abstract}
We study the structure-preserving space discretization of port-Hamiltonian (pH) systems defined with differential constitutive relations. Using the concept of Stokes-Lagrange structure to describe these relations, these are reduced to a finite-dimensional Lagrange subspace of a pH system thanks to a structure-preserving Finite Element Method. 

To illustrate our results,  the $1$D nanorod case and the shear beam model are considered, which are given by differential and implicit constitutive relations for which a Stokes-Lagrange structure along with boundary energy ports naturally occur.

Then, these results are extended to the nonlinear $2$D incompressible Navier-Stokes equations written in a vorticity--stream function formulation. It is first recast as a pH system defined with a Stokes-Lagrange structure along with a modulated Stokes-Dirac structure. A careful structure-preserving space discretization is then performed, leading to  a finite-dimensional pH system. Theoretical and numerical results show that both enstrophy and kinetic energy evolutions are preserved both at the semi-discrete and fully-discrete levels.

\textbf{keywords:}
    Port-Hamiltonian systems -- Structure-preserving discretization -- Nonlocal constitutive relations -- Euler-Bernoulli beam -- incompressible Navier-Stokes equation
\end{abstract}



\section{Introduction}

Port-Hamiltonian (pH) systems have been developed as a theoretical framework for modeling, numerical simulation, and control of macroscopic multiphysics systems. PH systems are a generalization of classical Hamiltonian systems to open systems, embedding additional pairs of port variables. These are interface variables that are necessary either for control, observation, or interconnection. This extension leads to the intrinsic representation of pH systems with a geometric Dirac interconnection structure, generalizing the Poisson bracket from classical Hamiltonian systems. The Dirac interconnection structure may include dissipation, input/output (control/observation) and interconnection ports, besides the storage (energy) port. It is geometrically defined through a canonical power pairing, extending the Poisson symplectic form of Hamiltonian mechanics. It may be represented, either explicitly or implicitly in various coordinate systems. We refer to \cite{van2014port} for a general introduction, to \cite{duindam2009modeling} for a survey of applications in different fields of physics, to \cite{beattie2018linear} for the formulation of pH systems with algebraic constraints as differential-algebraic equations (DAE) or linear descriptor systems in matrix representation, and to \cite{mehrmann2023control} for a survey of control applications of these pH-DAE systems. 

For infinite-dimensional pH systems, Stokes-Dirac structures have been introduced in \cite{van2002Hamiltonian}, extending the in-domain port variables pairs with pairs of boundary port variables, which are necessary to write the usual (mass, energy, momentum, etc.) balance equations. Skew symmetry in the Hamiltonian differential operator, integration by parts, and Stokes theorem may be combined to obtain a canonical representation for these infinite-dimensional pH systems and define appropriate pairs of boundary port variables in which physically consistent boundary conditions, boundary control and observation variables may be expressed.  In the linear case, the semigroup approach applied to distributed parameter linear pH systems allows a parametrization of all admissible boundary conditions (using these pairs of boundary port variables), which guarantee well-posedness of the resulting partial differential equations (PDE) (see \cite{le_gorrec_2005_siam} for the 1D case or \cite{brugnoli2023stokes} for the $n$D case). The literature on distributed pH systems has recently grown considerably. Both theoretical developments and application-based studies are surveyed in \cite{rashad2020twenty}. Recently, pH models for distributed parameter systems (either linear or not) have been proposed and investigated in various application domains, ranging from flexible mechanics (see \cite{ponce2024systematic} for a systematic methodology for pH modeling of flexible systems) to fluid dynamics (see for instance
\cite{cardoso2024port} for a literature review on pH modeling in fluid dynamics), electromagnetism \cite{farle2013port}, thermo-magneto-hydrodynamics (see for instance \cite{vu2016structured}, for the pH model of tokamak plasma dynamics), and phase separation \cite{vincent2020port,bendimerad2023structure}, just to cite a few examples. 

In most of these applications (and previously cited papers), the boundary pH formulation relies on the assumption that the
Hamiltonian differential operator, together with a generating Hamiltonian functional, are explicitly known and only depend on the state variables and not on their spatial derivatives. However, constraints may stem from implicit energy constitutive equations, affecting the relationship between effort and energy state variables. In this case, pH systems may be defined on a Lagrangian subspace (see \cite{van2018generalized}) or submanifold (for the nonlinear case, see \cite{van2020dirac}), as recently introduced in the finite-dimensional case. The energy is then no longer defined by a function on the state space, but instead by reciprocal constitutive relations between the state and co-state variables. Stokes-Lagrange subspaces have been proposed to cope with the infinite-dimensional case (see \cite{krhac2024port} or \cite{maschke2023linear} for 1D formulation and \cite{bendimerad2025stokeslagrange} for $n$D extension of this idea). It is then possible to represent systems where the Hamiltonian functional is constrained, implicitly defined and/or depends on spatial derivatives of the state. Examples of pH systems defined on Stokes-Lagrange subspaces include the implicit formulation of the Allen-Cahn equation \cite{yaghi2022port}, the Dzektser equation governing the seepage of underground water (see \cite{Dzektser72} for the original model or \cite{bendimerad2025stokeslagrange,jacob2022solvability} for pH formulations), or nonlocal visco-elastic models for nanorods \cite{heidari2019port,heidari2022nonlocal}. The authors have also shown how constrained pH systems defined on Stokes-Lagrange subspaces may formalize the reduction of energy constitutive equations, considering the simplification from the Reissner-Mindlin thick plate model to the Kirchhoff-Love thin plate model, and the low-frequency approximation of Maxwell’s equations, as two examples of such a reduction \cite{bendimerad2025stokeslagrange}.

At the same time as these advances in the pH representation of an increasing variety of physical systems, corresponding structure-preserving numerical methods have been proposed. Among others, preserving the structure in the discretization allows running under-resolved simulation and avoid some spurious non-physical modes in multiphysics systems interconnection. Both are of prime importance for pH systems which are fundamentally a modular modeling approach for (possibly real-time) control oriented applications. In this paper, we focus on space discretization of infinite-dimension pH systems. Time discretization for finite dimensional pH systems (among others, those obtained by structure-preserving spatial discretization) has been investigated in many previous works. Most of the time, discrete gradient methods \cite{mclachlan1999geometric,gonzalez1996time} are used such as in \cite{aoues2017hamiltonian,celledoni2017energy}, or symplectic RK/collocation schemes, such as in \cite{kotyczka2019discrete} (for explicit pH systems) or in \cite{mehrmann2019structure,kinon2025discrete} (for pH-DAEs). Note also that complete discretization schemes, such as space-time finite element method for multisymplectic Hamiltonian systems \cite{CelledoniJCP2021} or weak space-time formulation for dissipative Hamiltonian systems \cite{egger2019structure}, have been investigated yet and could possibly be extended to open pH systems. 

Regarding the structure-preserving spatial discretization of pH systems, various approaches have been proposed, including finite differences on staggered grids \cite{TrenchantJCP2018}, pseudo-spectral methods \cite{MoullaJCP2012,VU20171}, methods based on the finite discrete exterior calculus \cite{seslija2012discrete,HiemstraJCP2014} and finite element methods. Among these approaches, only finite element methods have been successfully applied to a large number of different 2D and 3D examples, either linear or not, with non-trivial  geometries (in that particular sense, they may be regarded as more general). Finite elements have been introduced for the structure-preserving discretization of 1D transmission line problems in \cite{golo2004hamiltonian}. The idea was then generalized for non linear shallow water equations in \cite{pasumarthy2012port}. These methods are mixed finite element methods (MFEM) such as those analyzed in \cite{BofBreFor15}, but extended for pH systems with time-varying boundary energy flows. Different directions have been followed to generalize these early 1D MFEM results. In \cite{kotyczka2018weak}, structure-preserving is obtained by choosing different compatible finite element spaces for the effort and flow variables approximations in the weak formulation of pH systems. In \cite{BRUGNOLI_JGP_2022} (generalized from \cite{zhangmass2022}), the authors define a virtual and redundant dual field representation of the original pH system to obtain the finite dimensional without the need to discretize the Hodge metric in the constitutive equations. Finally, in \cite{Cardoso-Ribeiro2020b}, a partitioned finite element method (PFEM) is proposed, based on the integration by parts of some of the balance equations, chosen to make the boundary control appear and obtain a finite-dimensional square invertible pH matrix representation with sparse matrices. Note that the PFEM is a particular instance of the MFEM applied to pH systems, where the partition of the variables is a natural feature. We propose in this paper a new application of the PFEM to pH systems defined on Stokes-Lagrange subspaces.

The PFEM is defined and applied to shallow water equations (in 1D and 2D), together with various extensions examples in \cite{Cardoso-Ribeiro2020b}. Since then, it has been applied to many different application examples, either in 2D or 3D, and has been proven to preserve the central Dirac interconnection structure of pH systems (hence, the power continuity). Among these examples, one can cite \cite{serhani2019anisotropic} for the anisotropic heterogeneous $n$D heat equation,  \cite{bendimerad2022structure,bendimerad2023structure} for the Cahn-Hillard and Allen-Cahn equations, \cite{haine2022structure} for the Maxwell equations, \cite{brugnoli2021mixed} for von Karman plate models, \cite{brugnoli2021port} for a linear thermo-elasticity model (through the interconnection of elasto-dynamics and heat equations), \cite{brugnoli2019portII} for Mindlin and Kirchhoff plate models, \cite{haine2021incompressible} for the incompressible Navier-Stokes equation, \cite{cardoso2024rotational} for 2D shallow water equations, or \cite{cardoso2024port} for a survey on applications of the PFEM to fluid models. The optimal choice of finite elements for the PFEM has been analyzed for the 2D wave equation on different geometries and asymptotics for the error are given in \cite{haine2023numerical}. An open simulation Python package for multiphysics pH systems relying on the PFEM is available at \url{https://g-haine.github.io/scrimp/}, with various examples. This package is further discussed with the help of some mechanical and thermodynamical examples in \cite{ferraro2024simulation}. Note also that a benchmark of pH numerical models with a repository of structure-preserving discretization methods may be found at \url{https://algopaul.github.io/PortHamiltonianBenchmarkSystems.jl/}.

This paper deals with the structure-preserving discretization of $n$D pH systems with implicitly defined energy (i.e., with differential constitutive equations for the energy). They are represented with two geometrical interconnection structures: a Stokes-Dirac structure for power continuity (balance equations) and a Stokes-Lagrange structure for the energy constitutive equations. Until now, to the best of our knowledge, the PFEM discretization of pH system with implicitly defined energy has been applied only in \cite{bendimerad2023implicit}, for a simple 1D nanorod viscoelastic example. In this work, we propose to extend these early results to more fundamental 1D and 2D examples, where the nonlocal constitutive equations play a fundamental role, to discretize these examples and analyze numerical consequences, of both the implicit pH formulation and the structure preservation, on simulation numerical results.

In section \ref{sec:sec2}, the geometric definition of pH systems with implicitly defined energy is recalled. Dirac structure and Lagrange subspace (finite dimensional case, section \ref{sec:sec2.1}) are defined and generalized to Stokes-Dirac and Stokes-Lagrange structures (infinite dimensional case, section \ref{sec:sec2.2}). In section \ref{sec:continuous-models}, three examples are considered: a 1D nanorod model (section \ref{sec:nanorod}), an implicit Euler-Bernoulli model (section \ref{sec:EBbeam}) and a 2D incompressible Navier-Stokes model (section \ref{sec:incompressibleNSE}). For the nanorod example, nonlocal constitutive equations are made necessary to cope with nonlocal effects between stress and strain variables at micro/nano-scales~\cite{eringen1983differential}. An equivalent implicit formulation with Robin boundary conditions and additional energy boundary port variables (see \cite{krhac2024port}) are introduced, which pop up in the power balance equation and in the Hamiltonian function. For the 1D Euler-Bernoulli beam, we consider an implicit model which arises when the (usual) simplification in the limit of low wavenumbers is not assumed  \cite{ducceschi2019conservative}. This extended Euler-Bernoulli model is made necessary for high frequencies approximation \cite{BilbaoJASA2016}. Regarding the homogeneous implicit 2D (or 3D) incompressible Navier-Stokes equation \cite{boyer2012mathematical}, we derive an equivalent model in terms of vorticity and stream function. These two variables are also related through a differential constitutive equation. The 2D case is cast as a pH system with implicitly defined energy function. This formulation leads to both power and enstrophy balance equations (the vorticity generation does not appear in the power balance equation). The case of no-slip boundary conditions is carefully investigated, with enstrophy creation and kinetic energy dissipation arising in the balance equations. In section \ref{sec:PFEM}, the structure-preserving spatial discretizations of the previously defined examples are derived. This leads to the definition of discrete Dirac and Lagrange structures, corresponding to their infinite-dimensional Stokes-Dirac and Stokes-Lagrange counter-parts. It is shown how the PFEM approach extends to these cases and ensures structural properties of the finite-dimensional pH matrices which guarantees power, energy (and enstrophy) conservation. In section \ref{sec:numerics}, numerical results are presented for the three examples, in each case, the evolution of the Hamiltonian during the simulation is shown along with the power balance. In the Euler-Bernoulli case, phase velocities are plotted against theoretical values. Similarly, incompressible Navier-Stokes results are compared to the benchmark~\cite{clercx2006normal} and MEEVC scheme~\cite{de2019inclusion}. 

\section{Dirac and Lagrange structures}
\label{sec:sec2}

The geometric definition of pH systems introduces two structures: the Dirac structure that describes the energy routing of the system, and the Lagrange structure that describes the Hamiltonian of the system. In the linear time-invariant finite-dimensional case, these structures correspond to algebraic restrictions on the matrices defining the dynamics, and ensure that the system satisfies certain properties such as passivity and Maxwell's reciprocity conditions.

In this article, we define the Dirac and Lagrange structures following \cite{gernandt2022equivalence}.

\subsection{Finite-dimensional case}
\label{sec:sec2.1}
Let us denote by $(\alpha,e) \in \mathbb{R}^{n} \times \mathbb{R}^{n}$ the state and costate of our system, resp., $(f_R,e_R) \in \mathbb{R}^{r} \times \mathbb{R}^{r}$ the resistive port, $(u_D,y_D) \in \mathbb{R}^{n_D} \times \mathbb{R}^{n_D}$ the power control port, and $(u_L,y_L) \in \mathbb{R}^{n_L} \times \mathbb{R}^{n_L}$ the energy control port\footnote{Indices stand for \textbf{R}esistive, \textbf{D}irac and \textbf{L}agrange.}. We dropped the time dependency to improve readability, as our primary focus is on linear time-invariant systems.

\begin{definition} \label{def:linear-ph-state-repr}Given a matrix $J\in \mathcal{M}_{n+r+n_D}(\mathbb{R})$, such that $J = - J^\top$, two matrices $P,S \in \mathcal{M}_{n+n_L}(\mathbb{R})$, such that $P^\top S = S^\top P$ and $\text{\rm rank}\begin{bmatrix} P \\ S \end{bmatrix} = n + n_L$,  and a matrix $R \in \mathcal{M}_{r}(\mathbb{R})$, such that $R = R^\top \geq 0$; the associated linear pH system reads:
\begin{subequations}\label{eqn:discr-linear-phs}
    \begin{align}  
     \begin{pmatrix}
        \frac{\d}{\d t} \alpha  \\ f_R \\ -y_D
    \end{pmatrix} &= J \begin{pmatrix}
        e \\ e_R \\ u_D
    \end{pmatrix} & \text{with }J=-J^\top, \label{eqn:discr-linear-phs:DIRAC}\\
    P^\top \begin{pmatrix}
        e \\ y_L
    \end{pmatrix}&= S^\top\begin{pmatrix}
        \alpha \\ u_L
    \end{pmatrix} & \text{with } \left\{\begin{aligned}
        P^\top S &= S^\top P, \\
    \text{\rm rank}\begin{bmatrix}
    P \\ S
\end{bmatrix} &= n + n_L,
    \end{aligned}\right. \label{eqn:discr-linear-phs:LAGRANGE}\\
    e_R &= R \, f_R & \text{with } R= R^\top \geq0. \label{eqn:discr-linear-phs:RESISTIVE}
\end{align}
\end{subequations}
$J$ is called the Dirac structure matrix, $P$ and $S$ are the Lagrange structure matrices, and $R$ is the resistive structure matrix.

Moreover, \eqref{eqn:discr-linear-phs:DIRAC} corresponds to the \blue{\emph{flow-effort interconnection}} (Dirac structure), \eqref{eqn:discr-linear-phs:LAGRANGE} corresponds to the \emph{constitutive relations} (Lagrange structure), and \eqref{eqn:discr-linear-phs:RESISTIVE} is the relation (resistive structure) between the dissipative flow and effort variables $f_R$ and $e_R$.
\end{definition}

\begin{remark}
    Usually, constitutive relations are explicit, i.e., $P=\Id$, and the energy control port is absent, i.e., $n_L=0$. In such a case, defining $Q:=S$, the typical linear pH system of the following form is retrieved:
\begin{subequations} \label{eqn:discr-linear-phs-explicit-no-energy-control}
    \begin{align}
     \begin{pmatrix}
        \frac{\d}{\d t} \alpha  \\ f_R \\ -y_D
    \end{pmatrix} &= J \begin{pmatrix}
        e \\ e_R \\ u_D
    \end{pmatrix} & \text{with }J=-J^\top, \label{eqn:discr-linear-phs-no-energy-control:DIRAC}\\
    e &= Q\alpha  & \text{with } Q=Q^\top > 0, \label{eqn:discr-linear-phs-no-energy-control:LAGRANGE}\\
    e_R &= R \, f_R & \text{with }R= R^\top \geq0. \label{eqn:discr-linear-phs-no-energy-control:RESISTIVE}
\end{align}
\end{subequations}
\end{remark}

In the general case, however, $P \neq \Id$, and the constitutive relations are subsequently specified as \emph{implicit}. Nevertheless, thanks to the symmetry of $P^\top S$ and the rank condition of $\begin{bmatrix} P \\ S \end{bmatrix}$, the following holds: 
\begin{lemma}\label{def:latent-state-ph}
    Let $P$ and $S$ as in Definition \ref{def:linear-ph-state-repr}. Let $\alpha,e \in \mathbb{R}^n$ and $y_L,u_L\in \mathbb{R}^{n_L}$, satisfying \eqref{eqn:discr-linear-phs:LAGRANGE}. Then, there exists a unique pair $(\lambda, \tilde u_L ) \in \mathbb{R}^{n} \times \mathbb{R}^{n_L}$ such that:
    \begin{equation}  \label{eqn:latent-state-control-def}
    \begin{pmatrix}
    \alpha \\ u_L
\end{pmatrix} = P \begin{pmatrix}
    \lambda \\ \widetilde u_L
\end{pmatrix}, \quad \begin{pmatrix}
    e \\ y_L
\end{pmatrix} = S \begin{pmatrix}
    \lambda \\ \widetilde u_L
\end{pmatrix}.
\end{equation}
These variables are called \emph{latent state} and \emph{latent control}, respectively.
\end{lemma}
\begin{proof}
    Proof is to be found in~\ref{proof:latent-state-ph}.
\end{proof}
\blue{
\begin{remark}
In general, the latent state is a mixture of energy and co-energy variables~\cite{van2018generalized}, while the latent control is an additional control variable.
\end{remark}
}

As it is shown in the next lemma, expressing \eqref{eqn:discr-linear-phs} in latent variables allows the definition of a Hamiltonian that satisfies a power balance.
\begin{lemma} \label{lemma:latent-state-hamiltonian-definition}
    Let $J$, $P$, $S$ and $R$ as in Definition~\ref{def:linear-ph-state-repr}. Decomposing $P = \begin{bmatrix}
        P_{1,1} & P_{1,2} \\ P_{2,1} & P_{2,2}
    \end{bmatrix}$ into blocks of appropriate sizes (precisely, $P_{1,1} \in \mathbb{R}^{2n}$, $P_{2,2} \in \mathbb{R}^{2n_L}$), the associated linear pH system~\eqref{eqn:discr-linear-phs} written in latent variables reads:
\begin{subequations}\label{eqn:discr-linear-phs-latent}
    \begin{align} 
     \begin{pmatrix}
        \frac{\d}{\d t}(P_{1,1} \, \lambda + P_{1,2} \, \widetilde u_L)  \\ f_R \\ -y_D
    \end{pmatrix} &= J \begin{pmatrix}
        e \\ e_R \\ u_D
    \end{pmatrix}, \label{eqn:discr-linear-phs-latent:DIRAC}\\
    \begin{pmatrix}
        e \\ y_L
    \end{pmatrix}&= S \begin{pmatrix}
        \lambda \\ \widetilde u_L
    \end{pmatrix} ,  \label{eqn:discr-linear-phs-latent:LAGRANGE}\\
    e_R &= R \, f_R.  \label{eqn:discr-linear-phs-latent:RESISTIVE}
\end{align}
\end{subequations}
The corresponding Hamiltonian is defined as:
\begin{equation}\label{eqn:linear-phs-power-balance}
    \mathcal H(\lambda,\widetilde u_L) := \frac{1}{2} \begin{pmatrix}
        \lambda \\ \widetilde u_L
    \end{pmatrix}^\top P^\top S \begin{pmatrix}
        \lambda \\ \widetilde u_L
    \end{pmatrix},
\end{equation}
and satisfies the power balance:
\begin{equation} \label{eqn:latent-state-ph-power-balance}
    \frac{\d}{\d t }\mathcal H(\lambda, \widetilde u) = e^\top \dot \alpha + y_L^\top \dot { u}_L = \underbrace{- f_R^\top R \, f_R}_{\leq0} + y_D^\top u_D  + y_L^\top \dot { u}_L \leq  y_D^\top u_D  + y_L^\top \dot { u}_L .
\end{equation}
\end{lemma}
\begin{proof}
    The proof is to be found in~\ref{proof:latent-state-hamiltonian-definition}.
\end{proof}

\begin{remark}
    Assuming that the energy control port is absent, i.e., that $n_L = 0$, one retrieves the port-Hamiltonian descriptor representation found in, e.g., \cite{beattie2018linear}.
\end{remark}
\begin{remark}
    The Dirac structure matrix may depend on the state $\alpha$, it is then called a \emph{modulated} Dirac structure. In particular, the derivation of the power balance is identical. In Section~\ref{sec:continuous-models}, the incompressible Navier-Stokes equation is cast in such a form.
    
    Additionally, Lagrange structure~\cite{van2020dirac} can describe \emph{nonlinear} constitutive relations. The structure enforces the Maxwell's reciprocity conditions\blue{, which establish that the response in a medium is symmetric with respect to interchange of source and detector~\cite{gangi2000constitutive}; this is} a necessary and sufficient condition for the existence of a Hamiltonian~\blue{\cite{bendimerad2025stokeslagrange}}.
\end{remark}

\subsection{A toy model: a moving mass}

In order to fix ideas and vocabulary, let us construct a simple worked out example.

The classical~\cite{van2014port} port-Hamiltonian formalism in the linear case, \textit{i.e.}, with a quadratic Hamiltonian $\mathcal H(\alpha) := \frac{1}{2} \alpha^\top Q \alpha$, reads:
$$
\left\lbrace
\begin{array}{ll}
\frac{\dd}{\dd t} \alpha &= J Q \alpha + B u_D, \\
y_D &= B^\top Q \blue{\alpha}.
\end{array}
\right.
$$
where $\alpha$ is the \emph{energy variable} (the variable of the Hamiltonian which often represents an energy), $J$ is the \emph{interconnection matrix} (because it \blue{couples the efforts} $e := \grad_\alpha \mathcal H = Q \alpha$ \blue{with the flows $\frac{\dd}{\dd t} \alpha$}), it is skew-adjoint and its graph\footnote{more precisely the graph of its inverse, see~\cite[Exercise 1, p. 17]{van2014port}} is the Dirac structure the \emph{trajectories} belong to, and $Q$ is a positive-definite self-adjoint matrix that relates (through constitutive relations) energy variables to \emph{co-energy variables} $e = Q \alpha$ (\textit{i.e.}, the derivative of the Hamiltonian with respect to $\alpha$). Finally, $B$ is a \emph{control matrix}, and $u_D$ and $y_D$ are the control (or input) and, respectively, the observation (or output). Together, $u_D$ and $y_D$ constitutes the \emph{power control port} as it allows to control the behavior of the \emph{power} (because the Hamiltonian is thought as an \emph{energy}) flowing into the system. Indeed, the power balance is trivially:
$$
\frac{\dd}{\dd t} \mathcal H = y_D^\top u_D.
$$

Now, consider a mass $m$ moving along a fixed $x$-axis, which can be controlled according to Newton's second law: $F = m a$, with $F$ the applied force and $a$ the resulting acceleration. \blue{Taking the $x$-axis constraint into account} implies that such a simple system does not fit into the classical framework briefly recalled above. Indeed, \blue{the constraint makes the potential energy due to gravity constant, hence, independent of the position $q$ of the mass. Therefore, it leads to} $Q$ being only a non-negative self-adjoint matrix. Lagrangian subspace allows us to go further to include this simple example.

Let us denote $q$ the position of the mass on the $x$-axis, $p := m \dot q$ its linear momentum, with $\dot q$ its velocity. The \blue{total} energy\blue{, the sum of kinetic and potential energies,} gives a Hamiltonian $\mathcal H := \frac{p^2}{2m} \blue{+ c}$. Taking the position and the linear momentum as energy variables (as usual in Hamiltonian mechanics), the co-energy variables are respectively $e_q := 0$ and $e_p := q$. Clearly, the problem lies in the $q$-independence of the Hamiltonian \blue{induced by the constraint}. Nevertheless, one can write:
$$
\frac{\dd}{\dd t} \underbrace{\begin{pmatrix} q \\ p \end{pmatrix}}_{\alpha} = \underbrace{\begin{bmatrix} 0 & 1 \\ -1 & 0 \end{bmatrix}}_{J} \underbrace{\begin{pmatrix} 0 \\ \dot q \end{pmatrix}}_{e} + \underbrace{\begin{bmatrix} 0 \\ 1 \end{bmatrix}}_{B} \underbrace{F}_{u}.
$$

Now, observe that:
$$
\begin{pmatrix}
\alpha \\ e
\end{pmatrix}
=
\begin{pmatrix}
q \\ m \dot q \\ 0 \\ \dot q
\end{pmatrix}
=
\begin{bmatrix}
1 & 0 \\ 0 & 1 \\ 0 & 0 \\ 0 & \frac{1}{m}
\end{bmatrix}\begin{pmatrix}
q \\ m \dot q
\end{pmatrix}
=
\begin{bmatrix} P \\ S \end{bmatrix}
\underbrace{\begin{pmatrix} \lambda_q \\ \lambda_p \end{pmatrix}}_{z},
$$
and the $\lambda$ variables are now the \emph{latent variables} that allow us to recover\blue{, up to the additive constant $c$,} the Hamiltonian $\mathcal H = \frac{1}{2} \lambda^\top P^\top S \lambda$ from the matrices $P$ and $S$, which define a Lagrangian subspace thanks to the following properties: $P^\top S = S^\top P$ and $\text{\rm rank} \begin{bmatrix} P \\ S \end{bmatrix} = 2$.

Altogether, this guarantees that the moving mass is a port-Hamiltonian system in the sense of Definition~\ref{def:linear-ph-state-repr} (without resistive port, nor energy control port)\blue{, and incorporate \emph{intrinsically} the constraint into the system}. \blue{Another way to model this system would have been to consider only the kinetic energy and to set $J = 0$ and $Q=1$, or in other words, to consider the constraint \emph{extrinsically} by omitting potential energy.}

Assume now a coefficient $k\ge0$ modeling a friction proportional to the velocity of the mass, \textit{i.e.}, a friction force $F_R := -k \dot q$. Then, such a dissipative term can be added to the previous model by adding a \emph{resistive port} $(f_R, e_R)$ as follows:
$$
\begin{pmatrix} \frac{\dd}{\dd t} \alpha \\ f_R \end{pmatrix} 
= 
\begin{bmatrix} J & \begin{matrix} 0 \\ -1 \end{matrix} \\ \begin{matrix} 0 & 1 \end{matrix} & 0 \end{bmatrix} \begin{pmatrix} e \\ e_R \end{pmatrix} + \begin{bmatrix} 0 \\ 1 \\ 0 \end{bmatrix} F,
\quad e_R = k f_R.
$$

Finally, assume now that we want to observe the position of the mass rather than its force or velocity (\textit{e.g.}, because this output is needed for interconnection with another system); it can be achieved by the introduction of an \emph{energy control port} (as it controls the behavior of the Hamiltonian directly instead of its time-derivative), as follows. Let us define $(y_L, u_L)$ the energy control port by:
$$
\begin{pmatrix}
\alpha \\ u_L \\ e \\ y_L
\end{pmatrix}
=
\begin{pmatrix}
q \\ m \dot q \\ u_L \\ 0 \\ \dot q \\ y_L
\end{pmatrix}
=
\begin{bmatrix}
1 & 0 & 0 \\ 0 & 1 & 0 \\ 0 & 0 & 1 \\ 0 & 0 & 1 \\ 0 & \frac{1}{m} & 0 \\ 1 & 0 & 0
\end{bmatrix}\begin{pmatrix}
q \\ m \dot q \\ \widetilde u_L
\end{pmatrix}
=
\begin{bmatrix} \widetilde P \\ \widetilde S \end{bmatrix}
\underbrace{\begin{pmatrix} \lambda_q \\ \lambda_p \\ \widetilde u_L \end{pmatrix}}_{\widetilde z},
$$
where $\widetilde u_L$ is the \emph{latent control}. Then, a Hamiltonian reads:
$$
\widetilde{\mathcal H} = \frac{1}{2} \widetilde \lambda^\top \widetilde P^\top \widetilde S \widetilde \lambda = \frac{1}{2} \lambda^\top P^\top S \lambda + \frac{1}{2} \lambda_q^\top \widetilde u_L,
$$
and the associated power balance becomes:
$$
\frac{\dd}{\dd t} \widetilde{\mathcal H} = - k f_R^2 + y_D^\top u_D + \dot y_L^\top u_L.
$$

Beside its simplicity, this example motivates the energetic and modular approach provided by Dirac structure and Lagrangian subspace. Furthermore, this geometric standpoint allows one to switch input and output. For instance, one could \emph{control} the position of the mass in the previous example by only switching the roles of $u_L$ and $y_L$. The price to pay is that the control becomes a constraint that has to be satisfied, \textit{e.g.}, using a Lagrange multiplier, for the resolution. \blue{Nevertheless, it is important to understand that integrating the constraint into the system using Lagrangian subspace gave access to the position $q$ for a control perspective, which would have been far more demanding otherwise.}

In summary, Dirac structures are meant for power-preserving exchanges between each elements of the system (through the interconnection matrix), while Lagrange structure enforce Maxwell's reciprocity conditions for the constitutive relations ($P^\top e = S^\top \alpha$), from which a Hamiltonian satisfying the expected power balance can be derived, thanks to the latent variables. 

\subsection{Infinite-dimensional case}
\label{sec:sec2.2}

Difficulties arise when extending the pH framework to distributed pH systems, \textit{i.e.}, when the system's states also depend on a space variable. In particular, boundary terms and operator domains have to be taken into account. In this section, only a brief overview of the matter is given, the reader may refer to~\cite{jacob2012linear} for an extensive study along with well-posedness results. In particular, Stokes-Dirac structures~\cite{kurula2010dirac} have been studied extensively as a special case of Dirac structures, where integration by parts (Stokes' theorem) is encoded in the subspace, this allows the corresponding boundary terms to be expressed as boundary control port~\cite{brugnoli2023stokes}.

Recently, Stokes-Lagrange structures have been defined on $1$D domains~\cite{maschke2023linear} and $n$D domains~\cite{bendimerad2025stokeslagrange} in order to adapt the Lagrange structure approach to infinite dimensional systems. Stokes-Lagrange structures are similar to Stokes-Dirac structures as they take into account the integration by parts. Moreover, they allow for the treatment of implicit constitutive relations in the distributed pH framework.

\begin{definition}
A Stokes-Dirac structure (respectively, a Stokes-Lagrange structure) is a Dirac structure (respectively, a Lagrange structure) of infinite dimension that makes use of the Stokes' divergence theorem (\textit{i.e.}, integration by parts).
\end{definition}

In this work, we will use the definition presented in~\cite{bendimerad2025stokeslagrange}. Let us consider a domain $\Omega \subset \mathbb{R}^n$ and define a storage port, a power control port, a resistive port~\cite{van2014port} and an energy control port~\cite{krhac2024port}. To define the storage port, let us consider a state space $\mathcal{X} = L^2(\Omega,\mathbb{R}^{n_\alpha})$, a flow space $\mathcal{F}_s:=\mathcal{X}$, an effort space $\mathcal{E}_s=\mathcal{X}' \simeq L^2(\Omega,\mathbb{R}^{n_\alpha})$, and a latent state space $\mathcal{Z} = L^2(\Omega,\mathbb{R}^{n_\alpha})$.
Then, regarding the power control port, let us define the observation space as a Hilbert space $\mathcal{F}_u$ and the control space as its dual\footnote{Duality is meant topologically.}  $\mathcal{F}_u := \mathcal{E}_u'$. 
Moreover, let us define the spaces of the resistive port as a Hilbert space $\mathcal{F}_R$ for the flows and its dual $\mathcal{E}_R:= \mathcal{F}_R'$ for the corresponding efforts.
Finally, to define the energy control port, let us consider the control space as a Hilbert space $\mathcal{U}$ and the observation space as its dual $\mathcal{Y}=\mathcal{U}'.$

Now, in order to define the Stokes-Dirac structure, we will consider three operators: $J:D(J)\subset \mathcal{E}_s \times \mathcal{E}_R \rightarrow \mathcal{E}_s \times \mathcal{E}_R$, along with observation/control operators $K:D(J) \rightarrow \mathcal{F}_u$ and $G: D(J) \rightarrow \mathcal{E}_u$ as defined in~\cite{bendimerad2025stokeslagrange}, \textit{i.e.}, such that \eqref{assumption:stokes-dirac-skew-symmetry} holds and $\begin{bmatrix} K \\ G \end{bmatrix}$ is onto in $\mathcal{F}_u \times \mathcal{E}_u$. The property extensively used in the following and which should be emphasized here is the \emph{skew-symmetry}:
\begin{equation} \label{assumption:stokes-dirac-skew-symmetry}
    \forall z \in D(J), \quad \langle Jz,z\rangle_{\mathcal{E}_s\times\mathcal{E}_R} = \langle Gz,Kz \rangle_{\mathcal{E}_u,\mathcal{F}_u}.
\end{equation}
This latter identity, an Abstract Stokes' divergence theorem, is indeed a skew-symmetry property if one considers it on the subspace $\mathcal{W}_0 := \ker K \cap \ker G$. Furthermore, if $J_0 := J|_{\mathcal{W}_0}$ satisfies $J_0^* = -J$ (the star denotes the adjoint operator), then $\mathcal{D} := \begin{bmatrix} J \\ K \\ \Id \\ G \end{bmatrix} \mathcal D(J)$ is a Stokes-Dirac structure, see~\cite{kurula2010dirac}.

Similarly for the Stokes-Lagrange structure, let us consider four linear operators: a possibly unbounded, closed, and densely defined operator $\, P:D(P)\subset \mathcal{Z} \rightarrow \mathcal{X}$, together with a bounded operator $\, S:  \mathcal{Z} \rightarrow \mathcal{E}_s$, and two control and observation operators $\beta: D(P) \rightarrow \mathcal{U}$ and $\gamma: D(P) \rightarrow \mathcal{Y}.$ The property that will be used in the following is the \emph{symmetry}:
\begin{equation} \label{assumption:stokes-lagrange-symmetry}
    \forall z_1,z_2 \in D(P), \quad \langle \gamma z_1, \beta z_2 \rangle_{\mathcal{Y},\mathcal{U}} +  \langle Pz_1,Sz_2\rangle_{\mathcal{X},\mathcal{E}_s} = \langle Pz_2, Sz_1\rangle_{\mathcal{X,\mathcal{E}_s}} + \langle \gamma z_2, \beta z_1 \rangle_{\mathcal{Y},\mathcal{U}}.
\end{equation}
Similarly to the finite dimensional case, the operators $P$ and $S$ must satisfy some maximality condition which, given a suitable subspace $\mathcal{Z}_0 \subset D(P)$, reads similar to~\eqref{eqn:ker-ran-equality-discrete}:
\begin{equation}
    \ker \begin{bmatrix}
        S_{|\mathcal{Z}_0}^* & -P_{|\mathcal{Z}_0}^*
    \end{bmatrix} \subset \textup{Ran} \begin{bmatrix}
        P \\S
    \end{bmatrix},
\end{equation}
where the star denotes the adjoint operator. \\
The reader may refer to~\cite{gernandt2025extension} for a study of the existence of such a subspace.

\begin{remark}
    In this work, we assume that $S$ is a bounded operator and allow $P$ being unbounded, as it will always be the case for the examples presented in Section~\ref{sec:continuous-models}. One may refer to~\cite{bendimerad2024implicit} for 1D distributed examples where $S$ is unbounded, and to~\cite{brugnoli2024discrete} for a discussion on the discrete equivalence of such representations.
\end{remark}

We can now state the definition of a distributed pH system.
\begin{definition} \label{def:dpHs}
Given Stokes-Dirac structure operators $J, K$, and $G$, Stokes-Lagrange operators $P, S, \gamma$, and $\beta$, as defined previously, let us define the latent state $\lambda:\mathbb{R}\times \Omega \rightarrow \mathbb{R}^{n_\alpha}$, co-state $e:\mathbb{R}\times \Omega \rightarrow \mathbb{R}^{n_\alpha}$ and resistive flow and effort $f_R:\mathbb{R}\times \Omega \rightarrow \mathbb{R}^{n_R}, e_R:\mathbb{R}\times \Omega \rightarrow \mathbb{R}^{n_R}$; the corresponding linear distributed pH system reads:
\begin{subequations} \label{eqn:continuous-linear-phs}
    \begin{align}
        \begin{pmatrix}
            P \, \partial_t \lambda \\ f_R
        \end{pmatrix} &= J \begin{pmatrix}
            e \\e_R
        \end{pmatrix}, \\
        e &= S \lambda, \\
        e_R &= R \, f_R, \\
        \begin{pmatrix}
            Ge \\ Ke
        \end{pmatrix} = \begin{pmatrix}
            u_D \\ y_D
        \end{pmatrix},& \quad \begin{pmatrix}
            \gamma \,\lambda\\ \beta \,\lambda
        \end{pmatrix} = \begin{pmatrix}
            u_L \\ y_L
        \end{pmatrix},
    \end{align}
\end{subequations}
with $(u_D, y_D)$ the external power port variables, and $(u_L ,y_L)$ the external energy port variables.
\end{definition}

\begin{remark}
    In previous work \cite{maschke2023linear}, the energy port variables are denoted by $(\chi_\partial,\varepsilon_\partial)$, here the choice is made to denote them by $(u_L,y_L)$ in order to emphasize on the control and interconnection purposes of these variables, related to the Lagrange structure.
    
    Furthermore, we point out that Definition~\ref{def:dpHs} is indeed an infinite-dimensional counterpart to Definition~\ref{def:linear-ph-state-repr}, but making use of the latent state variable to avoid the adjoint of $P$ and $S$, and with control and observation boundary terms expressed separately from $J$ (for $u_D$ and $y_D$) and $P$ and $S$ (for $u_L$ and $y_L$), as often in boundary control systems theory.
\end{remark}

The following lemma gives us the Hamiltonian and power balance:
\begin{lemma}
The following functional is a Hamiltonian of the system \eqref{eqn:continuous-linear-phs}:
\begin{equation} \label{eqn:lin-phs-continuous-hamiltonian}
    \mathcal H(\lambda) := \frac{1}{2} \langle P\lambda, S\lambda\rangle_{\mathcal{X},\mathcal{E}_s} + \frac{1}{2} \langle \gamma \lambda, \beta \lambda\rangle_{\mathcal{Y},\mathcal{U}}.
\end{equation}
The associated power balance reads:
$$
    \frac{\d}{\d t}\mathcal H =  \langle u_D,y_D\rangle_{\mathcal{E}_u,\mathcal{F}_u} + \langle \partial_t u_L, y_L\rangle_{\mathcal{Y},\mathcal{U}} - \langle f_R, Rf_R\rangle_{\mathcal{F}_R,\mathcal{E}_R}.
$$

\end{lemma}
\begin{proof}
    The proof can be found in~\cite[Subsection 2.5]{bendimerad2023implicit}.
\end{proof}

\begin{remark}
    It is important to notice that~\eqref{eqn:lin-phs-continuous-hamiltonian} is not the unique Hamiltonian that can be defined for system~\eqref{eqn:continuous-linear-phs}. Even if one imposes the Hamiltonian to be a quadratic functional, it remains non unique. Indeed, a Legendre transform on the boundary term $\gamma \lambda$ provides the alternative definition:
    $$
        \widetilde H(\lambda) := \frac{1}{2} \langle P \lambda, S \lambda\rangle_{\mathcal{X},\mathcal{E}_s} - \frac{1}{2} \langle \gamma \lambda, \beta \lambda\rangle_{\mathcal{Y},\mathcal{U}},
    $$
    for which the power balance becomes:
    $$
    \frac{\d}{\d t}\widetilde{\mathcal H} =  \langle u_D,y_D\rangle_{\mathcal{E}_u,\mathcal{F}_u} - \langle u_L, \partial_t y_L\rangle_{\mathcal{Y},\mathcal{U}} - \langle f_R, Rf_R\rangle_{\mathcal{F}_R,\mathcal{E}_R}.
    $$
    From a control perspective, choosing $\widetilde{\mathcal H}$ or $\mathcal H$ gives us access to different input and output variables.
\end{remark}

The goal of this work is to model physical continuous examples, the Lagrange structure of which contains a non-trivial $P$ operator (e.g., \emph{implicit} constitutive relations) and apply a \emph{structure-preserving} Finite Element method in order to discretize an infinite-dimensional pH system~\eqref{eqn:continuous-linear-phs} into a finite-dimensional pH system of the form~\eqref{eqn:discr-linear-phs-latent}. In other words, we aim at discretizing implicit distributed pH systems, while preserving both Dirac and Lagrange structures.

\section{Continuous models} \label{sec:continuous-models}

\subsection{Nanorod equation}
\label{sec:nanorod}
Given a domain $\Omega = (a,b)$, let us consider the following 1D equation that describes the longitudinal motion of a nanorod \cite{maschke2023linear}:
\begin{equation} \label{eqn:nanorod-equation-no-const}
    \partial_t\begin{pmatrix}
        \varepsilon \\
        p
    \end{pmatrix} = \begin{bmatrix}
        0 & \partial_x \\
        \partial_x & 0
    \end{bmatrix} \begin{pmatrix}
        \sigma \\ v
    \end{pmatrix},
\end{equation}
together with constitutive relations defined later, with $\varepsilon$ the strain, $p$ the  linear momentum, $\sigma$ the  stress, and $v$ the velocity. Now, we must define the constitutive relations that connect the state $(\varepsilon,p)$ and co-state $(\sigma,v)$ variables. In the linear local case, these are usually given by $v = p/\rho$ and Hooke's law $\sigma = E\varepsilon$,
where $E$ is the Young's modulus and $\rho$ the mass density. Then, Eq.~\eqref{eqn:nanorod-equation-no-const} becomes the classical wave equation. However, due to the scale of the studied system, these constitutive relations do not describe accurately the dynamics. To solve this issue, nonlocal effects between stress and strain variables have to be taken into account~\cite{eringen1983differential} in the constitutive relations.

\subsubsection{Nonlocal constitutive relations}
The general formulation of a linear nonlocal constitutive relations reads:
\begin{equation} \label{eqn:nonlocal-constitutive-general}
    \sigma(x) = \int_a^b G(x,x') \, \varepsilon(x') \,  \d x',
\end{equation}
with, for all $x,x'\in\mathbb{R}, \quad G(x,x') = G(x',x), \, G(x,x)\geq0$ the kernel operator. Different operators $G$ have been proposed to model this nonlocality, see in particular~\cite{eringen1983differential}. 
\begin{remark}
Such a relation could be used in a finite element software, however, as observed in data assimilation~\cite{mirouze2010representation}, such an explicit operator yields \emph{dense} stiffness matrices, which largely impedes numerical computations. Choosing a specific kernel that is the Green function of a differential operator allows for an implicit formulation that greatly reduces the computational burden~\cite{guillet2019modelling}; thanks to the sparsity of the corresponding matrices, and despite of the additional linear problem to solve at each time step.
\end{remark}
In the current 1D case, we consider the relation:
\begin{equation} \label{eqn:non-local-explicit}
    \sigma(x) = \int_a^b \frac{1}{2\ell}\exp\left({-\frac{|x - x'|}{\ell}}\right) E \varepsilon(x') \, \d x' = (h*E\varepsilon)(x),
\end{equation}
with $\ell$ the characteristic length of the nonlocality, and $G(x,x') := h(x-x') := \frac{1}{2\ell} \exp\left({-\frac{|x-x'|}{\ell }}\right)$. Then, $h$ is the Green function of~\cite{naylor1982linear}
$ \Id -  \ell^2\partial^2_{x^2}$. More precisely, following~\cite{romano2017nonlocal}, it holds:
\begin{theorem} \label{thm:nonlocal-explicit-to-implicit}
    For all $\varepsilon \in L^2(\Omega)$ and $ \ell,E \in \mathbb{R}$, such that $\ell>,E >0$, there exists a unique $\sigma \in H^2(\Omega)$ such that:
    \begin{equation}\label{eqn:non-local-implicit}
    (\Id- \ell^2 \, \partial^2_{x^2}) \sigma = E \varepsilon,
\end{equation}
and,
\begin{equation}\label{eqn:compatible-boundary-conditions}
    \begin{aligned}
        \sigma(a) - \ell \, \partial_x \sigma(a) = 0, \\
        \sigma(b) + \ell \, \partial_x \sigma(b) = 0.
    \end{aligned}
\end{equation}
Moreover, the solution $\sigma$ satisfies \eqref{eqn:non-local-explicit}.
\end{theorem}
\begin{remark} 
    As observed in \cite{romano2017nonlocal}, the Robin boundary conditions \eqref{eqn:compatible-boundary-conditions} are prescribed by the integral constitutive relations \eqref{eqn:non-local-explicit}. Hence, in addition to its computational efficiency, \eqref{eqn:non-local-implicit} also allows for assigning different boundary conditions.
\end{remark}
\subsubsection{Implicit nanorod equation}
Gathering the previous results leads to the following nanorod equation:
\begin{subequations} \label{eqn:nanorod-equation-implicit}
    \begin{empheq}[left=\empheqlbrace]{align}
         &\partial_t\begin{pmatrix}
        \varepsilon \\
        p
    \end{pmatrix} = \begin{bmatrix}
        0 & \partial_x \\
        \partial_x & 0
    \end{bmatrix} \begin{pmatrix}
        \sigma \\ v
    \end{pmatrix}, \label{eqn:nanorod-equation-implicit-dirac}\\
    &\begin{bmatrix}
         (\Id - \ell^2 \, \partial_{x^2}^2)  & 0\\
        0  & \Id
    \end{bmatrix}\begin{pmatrix}
        \sigma \\
        v
    \end{pmatrix} = \begin{bmatrix}
        E & 0\\
        0 & \frac{1}{\rho}
    \end{bmatrix} \begin{pmatrix}
        \varepsilon \\ p
    \end{pmatrix}, \label{eqn:nanorod-equation-implicit-lagrange}\\
    &\text{with the Robin boundary conditions: \eqref{eqn:compatible-boundary-conditions}.} \notag
    \end{empheq}
\end{subequations}
Notably, \eqref{eqn:nanorod-equation-implicit-dirac} defines the Stokes-Dirac structure and \eqref{eqn:nanorod-equation-implicit-lagrange} defines the  Stokes-Lagrange structure of the system, see, e.g.,~\cite{maschke2023linear}. Moreover, denoting the Lagrange structure operators by $P := \begin{bmatrix}
        (\Id - \ell^2 \, \partial_{x^2}^2) & 0\\
        0  & \Id
    \end{bmatrix} $ and $ S := \begin{bmatrix}
        E & 0\\
        0 & \frac{1}{\rho}
    \end{bmatrix},$ we have formally on $C^\infty_0([a,b],\mathbb{R})$, $P^*S=S^*P.$
Let us now follow~\cite{maschke2023linear} in order to find a Hamiltonian for~\eqref{eqn:nanorod-equation-implicit}.
Let us define the latent state variable as $ \lambda :=  \frac{1}{E}\sigma$, one can rewrite \eqref{eqn:nanorod-equation-implicit} as:
\begin{subequations} \label{eqn:nanorod-equation-implicit-latent}
     \begin{empheq}[left=\empheqlbrace]{align}
         &\partial_t\begin{bmatrix}
        (\Id - \ell^2 \, \partial_{x^2}^2) & 0\\
        0  &  \Id
    \end{bmatrix}\begin{pmatrix}
        \lambda \\
        p
    \end{pmatrix} = \begin{bmatrix}
        0 & \partial_x \\
        \partial_x & 0
    \end{bmatrix} \begin{pmatrix}
        \sigma \\ v
    \end{pmatrix}, \label{eqn:nanorod-equation-implicit-latent:DIRAC}\\
    &\begin{pmatrix}
        \sigma \\
        v
    \end{pmatrix} = \begin{bmatrix}
        E & 0\\
        0 & \frac{1}{\rho}
    \end{bmatrix} \begin{pmatrix}
        \lambda \\ p
    \end{pmatrix}, \label{eqn:nanorod-equation-implicit-latent:LAGRANGE} \\
    &\text{with the Robin boundary conditions:  \eqref{eqn:compatible-boundary-conditions}.}
    \end{empheq}
\end{subequations}
Both \eqref{eqn:nanorod-equation-implicit} and \eqref{eqn:nanorod-equation-implicit-latent} are a representation of the nonlocal nanorod equation. Nevertheless, the latter enables the computation of an explicit Hamiltonian that can be evaluated at each timestep, as shown in the following lemma.
\begin{lemma} \label{lemma:nanorod-hamiltonian-definition}
A possible Hamiltonian for the nonlocal nanorod equation \eqref{eqn:nanorod-equation-implicit-latent} which satisfies a power balance, given arbitrary boundary conditions for $\sigma$,  reads:
    \begin{equation}\label{eqn:hamiltonian-nonlocal-nanorod}
        \mathcal H(\lambda,p) := \frac{1}{2}\int_a^b E\, (\lambda^2 + \ell^2 \, (\partial_x \lambda)^2) + \frac{1}{\rho}p^2 \, \,  \d x .
    \end{equation}
    The power balance of~\eqref{eqn:nanorod-equation-implicit-latent} then reads:
    \begin{equation}
        \frac{\d}{\d t} \mathcal H(\lambda, p) = [ \sigma \, v ]_a^b + [ \partial_t \lambda \, \,   E \ell^2 \,  \partial_x \lambda ]_a^b\,,
    \end{equation}
    and with the Robin boundary conditions it becomes:
    \begin{equation} \label{eqn:nanorod-power-balance-with-robin-boundary-conditions}
        \frac{\d}{\d t} \mathcal H(\lambda, p) = -\ell \left[ \frac{1}{\rho} \left(\partial_t p(a)\, p(a) +   \partial_t p(b) \, p(b) \right) + E\, \left( \partial_t \lambda(a)\lambda(a) + \partial_t \lambda(b) \lambda(b)\right) \right].
    \end{equation}
\end{lemma}
\begin{proof}
    The proof is to be found in~\ref{proof:nanorod-hamiltonian-definition}
\end{proof}

In particular, the system may be considered {\em conservative} if one defines the appropriate Hamiltonian:
\begin{lemma} \label{lemma:nanorod-hamiltonian-robin-definition}
The conservative Hamiltonian of the nonlocal nanorod equation reads:
$$ \mathcal H_{\text{Rob}}(\lambda,p) := \mathcal H(\lambda,p) + \frac{\ell}{2} \left( \frac{1}{\rho} (p(b)^2 + p(a)^2) + E(\lambda(b)^2 + \lambda(a)^2)  \right).$$
\end{lemma}
\begin{proof}
    The proof is to be found in~\ref{proof:nanorod-hamiltonian-robin-definition}.
\end{proof}

\begin{remark}
    At first glance, one could think that the boundary condition~\eqref{eqn:compatible-boundary-conditions} is only taken into account in the constitutive relation. However, this result shows that it also affects the power control port of the Dirac structure.
    
    Moreover, one could have predicted the Hamiltonian $H$ to be modified into $H_{\text{Rob}}$ by adding an additional term that depends on $\lambda$ only, here a somewhat unexpected term that depends on $p$ is proven to be present. Such results highlight the fact that~\eqref{eqn:compatible-boundary-conditions} affects both  $\lambda$ and $p$ variables and ``adds'' energy at the boundary.
\end{remark}

Lastly, one can write the \emph{co-state formulation} of \eqref{eqn:nanorod-equation-implicit-latent} in order to remove the algebraic constraint due to the constitutive relations. Replacing $\lambda$ and $p$ by $\sigma$ and $v$ yields:
\begin{empheq}[]{align} 
     \label{eqn:nanorod-equation-implicit-coenergy}
         &\partial_t\begin{bmatrix}
        \frac{1}{E}(\Id - \ell^2 \, \partial_{x^2}^2) & 0\\
        0  & \rho
    \end{bmatrix}\begin{pmatrix}
        \sigma \\
        v
    \end{pmatrix} = \begin{bmatrix}
        0 & \partial_x\\
        \partial_x & 0
    \end{bmatrix} \begin{pmatrix}
        \sigma \\ v
    \end{pmatrix},
\end{empheq}
with the boundary conditions  \eqref{eqn:compatible-boundary-conditions}. Such a formulation reduces the number of variables required at the discrete level, accelerating the numerical simulations while lowering the memory usage of the system.

\subsection{Implicit Euler-Bernoulli beam}
\label{sec:EBbeam}

We focus now on the vibrations of a thin beam. However, it is well-known that the Euler-Bernoulli equation for this model is neither nonlocal nor implicit, see e.g.~\cite{ducceschi2019conservative,brugnoli2019portII}. Nevertheless, we recall that this is a matter of simplification in the limit of low wavenumbers, and that the \emph{shear} model obtained by neglecting the inertia of cross section in the Timoshenko model is indeed implicit, as mentioned in~\cite[Eq.~(13)]{ducceschi2019conservative}. It can also be derived in a more direct manner, as done in~\ref{apx:EB}.

Without external force, the equation of motion for the Euler-Bernoulli beam~\eqref{eq:EB-derivation} reads:
$$
\rho h \left(1- \frac{h^2}{12} \frac{\partial^2}{\partial x^2}\right) \frac{\partial^2}{\partial t^2} w + D \frac{\partial^4}{\partial x^4} w = 0.
$$
Let us recast it as a pH system.
\begin{lemma}
    Denoting by $\varepsilon = \frac{\partial^2}{\partial x^2}w$, the second spatial derivative of the displacement, $ p = \rho h (1- \frac{h^2}{12}) \frac{\partial}{\partial t} w$, the linear momentum, $v = \partial_t w$ the velocity, and $\sigma = D \varepsilon$ the stress, the implicit Euler-Bernoulli beam model written as a pH system reads:
\begin{subequations}
    \begin{align}
        \frac{\partial}{\partial t} \begin{pmatrix}
        \varepsilon \\p
    \end{pmatrix} &= \begin{bmatrix}
        0 & \partial_{x^2}^2 \\ - \partial_{x^2}^2 & 0 
    \end{bmatrix} \begin{pmatrix}
        \sigma \\v
    \end{pmatrix},\\
    \begin{bmatrix}
        \Id & 0 \\
        0 & \rho h (1-\frac{h^2}{12}\partial_{x^2}^2)
    \end{bmatrix} \begin{pmatrix}
        \sigma \\ v
    \end{pmatrix}& = \begin{bmatrix}
        D & 0 \\
        0 & \Id
    \end{bmatrix} \begin{pmatrix}
        \varepsilon \\ p
    \end{pmatrix}.
    \end{align}
\end{subequations}
Moreover, the co-energy formulation is:
\begin{equation}\label{eq:beam-co-energy}
    \begin{bmatrix}
        D^{-1} & 0 \\
        0 & \rho h (1 - \frac{h^2}{12} \partial_{x^2}^2)
    \end{bmatrix} \begin{pmatrix}
        \partial_t \sigma \\ \partial_t v
    \end{pmatrix} = \begin{bmatrix}
        0 & \partial_{x^2}^2 \\
        - \partial_{x^2}^2 & 0
    \end{bmatrix} \begin{pmatrix}
        \sigma \\ v
    \end{pmatrix}.
\end{equation}
\end{lemma}
\begin{proof}
    Once the choice of energy variable has been made, the port-Hamiltonian formulation of the system is simply a reformulation of the original system.
\end{proof}
Finally, let us write the Hamiltonian along with the power balance and boundary ports:
\begin{lemma}
    A Hamiltonian for the Euler-Bernoulli beam is made of  the sum of the kinetic and potential energy $\mathcal{K}$ and $\mathcal{P}$ written as functional of the state variables:\begin{equation}
    \label{eq:beam-Hamiltonian-def}
    \mathcal H(\sigma,v) = \frac{1}{2}\int_a^b \rho h\left(  v^2 + \frac{h^2}{12}\left( \frac{\partial}{ \partial x } v\right)^2 \right) + D^{-1} \,  \sigma^2 \, \d x,
    \end{equation}
    moreover, the power balance of the system is:
    \begin{equation}
    \label{eq:beam-Hamiltonian}
    \frac{\d}{\d t} \mathcal H(\sigma,v) = \left[ \rho h \, v \, \, \frac{h^2}{12} \, \partial_t \partial_x  v \right]_a^b\, + [\sigma \, \partial_x v]_a^b - [v \, \partial_x \sigma]_a^b.
    \end{equation}
\end{lemma}
\begin{proof}
    A direct computation yields the desired result.
\end{proof}

From this lemma, one can identify power and energy ports, e.g., as follows:
\begin{equation}\label{eq:beam-observation}
\begin{array}{c}
(u_L,y_L) = \left( \begin{pmatrix}
    v(a) \\ v(b)
 \end{pmatrix}, 
 \frac{\rho h^3}{12} \, \begin{pmatrix}
    - \partial_x v(a) \\ \partial_x v(b)
\end{pmatrix}
 \right)
, \\
(u^1_D,y^1_D) := \left(\begin{pmatrix}
   \sigma(a) \\ \sigma(b) 
\end{pmatrix}, \begin{pmatrix}
     - \partial_x v(a) \\ \partial_x v(b) 
\end{pmatrix} \right).
(u^2_D,y^2_D) := \left(\begin{pmatrix}
   v(a) \\ v(b)
\end{pmatrix}, \begin{pmatrix}
     \partial_x \sigma(a) \\ - \partial_x \sigma(b)
\end{pmatrix} \right).
\end{array}
\end{equation}
which correspond to the classical simply supported boundary conditions (when controls are set to zero). Other types of boundary conditions could be considered, leading to some kind of Robin boundary conditions, as for the nanorod case, see~\cite[Table~II]{BilbaoJASA2016}.

\begin{remark}\label{rem:phase-velocity}
    This model was studied previously in, e.g.,~ \cite{BilbaoJASA2016,ducceschi2019conservative} (with a tension $T_0 \neq 0$). It was demonstrated that the behavior of the phase velocity and group velocity is closer to the Timoshenko beam than the classical Euler-Bernoulli beam model. Additionally, due to the non-trivial P operator, and considering a solution given as a monochromatic wave, $e^{i(kx-\omega t)}$, the phase velocity $\frac{\omega}{k}  = {k \sqrt{D}}/\sqrt{\rho h(1 + (h^2/12) \,k^2)} $ remains bounded at high frequency, which is not the case for the classical Euler-Bernoulli beam model. These properties advocate for the use of this model, especially for beams of moderate thickness.
\end{remark}

\subsection{Incompressible Navier-Stokes equation}
\label{sec:incompressibleNSE}
\subsubsection{Port-Hamiltonian representation}
Let us now study a nonlinear model: the incompressible Navier-Stokes equations written in port-Hamiltonian formulation with vorticity and stream function as state variables. Let us consider a 2D or 3D domain $\Omega$ and recall the homogeneous incompressible Navier-Stokes equations \cite[Chap.~1]{boyer2012mathematical}:
\begin{equation}\label{eqn:navier-stokes} \left\{
    \begin{aligned}
         \rho_0(\partial_t + \bm{u}\cdot \grad)\bm{u} & = - \grad(p) + \mu\bm{\Delta u}, \\
         \diver(\bm{u}) &= 0
    \end{aligned} \right.
\end{equation}
With $\bm{u}$ the fluid velocity\footnote{\blue{Vector fields are written in boldface.}}, $\rho$ the fluid density, $p$ the pressure, and $\mu$ the viscosity.

Using classical vector calculus identities, this can be recast as:
\begin{equation*}
    \rho_0\,\partial_t \bm u = - \grad( p + \frac{\rho_0}{2}|\bm u|^2) - \rho_0 \,\curl(\bm u) \times \bm u + \mu \bm{\Delta} \bm u.
\end{equation*}
Now, taking the $\curl$ of the previous equation and defining the vorticity $\bm \omega := \curl(\bm u)$ yields the equation of vorticity:
$
    \rho_0 \, \partial_t \bm \omega = - \rho_0 \,\curl( \bm\omega \times \bm u) + \mu \bm{\Delta} \bm \omega.
$
Let us now restrict ourselves to a simply connected 2D domain, with $\bm u = \begin{pmatrix}
    \bm u_x & \bm u_y 
\end{pmatrix}^\top$ and $\omega = \bm \omega_z $ and define the reduced~\cite[Chap. 9]{assous2018mathematical} operators $\grad^\perp = \begin{bmatrix}
	\partial_y \\ -\partial_x
\end{bmatrix}$ and $\curldeuxD = \begin{bmatrix}
	- \partial_y & \partial_x 
\end{bmatrix}$. 
\begin{remark} \label{rmk:reduced-complex}
	In 3D, the operators $\grad$, $\curl$ and $\diver$ define a de Rham complex over $k$-forms as detailed in \cite{arnold2018finite}. In 2D, this leads us to a reduced complex, see Figure \ref{fig:reduced-complex}.
	\begin{figure}[h!]
		\centering
		\begin{tikzpicture}
			\node (A)  at (-4,0) {$\Lambda^0(\Omega)$};
			\node (B) at (0,0) {$\Lambda^1(\Omega)$};
			\node (C) at (4,0) {$\Lambda^2(\Omega)$};
			
			\draw[,->] (-3,0.2)  -- (-1,0.2);
			\draw[,<-] (-3,-0.2) -- (-1,-0.2);
			\draw[,->] (1,0.2) -- (3,0.2);
			\draw[,<-] (1,-0.2) -- (3,-0.2);
			
			\draw (-2,0.2) node[above] {$\grad$};
			\draw (2,0.2) node[above] {$\curldeuxD$};
			
			\draw (-2,-0.2) node[below] {$-\diver$};
			\draw (2,-0.2) node[below] {$\grad^\perp$};
			
		\end{tikzpicture}
        \caption{Reduced complex of the Navier-Stokes equation on a 2D domain, with $\Lambda^k(\Omega)$ denoting the space of $k$-forms over $\Omega$. } \label{fig:reduced-complex}
	\end{figure}
	
	In particular, $\curldeuxD  \grad = 0$, $\diver  \grad^\perp = 0$, moreover, $-\diver$ is the formal adjoint of $\grad$ and $\grad^\perp$ is the formal adjoint of $\curldeuxD$.
\end{remark}

Remark \ref{rmk:reduced-complex} along with the incompressibility assumption imply that there exists a \emph{stream function} $\psi$ such that $\bm{u} = \grad^\perp \psi$. Moreover, let us define the vorticity as: $\omega := \curldeuxD(\bm{u}).$ In particular, the constitutive relations linking $\psi$ and $\omega$ reads:
\begin{equation} \label{eqn:constitutive-relation-psi-omega}
    - \Delta \psi = \omega,
\end{equation}
the proof is to be found in~\ref{apx:vector-identities}. The equation for the vorticity becomes:
\begin{equation*}
    \rho_0 \, \partial_t \omega = - \rho_0 \,\curldeuxD( G(\omega) \, \grad^\perp \psi) + \mu \Delta \omega,
\end{equation*}
with $G(w) = \begin{bmatrix}
    0 & - \omega \\ \omega & 0
\end{bmatrix}.$
As explained in e.g~\cite{de2019inclusion,palha2017mass,zhangmass2022,zhangmeevc2024}, in addition to the kinetic energy, the \emph{enstrophy} is also a quantity of interest in the 2D setting. Using $\psi$ and $\omega$ as state variables, they are defined as:
\begin{equation}
    \mathcal{K}(\psi) := \frac{\rho_0}{2} \int_\Omega || \grad^\perp \psi||^2 \, \d x, \quad \mathcal{E}(\omega) := \frac{\rho_0}{2} \int_\Omega \omega^2 \, \d x.  
\end{equation}
Following \cite{haine2021incompressible}, let us define $J_\omega := \curl_{2D}( G(\omega) \grad^\perp(\cdot)) = \diver( \, \omega \,\grad^\perp(\cdot))$, and choose the stream function $\psi$ and vorticity $\omega$ as new state variables.  Using the kinetic energy as the (implicit) Hamiltonian, the dynamics of $\omega$ yields:
\begin{equation} \label{eqn:vorticity-dynamics-kinetic-hamiltonian}
    \rho_0\,\partial_t \omega = - \rho_0\,\diver(\omega \, \grad^\perp( e_\mathcal{K})) + \mu \Delta \omega, \quad \text{with } \quad -\Delta e_\mathcal{K} = \omega,
\end{equation}  
whereas using the enstrophy as the (explicit) Hamiltonian, the dynamics of $\omega$ reads:
\begin{equation}\label{eqn:vorticity-dynamics-enstrophy-hamiltonian}
    \rho_0\,\partial_t \omega = - \rho_0\,\diver( e_\mathcal{E} \, \grad^\perp( \psi)) + \mu \Delta \omega, \quad \text{with }\quad  e_\mathcal{E} = \omega.
\end{equation}
In  \eqref{eqn:vorticity-dynamics-kinetic-hamiltonian}, $\omega$ is the modulating variable and $\psi$ is the effort variable. In \eqref{eqn:vorticity-dynamics-enstrophy-hamiltonian}, $\psi$ is the modulating variable and $\omega$ is the effort variable.
Writing the dynamics of the stream function $\psi$ using the co-energy formulation of \eqref{eqn:vorticity-dynamics-kinetic-hamiltonian}, and writing the dynamics of $\omega$ using \eqref{eqn:vorticity-dynamics-enstrophy-hamiltonian} gives us the following coupled pH system:
\begin{lemma}
    The incompressible Navier--Stokes equations on a 2D domain written as a pH system in vorticity~--~stream function formulation reads:
\begin{equation} \label{eqn:navier-stokes-vorticite-courant}
    \underbrace{\begin{bmatrix}
        - \rho_0\Delta & 0\\
        0 & \rho_0
    \end{bmatrix}}_{P} \begin{pmatrix}
        \partial_t \psi \\
        \partial_t \omega
    \end{pmatrix} = \left[\underbrace{\rho_0 \, \begin{bmatrix}
        -J_\omega & 0 \\
        0 & -\diver(\grad^\perp(\psi) \, \cdot )
    \end{bmatrix}}_{=:J(\psi,\omega)} - \underbrace{\mu \begin{bmatrix}
        \Delta^2 & 0 \\ 0 & - \Delta
    \end{bmatrix}}_{R} \right] \begin{pmatrix}
        \psi \\ \omega
    \end{pmatrix}\,,
\end{equation}
with the Hamiltonian as the sum of the kinetic energy and enstrophy $\mathcal{H}(\psi,\omega) = \mathcal{K}(\psi) + \mathcal{E}(\omega).$
\end{lemma}
\begin{proof}
    First, 
    writing~\eqref{eqn:vorticity-dynamics-kinetic-hamiltonian} using $\psi$ instead of $\omega$  as the state variable gives us a co-energy formulation:
    $$-\rho_0\,\Delta\partial_t \psi = - \rho_0\,\diver(\omega \, \grad^\perp( e_\mathcal{K})) - \mu \Delta^2 \psi, \quad \text{with} \quad e_\mathcal{K}=\psi.$$
    Then, recalling that $J_\omega = \curl_{2D}( G(\omega) \grad^\perp(\cdot)) = \diver( \, \omega \,\grad^\perp(\cdot))$, gives us the first line of~\eqref{eqn:navier-stokes-vorticite-courant}. 
    
    Most notably, $\psi$ is the energy variable, the kinetic energy $\mathcal{K}$ is the Hamiltonian, and $\omega$ modulates the Dirac structure.

    Then, the dynamics of the vorticity $\omega$ is given directly by~\eqref{eqn:vorticity-dynamics-enstrophy-hamiltonian}, where the enstrophy $\mathcal{E}$  is the Hamiltonian, with $\omega$ as energy variable and $\psi$ modulates the Dirac structure.
\end{proof}

\begin{remark}
    Note that both \eqref{eqn:vorticity-dynamics-kinetic-hamiltonian} and \eqref{eqn:vorticity-dynamics-enstrophy-hamiltonian} are equivalent. However, choosing different Hamiltonian allows for different Dirac structure matrices to be obtained, and will allow us to prove the conservative properties of the discretized system in the following sections.
\end{remark}

\begin{remark}
    The Hamiltonian $\mathcal{H}(\psi,\omega)$ is separable,  and $J(\psi,\omega)$ is block diagonal. Hence, the coupling only occurs through the modulation of $J$ by each state variable, which is a quite original feature.
\end{remark}
Let us show that such a system is indeed pH. Defining $S = I_2$, we have that \eqref{eqn:navier-stokes-vorticite-courant} reads: $ P \partial_t \begin{pmatrix}
    \psi \\ \omega
\end{pmatrix} = (J(\psi,\omega) - R) \, \, S \begin{pmatrix}
    \psi \\ \omega
\end{pmatrix}.$
From the fact that $\Delta$ is formally symmetric, we deduce that $P^*S$ is formally symmetric. Let us now state that $J$ is formally skew symmetric.
\begin{lemma} \label{lemma:navier-stokes-skew-symmetry}Let $\omega \in H^1$.
    Let $\psi,\phi_1, \phi_2 \in H^2,$ then: 
    $$\int_\Omega \phi_1 J_\omega \phi_2 \, \d x = - \int_\Omega \phi_2 J_\omega \phi_1 \, \d x + \int_{\partial\Omega} \phi \, \omega \,  \bm{n}\cdot \grad^\perp(\psi) \, \d s + \int_{\partial \Omega} \psi \, \omega \,\bm{n} \cdot \grad^\perp(\psi) \, \d s .$$
    $$ \int_\Omega \phi_1 \diver(\grad^\perp(\psi) \phi_2) \, \d x = - \int_\Omega \phi_2 \diver(\grad^\perp(\psi) \phi_1) \, \d x. + \int_{\partial \Omega} \phi_1 \, \phi_2 \, \bm{n} \cdot \grad^\perp(\psi)  \, \d s $$
\end{lemma}
\begin{proof}
    Proof is to be found in \ref{proof:navier-stokes-skew-symmetry}
\end{proof}
 The Stokes-Dirac structure of this model is modulated by the state variables $\omega,\psi$, hence the nonlinearity lies in the Dirac structure, but not in the Lagrange structure. Such a property proves useful as it allows for energy-preserving time integration methods with linear steps, more precisely, fixing $J(\omega,\psi)$ between two time steps makes the system linear while keeping the operator skew-symmetric \cite{palha2017mass,roze2024passive}. The power balance reads:
\begin{theorem}\label{thm:navier-stokes-vorticite-courant-power-balance}
    Given two state variables $\omega(t,x),\psi(t,x)$, the power balance along a trajectory of~\eqref{eqn:navier-stokes-vorticite-courant} reads:
    \begin{equation}
    \begin{aligned}
        \frac{\d}{\d t} \mathcal{K}(\psi) = & - \mu \int_\Omega (\Delta \psi)^2 \, \d x  - \int_{\partial \Omega} \underbrace{\rho_0\,\omega \, \psi}_{=:y_D^1} \, \underbrace{ \bm{n} \cdot \grad^\perp(\psi)}_{=:u_D^1} \, \d s \,  - \int_{\partial \Omega} \underbrace{   \psi}_{=:y_D^2} \, \underbrace{  \mu \,\grad(\Delta\psi) \cdot \bm{n}}_{=:u_D^2} \, \d s  \\
        &  +  \int_{\partial \Omega} \underbrace{ \bm{n} \cdot \grad(\psi)}_{=:y_D^3}\,\underbrace{\mu \, \Delta \psi}_{=:u_D^3}  \, \d s
         + \int_{\partial \Omega} \underbrace{\rho_0\,\psi}_{=:y_L} \, \partial_t \underbrace{ \bm{n} \cdot \grad(\psi)}_{=:u_L} \, \d s,
    \end{aligned}
    \end{equation}
    \begin{equation}
        \frac{ \d}{ \d t} \mathcal{E}(\omega) = -  \mu \int_\Omega ||\grad (\omega)||^2 \, \d x - \int_{\partial \Omega} \underbrace{\rho_0\,\omega \, \psi  }_{=:y_D^4} \, \underbrace{ \bm{t} \cdot \grad(\omega)}_{=:u_D^4} \, \d s \, -  \int_{\partial \Omega} \underbrace{\omega}_{=:y_D^5} \, \underbrace{\mu \,   \bm{n} \cdot \grad(\omega)}_{:=u_D^5}\, \d s.
    \end{equation}
    In particular, with zero boundary conditions, and using the fact that $-\Delta \psi = \omega$, one gets:

    $$\frac{\d}{\d t} \mathcal{K}(\psi) = - 2 \, \mu\,  \mathcal{E}(\omega), \qquad \frac{\d}{\d t} \mathcal{
    E
    }(\omega)= - \mu \int_\Omega ||\grad (\omega)||^2 \, \d x.$$
\end{theorem}
\begin{proof}
    A direct computation yields the desired result.
\end{proof}
Notably, the observation $y_D^2$ is the Dirichlet trace of $\psi$ and $y^3_D$ is the Neumann trace of~$\psi$.

\subsubsection{Boundary conditions}

Boundary conditions for the vorticity are notoriously difficult to treat, as they do not appear naturally in the equation~\cite{lequeurre_vorticity_2020,Rempfer2006boundary} and adding such condition usually leads to nonlocal boundary conditions that require an additional system to be solved at each time step~\cite{weinan1996vorticity}. Alternative approaches have been discussed in~\cite{de2019inclusion}, where numerical results are presented to exhibit the properties of the various conditions discussed.
In this work, the impermeable no-slip boundary condition is considered on the whole boundary of the domain. As a consequence, it is required that:
 $\bm{n} \cdot \grad^\perp( \psi) _{|\partial\Omega} = 0, \quad  \bm{n} \cdot \grad(\psi)_{|\partial \Omega} = 0.$
 
Using the previously defined control and observations, this gives us:
$$ u_D^1=0, \quad u_D^4 = 0,  \quad y_D^3 = 0. $$
Contrarily to previous approaches, i.e., \cite{weinan1996vorticity,de2019inclusion}, here the no-slip condition $0 = \bm{n} \cdot \grad(\psi)_{|\partial \Omega} =y_D^3$ is kept as an algebraic constraint. Usually, one can differentiate this boundary condition with respect to time and replace it by the vorticity; yielding a nonlocal boundary condition on the vorticity. Here, the constraint is kept, and the system becomes a differential algebraic equation where $\omega_{|\partial \Omega}$ plays the role of a Lagrange multiplier. This offers two advantages; firstly, additional costly steps are avoided, such as computing the pressure on the whole domain. Secondly, boundary ports are available in order to compute the enstrophy balance at each timestep: showing where and when enstrophy is generated and allowing for the computation of the enstrophy balance error. 

Using $-\Delta \psi = \omega,$  additional constraints are derived:
$ u_D^3 =  - \mu\,y_D^5,\quad u_D^5 = - u_D^2.  $
Finally, gathering these conditions yields:
\begin{lemma}\label{lemma:no-slip-boundary-condition}
    Imposing the impermeable no-slip boundary condition to the incompressible Navier-Stokes equation \eqref{eqn:navier-stokes-vorticite-courant} yields the following set of constraints:
    \begin{equation} \label{eqn:no-slip-boundary-condition}
    \begin{bmatrix}
   0& 0 & 1 & 0 &0& 0 & 0  & 0\\
   0& 0 & 0 &  \mu &0& 0 & 1 & 0 \\
   0& 0 & 0 & 0 & 1 & 0 & 0 & 0  \\
   0& 0 & 0 & 0 &0& 1  & 0 & 1 \\
\end{bmatrix}
\begin{bmatrix}
    y_D^1 \\ y_D^2 \\ y_D^3 \\ y_D^5 \\ u_D^1 \\u_D^2 \\ u_D^3 \\ u_D^5
\end{bmatrix} = 0.
\end{equation}
\end{lemma}
This set of equations shows that in order to impose the no-slip condition $y_D^3=0$, the corresponding control $u_D^3$ becomes a Lagrange multiplier. Moreover, from $ u_D^3 = - \mu \,y_D^5 = - \mu \,\omega_{|\partial \Omega},$ we deduce that vorticity has to be nonzero at the boundary. In particular, we get that $\omega_{|\partial \Omega} = \frac{1}{\mu}u_D^3,$ hence the lower the viscosity $\mu$ (i.e., the higher the Reynolds number) the more vorticity is being generated at the boundary.

Note that $u_D^5= \mu \, \bm n \cdot \grad(\omega)_{|\partial \Omega}$ is a Lagrange multiplier as well, and corresponds to the vorticity generated at the boundary. Finally, given that $u_D^5 = - u_D^3,$ this generation of vorticity is also affecting the dynamics of the stream function, however since the collocated observation $y_D^3$ is set to zero, this vorticity generation does not appear in the power balance of the kinetic energy.
\subsubsection{Enstrophy creation and kinetic energy dissipation}
Using Lemma \ref{lemma:no-slip-boundary-condition}, we get the enstrophy balance and power balance:
\begin{lemma} \label{lemma:continuous-enstrophy-creation}
    The enstrophy balance with the no-slip boundary conditions reads:
    \begin{equation}
        \frac{ \d}{ \d t} \mathcal{E}(\omega) = -  \mu \int_\Omega ||\grad (\omega)||^2 \, \d x -  \frac{1}{\mu}\int_{\partial \Omega} u_D^5 u_D^3 \, \d s.
    \end{equation}
    Moreover, the kinetic energy balance reads:
    \begin{equation}
        \frac{\d}{\d t} \mathcal{K}(\psi) = - 2 \, \mu \, \mathcal{E}(\omega).
    \end{equation}
\end{lemma}

\begin{proof}
    Using the power balance provided in Theorem~\ref{thm:navier-stokes-vorticite-courant-power-balance} with the boundary conditions of Lemma~\ref{lemma:no-slip-boundary-condition} yields the desired result.
\end{proof}

Lemma~\ref{lemma:continuous-enstrophy-creation} proves that enstrophy is dissipated in the domain and is generated solely at the boundaries of the domain. Moreover, for some fixed Lagrange multipliers $u_D^3  ,u_D^5,$ the generation of enstrophy at the boundary increases as $\mu$ decreases.

\section{Spatial Discretization}
\label{sec:PFEM}

This section aims at presenting structure-preserving spatial discretizations of the three previous models. In particular, this procedure yields discrete Dirac and Lagrange structures \cite{gernandt2022equivalence}, corresponding to their infinite-dimensional Stokes-Dirac and Stokes-Lagrange counterparts.

Keeping structures intact ensures that the discrete system satisfies a power balance that mimics its continuous counterpart; hence ensuring properties such as passivity and energy preservation. In particular, neglecting these structures can lead to nonphysical behaviour and even make dissipative systems unstable, see \cite{egger2019structure} for an example.

\subsection{Nanorod equation}

The structure-preserving spatial discretization of the wave equation as a pH system is already well known (see e.g., \cite{haine2023numerical}). Here the additional difficulty is the implicit constitutive relation. In order to properly discretize \eqref{eqn:nanorod-equation-implicit-coenergy} into a finite-dimensional pH system, we will take care of preserving the symmetry and skew-symmetry of the Lagrange structure and Dirac structure matrices, respectively. Moreover, one should be particularly cautious of the Robin boundary condition \eqref{eqn:compatible-boundary-conditions} as it will be of importance both in the Dirac and Lagrange structure discretizations. In particular, we will first write the weak formulation with general boundary conditions in order to exhibit the power control port and energy control port at the discrete level, then apply these conditions. As a consequence, one first retrieves the energy and power control ports that correspond to the continuous case, then closes the system using provided boundary conditions.

\subsubsection{Weak formulation}

Let us consider a test function $\phi \in H^1([a,b])$, then let us multiply both lines of \eqref{eqn:nanorod-equation-implicit-coenergy} by $\phi$ and integrate over the domain:
\begin{equation*}
    \begin{cases}
        \partial_t\int_a^b \phi \, \frac{1}{E}(1 - \ell ^2 \partial_{x^2}^2) \,  \sigma \, \d x = \int_a^b \phi \, \partial_x v \, \d x, \\
        \partial_t\int_a^b \phi \, \rho \, v \, \d x = \int_a^b \phi \, \partial_x \sigma \, \d x.
    \end{cases}
\end{equation*}
The left-hand side of both equations corresponds to the Lagrange structure operator $P$, while the right-hand side corresponds to the Dirac structure operator $J$.
\paragraph{Symmetry of $P$} To exhibit the symmetry of $P$, we will apply an integration by parts on the left-hand side of the first line, additionally revealing the energy control port.
\paragraph{Skew-symmetry of $J$} To show that $J$ is skew-symmetric, we have to choose on which line the integration by parts is carried \cite{haine2023numerical}. Here, let us choose the boundary force as a control and integrate by parts on the second line.

Carrying out both integrations by parts yields:
\begin{equation} \label{eqn:nonlocal-nanorod-weak-formulation}
    \begin{cases}
        \partial_t\int_a^b  \, \frac{1}{E}( \, \phi \, \sigma +  \ell^2 \, \partial_x\phi \, \partial_{x}\sigma )\, \d x  \, \underbrace{ - [\phi \frac{\ell^2}{E} \, \partial_x \partial_t \sigma ]_a^b}_{\text{Energy control port}}= \int_a^b \phi \, \partial_x v \, \d x, \\
        \partial_t\int_a^b \phi \, \rho v \, \d x = - \int_a^b \partial_x\phi \, \sigma \, \d x \, \underbrace{+ \, [\phi \sigma]_a^b}_{\text{Power control port}}.
    \end{cases}
\end{equation}

\paragraph{Robin boundary conditions:}
Before discretizing the system, let us finally apply the Robin boundary conditions~\eqref{eqn:compatible-boundary-conditions}. Denoting by $n_a = -1, \, n_b=1$ the outer normals, we get for the integrated terms of~\eqref{eqn:nonlocal-nanorod-weak-formulation}:
$$
\forall s \in \{a,b\}, \quad \left\{ \begin{aligned}
     &\frac{\ell}{E} \partial_t\sigma(t,s) + \frac{\ell^2}{E} (\partial_x \partial_t\sigma(t,s))  n_s  = 0,  \\
     & \sigma(t,s)= - \ell ( \partial_x \sigma(t,s))  n_s = - \ell ( \rho \, \partial_t v(t,s) )  n_s,
\end{aligned}\right.
$$
where the latter equality is due to the equality $ \rho \partial_t v = \partial_x \sigma$. Using these two results gives us:
\begin{equation*}
    \begin{cases}
        \partial_t\int_a^b  \, \frac{1}{E}( \, \phi \, \sigma +  \ell^2 \, \partial_x\phi \, \partial_{x}\sigma )\, \d x  \, + \phi(a) \frac{\ell}{E} \,  \partial_t \sigma(t,a) +  \phi(b) \frac{\ell}{E} \,  \partial_t \sigma(t,b) = \int_a^b \phi \, \partial_x v \, \d x, \\
        \partial_t\int_a^b \phi \, \rho v \, \d x + \phi(a) \, \ell \rho v(t,a) + \phi(b) \, \ell \rho v(t,b) = - \int_a^b \partial_x\phi \, \sigma \, \d x.
    \end{cases}
\end{equation*}
Note that since the conditions are homogeneous, the controls do not appear in the equations anymore. For the sake of completeness and clarity, we will now carry the discretization without the Robin boundary conditions, then in a second stage, apply these conditions.
\subsubsection{Finite Element approximation}
Let us now discretize \eqref{eqn:nonlocal-nanorod-weak-formulation}. This system is composed of two state variables $\sigma$ and $v$. Due to the choice of causality, the spatial derivative of $v$ is evaluated strongly, hence has to be discretized in $H^1$. Moreover, due to the operator in front of $\sigma$, it has to be discretized in $H^1$ as well.
Let us choose $(\phi_i)_{i \in [1,N] }$ the set of $P^1$ Finite Element family over the points $ a=x_1 < x_2 < ... < x_N = b$ satisfying $\forall i,j \in [1,N], \, \phi_i(x_j) = \delta_{i}(j)$  and write the discretized versions of $\sigma$ and $v$ as follows:
$$ \sigma^d(t,x) = \sum_{i=1}^N \phi_i(x) \overline{\sigma}_i(t), \quad v^d(t,x) = \sum_{i=1}^N \phi_i(x) \overline{v}_i(t).$$
Moreover, we will denote the value of $\partial_x \sigma \cdot n$ at both boundaries by $u_L(t) \in \mathbb{R}^2$ and the value of $\sigma \cdot n$ at both boundaries by $u_D(t) \in \mathbb{R}^2.$ Then, for all $i \in [1,N]$, \eqref{eqn:nonlocal-nanorod-weak-formulation} reads:
\begin{equation*}
    \left\{\begin{aligned}
        \sum_{j=1}^N \frac{1}{E} \left[\int_a^b \phi_i\phi_j \, \d x +  \ell^2 \, \int_a^b \partial_x \phi_i \, \partial_x \phi_j \, \d x \right] \frac{\d}{\d t} \overline{\sigma}_j    =  & \sum_{j=1}^N \int_a^b \phi_i \, \partial_x \phi_j \, \d x \, \overline{v}_j \\
        &   + \delta_{1}(i) \frac{\ell^2}{E} \frac{\d}{\d t} (u_L)_1   + \delta_{N}(i) \frac{\ell^2}{E} \frac{\d}{\d t} (u_L)_2, \\
        \sum_{j=1}^N \int_a^b \phi_i\, \rho \, \phi_j \, \d x \, \frac{\d}{\d t} \overline{v}_j =& - \sum_{j=1}^N \int_a^b \partial_x \phi_i \, \phi_j \, \d x \, \overline{\sigma}_j \\ & + \delta_1(i) \, (u_D)_1  + \delta_N(i) \, (u_D)_2 . 
    \end{aligned}\right.
\end{equation*}

\paragraph{Matrices definition} Let us now define the matrices:
$$ \bm{M}_{ij} = \int_a^b \phi_i \phi_j \, \d x, \quad \bm{M}^\rho_{ij} = \int_a^b \rho \, \phi_i \phi_j \, \d x, \quad \bm{K}_{ij} = \int_a^b \partial_x \phi_i \, \partial_x \phi_j \, \d x,  $$
$$ \bm{D}_{ij} = \int_a^b \phi_i \, \partial_x \phi_j \, \d x, \quad \bm{B} = \begin{bmatrix}
    \phi_1(a) & \phi_1(b)\\
    \phi_2(a) & \phi_2(b)\\
     \vdots & \vdots\\
    \phi_N(a) & \phi_N(b)\\
\end{bmatrix} = \begin{bmatrix}
    1 & 0 \\
    0 & \vdots\\
    \vdots & 0\\
    0 & 1
\end{bmatrix}, \quad \bm{M}^\partial = \begin{bmatrix}
    1 & 0 \\ 0 & 1
\end{bmatrix}. $$
This allows us to write the discrete system as:
$$\left \{\begin{aligned}
    \frac{1}{E} (\bm{M} + \ell^2 \, \bm{K}) \frac{\d}{\d t} \overline{\sigma} &= \bm{D} \, \overline{v} + \frac{\ell^2}{E} \bm{B} \, \frac{\d}{\d t} u_L, \\
    \bm{M}^\rho \frac{\d}{\d t} \overline{v} &= - \bm{D}^\top \overline{\sigma} + \bm{B} \, u_D.
\end{aligned} \right.$$

\paragraph{Lagrange and Dirac structures} Let us now define two matrices that will allow us to bring to light the pH structure.

\begin{definition}
    The discrete Lagrange and Dirac structure operators read:
$$ \bm{P} := \begin{bmatrix}
    \frac{1}{E}(\bm{M}+\ell^2\bm{K}) & 0 & - \frac{\ell^2}{E} \bm{B}\\
    0 & \bm{M}^\rho & 0 \\
    - \frac{\ell^2}{E}\bm{B}^\top & 0 & 0
\end{bmatrix},\;  \bm{S} = \mathbb{M}: = \begin{bmatrix}
    \bm{M} & 0 & 0 \\
    0 & \bm{M} & 0\\
    0 & 0 & \bm{M}^\partial
\end{bmatrix},\; \bm{J} := \begin{bmatrix}
    0 & \bm{D} & \bm{B} \\ -\bm{D}^\top & 0 & 0\\
    - \bm{B}^\top & 0 & 0
\end{bmatrix}.$$
\end{definition}
The energy control port is linked to an energy observation port defined as: $\bm{M}^\partial y_L = - \frac{l^2}{E}\bm{B}^\top \overline{\sigma}$, and the power control port to a power observation port as: $\bm{M}^\partial y_D = - \bm{B}^\top \overline{v}.$ Moreover we get the following lemma immediately:
\begin{lemma} The discrete Lagrange and Dirac structure operators satisfy the following symmetry and skew-symmetry properties:
    $$\bm{P}^\top \mathbb{M}^{-1} \bm{S} = \bm{S}^\top \mathbb{M}^{-1} \bm{P}, \qquad \bm{J} = -\bm{J}^\top.$$
\end{lemma}

\begin{remark}
    In the setting of Section~\ref{sec:sec2.1}, the symmetry of $\bm P^* \bm S$ (where a star denotes the adjoint) is expressed with the usual Euclidean inner product on $\mathbb{R}^{n+n_L}$, hence $\bm P^* = \bm P^\top$, and checking the symmetry of the operator $\bm P^* \bm S$ amounts to checking the symmetry of the matrix $\bm P^\top \bm S$. However, when applying the Finite Element Method, an invertible mass matrix $\mathbb{M}$ appears on the left-hand side of~\eqref{eqn:discr-linear-phs-latent:LAGRANGE}. The symmetry condition is then expressed using the weighted inner product $\langle \cdot, \cdot\rangle_{\mathbb{M}^{-1}}$, and checking the symmetry of $\bm P^* \bm S$ is then achieved by checking the symmetry of the matrix $\bm P^\top \mathbb{M}^{-1} \bm S$.
\end{remark}

\begin{remark}\label{rem:P-S}
    The choice of representation made for this system yields a trivial operator $\bm{S} = \mathbb{M},$ hence the symmetry condition $\bm{P}^\top \mathbb{M}^{-1} \bm{S} = \bm{S}^\top \mathbb{M}^{-1} \bm{P}$  can be replaced by $\bm{P}^\top = \bm{P}$ similarly to the Dirac structure operator for which $\bm{J}^\top = - \bm{J}$.
\end{remark}

Let us now define the Hamiltonian
\begin{definition}
    The Hamiltonian of the discretized nonlocal nanorod equation reads:
$$ \mathcal H^d = \frac{1}{2} \begin{pmatrix}
    \overline{\sigma} \\ \overline{v} \\ u_L
\end{pmatrix}^\top \bm{P} \begin{pmatrix}
    \overline{\sigma} \\ \overline{v} \\ u_L
\end{pmatrix}.  $$
\end{definition}

\paragraph{Robin boundary conditions} Finally, let us add the Robin boundary conditions. 

\begin{lemma} \label{lemma:nanorod-discrete-robin}
    At the discrete level, the Robin boundary conditions read for the energy control port:
\begin{equation}\label{eqn:discr-robin-boundary-conditions-lagrange}
    \bm{B}^\top \overline{\sigma} + \ell \, \bm{M}^\partial u_L = 0,
\end{equation}
and for the power control port:
\begin{equation}\label{eqn:discr-robin-boundary-conditions-dirac}
    \bm{M}^\partial u_D + \ell \, \underbrace{\begin{bmatrix}
    \rho(a) & 0 \\ 0& \rho(b)
\end{bmatrix}}_{=:\bm{M}^\partial_\rho} \bm{B}^\top  \frac{\d}{\d t}\overline{v} = 0.
\end{equation}  

\end{lemma}
\begin{proof}
    The proof is to be found in~\ref{proof:nanorod-discrete-robin}
\end{proof}

The discrete system endowed with Robin boundary conditions now reads:
$$\left \{\begin{aligned}
    \frac{1}{E} (\bm{M} + \ell^2 \, \bm{K} + \ell \, \bm{B  B}^\top) \frac{\d}{\d t} \overline{\sigma} &= D \, \overline{v}, \\
    (\bm{M}^\rho + \ell \, \bm{B} \bm{M}^\partial_\rho \bm{B}^\top) \frac{\d}{\d t} \overline{v} &= - D^\top \overline{\sigma}.
\end{aligned} \right.$$

Finally, one can define the Lagrange structure matrix with the Robin boundary conditions as:
$$\bm{P}_{\text{Rob}}:= \begin{bmatrix}
    \frac{1}{E} (\bm{M} + \ell^2 \, \bm{K} + \ell \, \bm{B  B}^\top) & 0 \\
    0 & (\bm{I}_N + \ell \, \bm{B} \bm{M}^\partial_\rho \bm{B}^\top)
\end{bmatrix}.$$
Which satisfies the symmetry assumption: $\bm{P}_{\text{Rob}} = \bm{P}_{\text{Rob}}^\top.$
\begin{definition}
    The Hamiltonian of the discretized nonlocal nanorod equation with Robin boundary conditions reads:
    $$ \mathcal H^d_{\text{Rob}} = \frac{1}{2} \begin{pmatrix}
        \overline{\sigma} \\ \overline{v}
    \end{pmatrix}^\top \bm{P}_{\text{Rob}} \begin{pmatrix}
        \overline{\sigma} \\ \overline{v}
    \end{pmatrix}. $$
\end{definition}

\subsection{Implicit Euler-Bernoulli beam}

The philosophy remains identical for the implicit and explicit Euler-Bernoulli beam models. We first write down the variational formulations and perform integrations-by-parts where appropriate, i.e., on each occurrence of $\partial^2_{x^2}$ in~\eqref{eq:beam-co-energy}--\eqref{eq:beam-observation}. Let $\phi \in H^1(a,b)$, the variational formulations read:
$$
\left\lbrace
\begin{array}{rcl}
\dsp D^{-1} \, \partial_t \int_a^b \phi \, \sigma \, \d x &=& \dsp - \int_a^b \partial_x \phi \, \partial_x v \, \d x + \left[ \phi \, \partial_x v \right]_a^b, \\
\dsp \rho h \, \partial_t \int_a^b \phi \, v \, \d x + \rho \frac{h^3}{12} \partial_t \int_a^b \partial_x \phi \, \partial_x v \, \d x &=& \dsp \int_a^b \partial_x \phi \, \partial_x \sigma \, \d x - \left[ \phi \, \partial_x \sigma \right]_a^b + \rho \frac{h^3}{12} \, \partial_t \left[ \phi \, \partial_x v \right]_a^b.
\end{array}
\right.
$$
Using the same discretization and notations as for the nanorod, the above formulations allow to obtain the space discretization of~\eqref{eq:beam-co-energy}--\eqref{eq:beam-observation} as:
\begin{multline}\label{eq:beam-PFEM}
\underbrace{
\begin{bmatrix}
D^{-1} \bm{M} & 0 & 0 & 0 & 0 \\
0 & \rho h \bm{M} + \frac{\rho h^3}{12} \bm{K} & - \frac{\rho h^3}{12} \bm{B} & 0 & 0 \\
0 & - \frac{\rho h^3}{12} \bm{B}^\top & 0 & 0 & 0 \\
0 & 0 & 0 & 0 & 0 \\
0 & 0 & 0 & 0 & 0
\end{bmatrix}
}_{\bm{P}}
\frac{\d}{\d t}
\begin{pmatrix}
\overline{\sigma} \\
\overline{v} \\
y_L \\
y^1_D \\
y^2_D
\end{pmatrix} \\
=
\underbrace{
\begin{bmatrix}
0 & -\bm{K} & 0 & 0 & \bm{B} \\
\bm{K} & 0 & 0 & -\bm{B} & 0 \\
0 & 0 & 0 & 0 & 0 \\
0 & \bm{B}^\top & 0 & 0 & 0 \\
-\bm{B}^\top & 0 & 0 & 0 & 0
\end{bmatrix}
}_{\bm{J}}
\begin{pmatrix}
\overline{\sigma} \\
\overline{v} \\
y_L \\
y^1_D \\
y^2_D
\end{pmatrix}
+
\begin{bmatrix}
0 & 0 & 0 \\
0 & 0 & 0 \\
\bm{M}^\partial & 0 & 0 \\
0 & -\bm{M}^\partial & 0 \\
0 & 0 & \bm{M}^\partial \\
\end{bmatrix}
\begin{pmatrix}
\frac{\d}{\d t} u_L \\
u^1_D \\
u^2_D
\end{pmatrix}.
\end{multline}

\begin{lemma}
Following~\eqref{eqn:linear-phs-power-balance}, an appropriate discrete Hamiltonian is given by:
$$
\mathcal H^d_1 := \frac{D^{-1} }{2} \overline{\sigma}^\top \bm{M}\overline{\sigma} + \frac{\rho h}{2} \left( \overline{v}^\top \bm{M}\overline{v} + \frac{h^2}{12} \overline{v}^\top \bm{K} \overline{v} \right)  - \frac{\rho h^3}{12} y_L^\top \bm{B}^\top \overline{v}.
$$
The power balance then reads:
$$
\frac{\d}{\d t}\mathcal H^d_1 = -y_D^{1\top} \bm{M}^\partial u_D^1 + y_D^{2\top} \bm{M}^\partial u_D^2 +  y_L^{\top} \bm{M}^\partial \frac{\d}{\d t}u_L.
$$
Another possibility is to discretize the continuous Hamiltonian~\eqref{eq:beam-Hamiltonian-def} directly, yielding:
$$
\mathcal H^d_2 := \frac{D^{-1} }{2} \overline{\sigma}^\top \bm{M}\overline{\sigma} + \frac{\rho h}{2} \left( \overline{v}^\top \bm{M}\overline{v} + \frac{h^2}{12} \overline{v}^\top \bm{K} \overline{v} \right).
$$
The power balance then reads as the continuous one~\eqref{eq:beam-Hamiltonian}:
$$
\frac{\d}{\d t}\mathcal H^d_2 = -y_D^{1\top} \bm{M}^\partial u_D^1 + y_D^{2\top} \bm{M}^\partial u_D^2 +  \frac{\d}{\d t}y_L^{\top} \bm{M}^\partial u_L.
$$
Furthermore, the Lagrange structure matrices are $\bm{P}$ defined on the left-hand side of~\eqref{eq:beam-PFEM} and $\bm{S} := \mathbb{M} := {\rm Diag}(\bm{M}, \bm{M}, \bm{M}^\partial, \bm{M}^\partial, \bm{M}^\partial)$, satisfying $\bm{P}^\top \mathbb{M}^{-1} \bm{S} = \bm{S}^\top \mathbb{M}^{-1} \bm{P}$. The Dirac structure matrix is given by $\bm{J}$ (on the right-hand side of~\eqref{eq:beam-PFEM}), and satisfies: $\bm{J}^\top = - \bm{J}$.
\end{lemma}

\begin{proof}
First, following Lemma~\ref{lemma:latent-state-hamiltonian-definition} we have $\mathcal H^d_1 := \frac{1}{2} \begin{pmatrix}
\overline{\sigma} \\
\overline{v} \\
y_L 
\end{pmatrix} ^\top \begin{pmatrix}
    D^{-1} \bm{M} & 0 & 0  \\
0 & \rho h \bm{M} + \frac{\rho h^3}{12} \bm{K} & - \frac{\rho h^3}{12} \bm{B}  \\
0 & - \frac{\rho h^3}{12} \bm{B}^\top & 0 & \\
\end{pmatrix}  \begin{pmatrix}
\overline{\sigma} \\
\overline{v} \\
y_L 
\end{pmatrix} = \frac{D^{-1} }{2} \overline{\sigma}^\top \bm{M}\overline{\sigma} + \frac{\rho h}{2} \left( \overline{v}^\top \bm{M}\overline{v} + \frac{h^2}{12} \overline{v}^\top \bm{K} \overline{v} \right)  - \frac{\rho h^3}{12} y_L^\top \bm{B}^\top \overline{v}$, the power balance is then obtained by a direct computation.

Regarding $\mathcal H^d_2$, we have:
$$
\begin{aligned}
    \mathcal H_2^d(\overline{\sigma},\overline{v}) :&= \mathcal H(\sigma^d,v^d) = \frac{1}{2}\int_a^b \rho h\left(  (v^d)^2 + \frac{h^2}{12}\left( \frac{\partial}{ \partial x } v^d\right)^2 \right) + D^{-1} \,  (\sigma^d)^2 \, \d x, \\
    &= \frac{1}{2} \sum_i\sum_j\int_a^b \rho h\left(  \phi_i\phi_j + \frac{h^2}{12}\left( \frac{\partial}{ \partial x } \phi_i\right)\left( \frac{\partial}{ \partial x } \phi_j\right) \right) \, \overline{v}_i\overline{v}_j \, \d x \\& \quad + \sum_i\sum_j\int_a^bD^{-1} \,  \phi_i\phi_j \, \, \overline{\sigma}_i\overline{\sigma}_j  \, \d x, \\
    &=\frac{D^{-1} }{2} \overline{\sigma}^\top \bm{M}\overline{\sigma} + \frac{\rho h}{2} \left( \overline{v}^\top \bm{M}\overline{v} + \frac{h^2}{12} \overline{v}^\top \bm{K} \overline{v} \right),
\end{aligned}
$$
and the power balance follows by direct computation.
\end{proof}

Note that Remark~\ref{rem:P-S} also holds in this example.

\subsection{Incompressible Navier--Stokes equations}

In \cite{zhangmass2022}, a staggered time-scheme is presented, allowing for both the enstrophy and kinetic energy to be preserved. More precisely, for a given time step $t_k$, the fluid velocity $\bm{u}$ is computed from $t_k$ to $t_{k+1}$ using the vorticity $\omega$ at the time $t_{k+1/2}$; likewise, the vorticity is then computed between $t_{k+1/2}$ and $t_{k+3/2}$ with the value of $\bm{u}$ at $t_{k+1}$. Such an approach has two advantages: firstly, it allows for linearizing the evaluation of the Lamb vector $\omega_k \times \bm{u}_k \approx \omega_{k+1/2} \times \bm{u}_k$ between each timestep, and secondly, this scheme preserves both the kinetic energy and enstrophy up to machine precision.

Writing the INSE as two coupled modulated pH systems~\eqref{eqn:navier-stokes-vorticite-courant} allows us to directly conclude on the conservative properties of the staggered scheme presented in \cite{zhangmass2022}. Indeed, fixing $J(\psi,\omega)$ at timestep $t_k$ in order to compute timestep $t_{k+1}$ keeps the structure matrix $J(\psi,\omega)$ skew-symmetric; ensuring the preservation of the Hamiltonian as proven in the following lemma.
\begin{lemma}
    Let $ \begin{pmatrix}
        \omega_0&\psi_0
    \end{pmatrix}^\top \in H^1 \times H^1$, a fixed state. Let $\begin{bmatrix}
        \omega(t,x)&\psi(t,x)
    \end{bmatrix}^\top$ the state satisfying the dynamics of the fixed Dirac system defined as:
    \begin{equation}
        \begin{bmatrix}
            -\rho_0 \Delta & 0\\
            0 & \rho_0
        \end{bmatrix} \begin{pmatrix}
            \partial_t \psi \\ \partial_t \omega
        \end{pmatrix} = \left[\rho_0\,\begin{bmatrix}
            -J_{\omega_0} & 0 \\
            0 & -\diver (\grad^\perp(\psi_0) \, \cdot \,)
        \end{bmatrix} - \mu \begin{bmatrix}
            \Delta^2 & 0 \\
            0 & -\Delta
        \end{bmatrix}\right] \begin{pmatrix}
            \psi \\ \omega
        \end{pmatrix}.
    \end{equation}
    Then, such a state satisfies the following power balance:
    $$ \begin{aligned}
        \frac{\d}{\d t} \mathcal{K} = \int_{\partial \Omega} \left(\rho_0 \, \psi \, \,  \partial_t (  \bm{n} \cdot \grad(\psi)) \,\, + \, \, \rho_0 \, \psi \, \,  \omega_0  \bm{n} \cdot \grad^\perp(\psi) \right) \, \d s\\
        +  \mu\int_{\partial \Omega}\left(- \,\psi \,  \bm{n}\cdot \grad(\Delta\psi) + \Delta \psi \, \bm{n} \cdot \grad(\psi)\right)\, \d s - \mu \int_\Omega |\Delta\psi|^2 \, \d x. 
    \end{aligned}   $$
    $$ \frac{\d}{\d t} \mathcal{E} = \int_{\partial \Omega} \left( \rho_0 \, \frac{1}{2}\, \omega^2 \,\bm{n} \cdot \grad^\perp(\psi_0) -  \mu\, \omega \, \bm{n} \cdot\grad^\perp(\omega)\right)\,   \d s - \mu\int_\Omega |\grad^\perp(\omega)|^2\, \d x.  $$
    In particular, with zero boundary conditions, this system is dissipative.
\end{lemma}

\begin{proof}
    The result is obtained simply by using the previously stated Theorem~\ref{thm:navier-stokes-vorticite-courant-power-balance} and replacing $\omega$ by the fixed state $\omega_0$ for the kinetic energy, and $\psi$ by the fixed state $\psi_0$ in the enstrophy balance.
\end{proof}

\subsubsection{Weak formulation}
Let $\phi_1,\phi_2\in H^1$, two test functions and write the weak formulation of the INSE written in vorticity--stream function formulation \eqref{eqn:navier-stokes-vorticite-courant}:
\begin{equation}
	\left\lbrace \begin{aligned}
		- \rho_0 \int_{\Omega} \phi_1  \Delta \partial_t \psi \, \d x &=  - \int_{\Omega} \rho_0\,\phi_1 \diver(\omega \, \,  \grad^\perp(\psi)) \, \d x - \mu \int_\Omega \phi_1 \,  \Delta^2 \psi \, \d x, \\
		\rho_0 \int_{\Omega} \phi_2  \partial_t \omega \, \d x &=  - \int_{\Omega} \rho_0\,\phi_2 \diver(\grad^\perp(\psi)  \, \, \omega ) \, \d x + \mu \int_\Omega  \phi_2 \, \Delta\omega \, \d x.		
	\end{aligned} \right.
\end{equation}

Let us now perform multiple integrations by parts to reveal the boundary controls of the system:
\begin{equation}\label{eqn:inse-coupled-weak-form}
	\left\lbrace \begin{aligned} 
		\rho_0 \int_{\Omega}  \grad(\phi_1) \cdot  \partial_t \grad(\psi) \, \d x& =  {\color{red} \int_{\partial \Omega}\rho_0\,\phi_1 \, \partial_t \,  \bm{n} \cdot \grad(\psi) \, \d s } + \int_{\Omega}  \rho_0\,\omega \, \,  \grad(\phi_1) \cdot \grad^\perp(\psi) \, \d x  \\
		&  \quad {\color{blue}- \int_{\partial \Omega} \rho_0\,\phi_1 \, \, \omega \,   \bm{n} \cdot \grad^\perp(\psi) \, \d s}	- \mu \int_\Omega \Delta\phi_1  \, \Delta\psi\, \d x \\
		& \quad  {\color{violet} - \mu \int_{\partial \Omega} \phi_1 \, \bm{n} \cdot \grad(\Delta \psi) \, \d s + \mu \int_{\partial\Omega} \Delta \psi \,  \bm{n} \cdot \grad(\phi_1) \, \d s}, \\ 	
		\rho_0 \int_{\Omega} \phi_2  \partial_t \omega \, \d x &=   - \int_{\Omega} \rho_0\, \grad^\perp(\phi_2) \cdot \grad(\omega)  \, \, \psi  \, \d x  {\color{blue} \, - \int_{\partial \Omega}\rho_0\, \phi_2 \, \psi  \,  \bm{t} \cdot \grad(\omega) \,  \d s }\\
		& \quad - \mu \int_\Omega  \grad(\phi_2)  \cdot \grad(\omega) \, \d x  {\color{violet} + \mu \int_{\partial \Omega} \phi_2 \,  \bm{n} \cdot \grad(\omega) \, \d s }.
	\end{aligned} \right.
\end{equation}
Note that the integration by parts applied on the integral $ \int_{\partial \Omega} \phi \, \diver(\omega \, \grad^\perp(\psi))\, \d x$ is carried out differently between the first and second equations, as in the first (resp. second) equation $\omega$ (resp. $\psi$) is considered the modulating variable, while $\psi$ (resp. $\omega$) is the effort variable. Making this distinction will allow us to prove the skew-symmetry of the corresponding matrices.

The blue boundary terms in the first and second equations correspond to power ports of the Stokes-Dirac structure, and the violet boundary terms to power ports of the resistive structure (friction at the boundary). Lastly, the red boundary term in the first equation corresponds to the energy port of the underlying Stokes-Lagrange structure. Notably, the $5$ boundary controls are: $u_L = \bm{n} \cdot \grad(\psi) |_{\partial \Omega} $, $u_D^1 = \bm{n} \cdot \grad^\perp(\psi) |_{\partial \Omega} $, $u_D^2 =  \bm{n} \cdot \grad(\Delta \psi) |_{\partial \Omega}  $, $u_D^3 = \Delta \psi |_{\partial \Omega} $, $u_D^4 = \bm{t}\cdot  \grad(\omega)_{|\partial \Omega}$, and $u_D^5 =  \bm{n} \cdot \grad(\omega)$. Using the fact that  $ \bm{t} \cdot \grad^\perp(\psi) = - \bm{n} \cdot \grad(\psi) $ allows us to rewrite \eqref{eqn:inse-coupled-weak-form} as:
\begin{equation}\label{eqn:inse-coupled-weak-form-with-controls}
	\left\lbrace \begin{aligned} 
		 \rho_0 \int_{\Omega}  \grad(\phi_1) \cdot  \partial_t \grad(\psi) \, \d x& =  {\color{red} \rho_0 \, \int_{\partial \Omega} \phi_1 \, \partial_t \, u_L \, \d s } + \int_{\Omega}\rho_0\,  \omega \, \,  \grad(\phi_1) \cdot \grad^\perp(\psi) \, \d x  \\
		&  \quad + {\color{blue}- \int_{\partial \Omega} \rho_0\,\phi_1 \, \, \omega \, u_D^1 \, \d s}	- \mu \int_\Omega \Delta\phi_1 \, \Delta\psi\, \d x \\
		& \quad  {\color{violet} - \mu \int_{\partial \Omega} \phi_1 \, u_D^2 \, \d s + \mu \int_{\partial \Omega}  u_D^3 \, \bm{n} \cdot \grad(\phi_1) \, \d s}, \\ 	
		\rho_0 \int_{\Omega} \phi_2  \partial_t \omega \, \d x &=  - \int_{\Omega} \rho_0\, \grad^\perp(\phi_2) \cdot \grad(\omega)  \, \, \psi  \, \d x  {\color{blue} + \int_{\partial \Omega} \rho_0\,\phi_2 \, \psi  \, u_D^4 \d s } \\
		& \quad - \mu \int_\Omega  \grad(\phi_2)  \cdot \grad(\omega) \, \d x  {\color{violet} + \mu \int_{\partial \Omega} \phi_2 \, u_D^5 \, \d s }.
	\end{aligned} \right.
\end{equation}

\begin{remark}
	Note that $\psi$  (resp. $\omega$) corresponds to a control in the first (resp. second) equation  and modulates the control operator on the second (resp. first) equation.
\end{remark}

\begin{remark}
	Reminding ourselves of the definition of the stream function $\bm{u} = \grad^\perp(\psi)$, we have that, $ \bm{t} \cdot \grad^\perp(\psi)$ corresponds to the tangential part of the velocity at the boundary and $ \bm{n} \cdot \grad^\perp(\psi)$ to the normal part of the velocity at the boundary.
\end{remark}

\subsubsection{Finite Element approximation}
Due to the dissipation term $\int_\Omega \Delta \psi \Delta \phi_1 \, \d x$, $\psi$ has to be in $H^2$, hence let us choose an Argyris finite element family $(\phi^1_i)_{i \in [1,N_\psi]}$. Given that Argyris elements are polynomials of degree 5 and the relation $-\Delta \psi = \omega$, let us choose $H^1$ conforming $P^3$ Lagrange finite element family $(\phi^2_i)_{i \in [1,N_\omega]}$ over $\Omega$. The discretized state variables are denoted by $ \hat\omega, \hat\psi$ and are defined as: $\hat\psi(t,x) = \sum_{i=1}^{N_\psi} \phi_i^1(x) \overline{\psi}_i(t)$, $\hat \omega(t,x) = \sum_{i=1}^{N_\omega} \phi^2_i(x) \overline{\omega}_i(t)$.

Moreover, let us consider a $P^1$ finite element family $(\nu_i)_{i \in [1,M]}$ defined on the boundary of our domain $\partial \Omega$ and write for all $k\in \{1,2,3,4,5\}$ the discretized controls and observations $\hat u_D^k,\hat y_D^k$ as: 
$$\hat u_D^{k}(t,x) := \sum_{i=1}^{M} \nu_i(x) (\overline{u}_D^k)_i(t), \qquad \hat y_D^k(t,x) := \sum_{i=1}^{M} \nu_i(x) (\overline{y}_D^k)_i(t),$$
$$\hat u_L(t,x) := \sum_{i=1}^{M} \nu_i(x) (\overline{u}_L)_i(t), \qquad \hat y_L(t,x) := \sum_{i=1}^{M} \nu_i(x) (\overline{y}_L)_i(t).$$

Let us now define the finite element matrices as follows:
\begin{equation}
    \begin{aligned}
        \bm{M}_{ij} = \int_\Omega \phi^1_i \, \phi^1_j \, \d x, \quad \bm{M}^\partial_{ij} =& \int_{\partial \Omega} \nu_i \, \nu_j \, \d s, \quad \bm{K}_{ij} = \int_\Omega \grad(\phi^1_i) \cdot \grad(\phi^1_j) \, \d x \\
        \bm{D}^1(\overline{\omega})_{ij} = \int_\Omega \omega^d \, \grad(\phi^1_i) \cdot \grad^\perp(\phi^1_j) \, \d x&, \quad \bm{D}^2(\overline{\psi})_{ij} = \int_\Omega \psi^d \, \grad^\perp(\phi^2_i) \cdot \grad(\phi^2_j) \, \d x,\\
         \bm{R}^1_{ij} = \int_\Omega \Delta\phi^1_i \, \Delta\phi^1_j \, \d x,& \quad  \bm{R}^2_{ij} = \int_\Omega \grad(\phi^2_i) \cdot \grad(\phi^2_j) \, \d x,  \\
         \bm{B}^1_{ij} = \int_{\partial \Omega} \phi^1_i \, \nu_j \, \d s, \quad    \bm{B}^2_{ij}(\overline{\omega}) = \int_{\partial \Omega} \phi_i^1 \, \omega^d \, \nu_j \, \d s,& \quad \bm{B}^3_{ij} = \int_{\partial \Omega} \nu_j \, \bm{n} \cdot \grad(\phi^1_i) \, \d s, \quad  \\
        \bm{B}^4_{ij}(\overline{\psi}) = \int_{\partial \Omega} \phi_i^2 \, \psi^d \, \nu_j \, \d s &, \quad  \bm{B}^5_{ij} = \int_{\partial \Omega} \phi^2_i \, \nu_j \, \d s.
    \end{aligned}
\end{equation}

Denoting by:
\begin{equation}
    \begin{aligned}
        \mathbb{B}^1(\omega)& = \begin{bmatrix}
             - \rho_0\, \bm{B}^2(\overline{\omega}) & - \mu  \bm{B}^1 & \mu \bm{B}^3 
             \end{bmatrix}, \quad \mathbb{M}_\partial^1 = \rm{Diag} \begin{bmatrix}
                 \bm{M}^\partial & \bm{M}^\partial & \bm{M}^\partial
             \end{bmatrix}, \\
            \mathbb{B}^2(\psi) &=  \begin{bmatrix}
                \rho_0\,\bm{B}^4(\overline{\psi})&  \mu \bm{B}^5
            \end{bmatrix}, \quad \mathbb{M}_\partial^2 = \rm{Diag} \begin{bmatrix}
                 \bm{M}^\partial & \bm{M}^\partial
             \end{bmatrix},
    \end{aligned}
\end{equation}
the blocks defining total power control port matrices and mass matrices. Additionally, let us denote by $U_D^1 = \begin{pmatrix}
            (\overline{u}_D^1)^\top & (\overline{u}_D^2)^\top &  (\overline{u}_D^3)^\top 
        \end{pmatrix}^\top$, and $ U_D^2 = \begin{pmatrix}
            (\overline{u}_D^4)^\top & (\overline{u}_D^5)^\top
        \end{pmatrix}^\top $, the total power control port vectors. Finally, let us define the observation of the system as: $\mathbb{M}_\partial^1Y^1_D = \mathbb{B}^1(\omega)^\top U_D^1, \mathbb{M}_\partial^2Y^2_D = \mathbb{B}^2(\psi)^\top U_D^2$, and $\bm{M}^\partial \overline{y}_L = \rho_0 \bm{B}^{1\top} \overline{u}_L$.
Then, the discretized INSE reads:
\begin{equation}\label{eqn:inse-coupled-discr-with-controls}
\begin{aligned}
         \rho_0 \,  \frac{\d}{ \d t} \, \begin{bmatrix}
            \bm{K} & 0 \\0 & \bm{M}
        \end{bmatrix}\begin{pmatrix}
            \overline{\psi}  \\ \overline{\omega}
        \end{pmatrix} = &\begin{bmatrix}
            \begin{bmatrix}
                \bm{D}^1(\overline{\omega}) & 0 \\
                0 & -\bm{D}^2(\overline{\psi})
            \end{bmatrix} - \mu \begin{bmatrix}
                \bm{R}^1 & 0 \\0 & \bm{R}^2
            \end{bmatrix}
        \end{bmatrix} \begin{pmatrix}
            \overline{\psi} \\ \overline{\omega}
        \end{pmatrix} \\ & + \begin{bmatrix}
            \mathbb{B}^1(\omega) & 0 \\0 & \mathbb{B}^2(\psi)
        \end{bmatrix} \begin{pmatrix}
            U_D^1 \\ U_D^2
        \end{pmatrix}+ \begin{bmatrix}
            \rho_0\, \bm{B}^1 & 0 \\0 & 0
        \end{bmatrix} \begin{pmatrix}
            \frac{\d}{\d t} \overline u_L \\ 0
        \end{pmatrix}.
\end{aligned} 
\end{equation}
Regarding the structure of the system, we get the following lemma:
\begin{lemma}\label{lemma:discrete-navier-stokes-algebraic-properties} For all $\overline{\psi} \in \mathbb{R}^{N_\psi},\overline{\omega} \in \mathbb{R}^{N_\omega},$ 
the Dirac structure matrices satisfy the skew-symmetry property:
$\bm{D}^1(\overline{\omega}) = - \bm{D}^1(\overline{\omega})^\top \quad \bm{D}^2(\overline{\psi}) = - \bm{D}^2(\overline{\psi})^\top.$
The Lagrange structure matrices satisfy the symmetry property:
     $  \bm{K} = \bm{K}^\top, \quad \bm{M} = \bm{M}^\top, $
and the resistive structure matrices satisfy the symmetry property:
$\bm{R}^1=\bm{R}^{1\top}, \quad \bm{R}^2=\bm{R}^{2\top}.$
\end{lemma}
Let us now define the discrete Hamiltonian, enstrophy and kinetic energy:
\begin{definition} The discrete kinetic energy, enstrophy and Hamiltonian are defined as:
    $$ \mathcal{K}^d(\overline{\mathcal{\psi}}) := \frac{\rho_0}{2} \int_\Omega \grad(\hat\psi) \cdot \grad(\hat\psi) \, \d x = \frac{\rho_0}{2}  \, \, \overline{\psi}^\top \, \bm{K}  \,\overline{\psi}, \qquad \mathcal{E}^d(\overline{\omega}) := \frac{\rho_0}{2} \int_\Omega (\hat \omega)^2 \, \d x = \frac{\rho_0}{2} \overline{\omega}^\top  \, \bm{M} \,  \overline{\omega},$$
    $$ \mathcal{H}^d(\overline{\psi},\overline{\omega}) = \mathcal{K}^d(\overline{\psi}) +\mathcal{E}^d(\overline{\omega}).$$
\end{definition}
Computing the power balance  yields:

\begin{lemma} The power balance of the discretized INSE~\eqref{eqn:inse-coupled-discr-with-controls} reads:
    $$ \frac{\d}{\d t} \mathcal{K}^d(\overline{\psi}) = -  \mu \, \overline{\psi}^\top  \bm{R}^1 \, \overline{\psi} + (Y_D^1)^\top \mathbb{M}^2_\partial \, U_D^1 + \overline{y}_L^\top \bm{M}^\partial \frac{\d}{\d t} \overline u_L, $$

    $$ \frac{\d}{\d t} \mathcal{E}^d(\overline{\omega}) = - \mu \, \overline{\omega}^\top \bm{R}^2 \, \overline{\omega} + (Y_D^2)^\top \mathbb{M}_\partial^2 \, U_D^2.$$
\end{lemma}
\begin{proof}
    Computing $\frac{\d}{\d t} \mathcal{K}^d(\overline{\omega})$ and $\frac{\d}{\d t} \mathcal{E}^d(\overline{\psi})$, then using the properties of $\bm{D}^1(\overline{\omega})$ and $\bm{D}^2(\overline{\psi})$ given in Lemma~\ref{lemma:discrete-navier-stokes-algebraic-properties}, yield the result.
\end{proof}
Finally, let us add the no-slip impermeable boundary conditions:
\begin{equation}\label{eqn:inse-coupled-discr-with-noslip}\left\{
\begin{aligned}
         \rho_0 \,  \frac{\d}{ \d t} \, \begin{bmatrix}
            \bm{K} & 0 \\0 & \bm{M}
        \end{bmatrix}\begin{pmatrix}
            \overline{\psi}  \\ \overline{\omega}
        \end{pmatrix} = &\begin{bmatrix}
            \begin{bmatrix}
                \bm{D}^1(\overline{\omega}) & 0 \\
                0 & -\bm{D}^2(\overline{\psi})
            \end{bmatrix} - \mu \begin{bmatrix}
                \bm{R}^1 & 0 \\0 & \bm{R}^2
            \end{bmatrix}
        \end{bmatrix} \begin{pmatrix}
            \overline{\psi} \\ \overline{\omega}
        \end{pmatrix} \\ & +  \begin{bmatrix}
        \mu\,\bm{B}^1 &
            \mu\,\bm{B}^3 & \bm{B}^3 \\ 0&0 & 0
        \end{bmatrix} \begin{pmatrix}
            u_D^1 \\ u_D^5 \\ \widetilde u_D
        \end{pmatrix} + \mu \begin{bmatrix}
             0 \\\bm{B}^5
        \end{bmatrix} u_D^5, \\
        0 =& - \bm{B}^{1\top} \overline{\psi}, \qquad \text{(No slip B.C.)}\\
        0 =& - \bm{B}^{3\top} \overline{\psi}, \qquad \text{(Impermeable B.C.)} \\
        \bm{M}^\partial u_D^1 =& \bm{B}^{5\top}\overline{\omega}, \qquad \text{(Vorticity generation)}
\end{aligned} \right.
\end{equation}
where $\widetilde u_D$ is the Lagrange multiplier enforcing the Dirichlet boundary condition. Note that the boundary vorticity value $\bm{B}^{5\top}\overline{\omega}$ is driven by the Lagrange multiplier $u_D^1$, and its corresponding vorticity flux (Lagrange multiplier $u_D^5$) is present in the vorticity and stream function dynamics.

\section{Numerical results}
\label{sec:numerics}
Hereafter, numerical results were obtained using the Python interface of the finite element library GetFEM \cite{renard2020getfem} and \href{https://gmsh.info/}{Gmsh} for mesh generation. Visualization was carried out using \href{https://www.paraview.org/}{ParaView} and \href{https://matplotlib.org/}{Matplotlib}. \blue{The source code is open and can be found at \url{https://gitlab.isae-supaero.fr/an.bendimerad-hohl/implicit-constitutive-relations.git}. Furthermore, raw data of the presented simulations are available at \url{https://doi.org/10.5281/zenodo.19488237}.}
\subsection{Nanorod equation}
Let us now study the numerical solution of the nanorod equation obtained after discretization. We will consider the domain $\Omega = [0,1]$, the Young's modulus $E = 1$, the mass density $\rho=10$, and as initial conditions, we choose $v_0(x) = \exp(- 80(x-0.3)^2)$ and $\sigma_0(x)=0$. We will then study the effect of the parameter $\ell$ over $ 0$, $0.01$, and $0.05$. Finally, the implicit midpoint rule is chosen as the time scheme.
  The simulation parameters are: final time $T_f = 10$, time step $\d t = 0.1$ and number of discretization points $N = 100.$
  Figure~\ref{fig:nanorod-hamiltonian-power-balance}  shows of the evolution of the Hamiltonian for the different values of $\ell$, along with respective relative energy errors $|\mathcal H_{\text{Rob}}^d(t) - \mathcal H_{\text{Rob}}^d(0)|/\mathcal H_{\text{Rob}}^d(0)$. 
  
  In each setting, the model is conservative hence the total mechanical energy should remain constant, however, it is clear that increasing the parameter $\ell$ decreases the accuracy of the method. Such behaviour might be caused by the degradation of the condition number of the matrix $\bm{M} + \ell^2 \bm{K} + \ell \bm{BB}^\top$ as $\ell$ is increased (see~\ref{apx:nanorod-condition-number}). Additionally, on the top row, the green and red lines isolate the energy contribution of the velocity energy port $\ell \, \overline{v}\bm{BB}^\top \overline{v}$ and stress energy port $\frac{\ell}{2E} \, \overline{\sigma}\bm{BB}^\top \overline{\sigma}$ respectively. These show that the amount of energy stored at the boundaries becomes significant as $\ell$ increases and should not be omitted when considering Robin boundary conditions.
\begin{figure}
    \centering
    \includegraphics[width=0.8\textwidth]{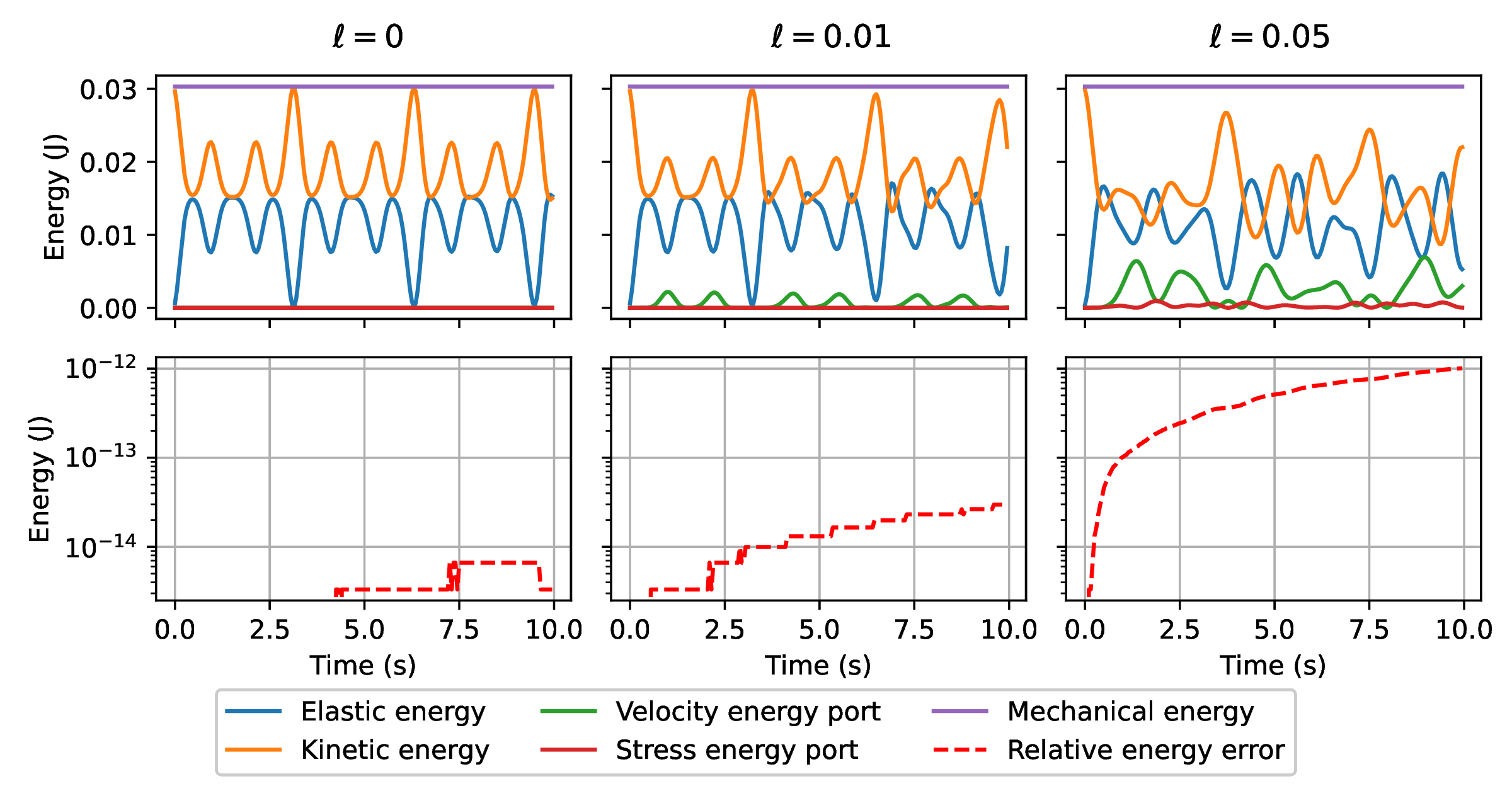}
    \caption{Plots of the nanorod Hamiltonian and relative energy error for $\ell = 0,0.01,0.05$.}
    \label{fig:nanorod-hamiltonian-power-balance}
\end{figure}
\subsection{Implicit Euler-Bernoulli beam}
In order to illustrate the efficiency of the approach, we simulate a simply supported steel beam of 1 meter length, for two different radii $r$: 5~cm, and 2.5~cm. The constant physical parameters of steel can be found in many references, we use the values given in~\cite{BilbaoJASA2016}. More precisely:
the mass density is $\rho = 7.86\times10^3$ kg/m\textsuperscript{3}, the Young's modulus is $E = 2.02\times10^{11}$ Pa, and the Poisson's ratio is $\nu = 0.3$. These values give as cross-section areas:
$$
h = \pi r^2 \in \{ 7.85\times10^{-3}, 1.9625\times10^{-3} \} ~ \text{m}^2,
$$
and as modulus of flexural rigidity:
$$
D = E \frac{h^3}{12(1-\nu^2)} \in \{ 8.95\times10^{+3}, 1.40\times10^{+2} \} ~ \text{Nm}.
$$
The finite element model of Section~\ref{sec:PFEM}, and especially the matrices of~\eqref{eq:beam-PFEM}, are validated using the modal analysis mentioned in Remark~\ref{rem:phase-velocity}. The phase velocity is computed for the different configurations and two mesh size parameters in Figure~\ref{fig:phase-velocity}. One may appreciate a relative error spanning between $\sim10^{-3}$ (at low frequencies) and $\sim10^{-1}$ (at high frequencies). These errors obviously strongly depend on the mesh size.
\begin{figure}
    \centering
    \includegraphics[width=0.49\linewidth]{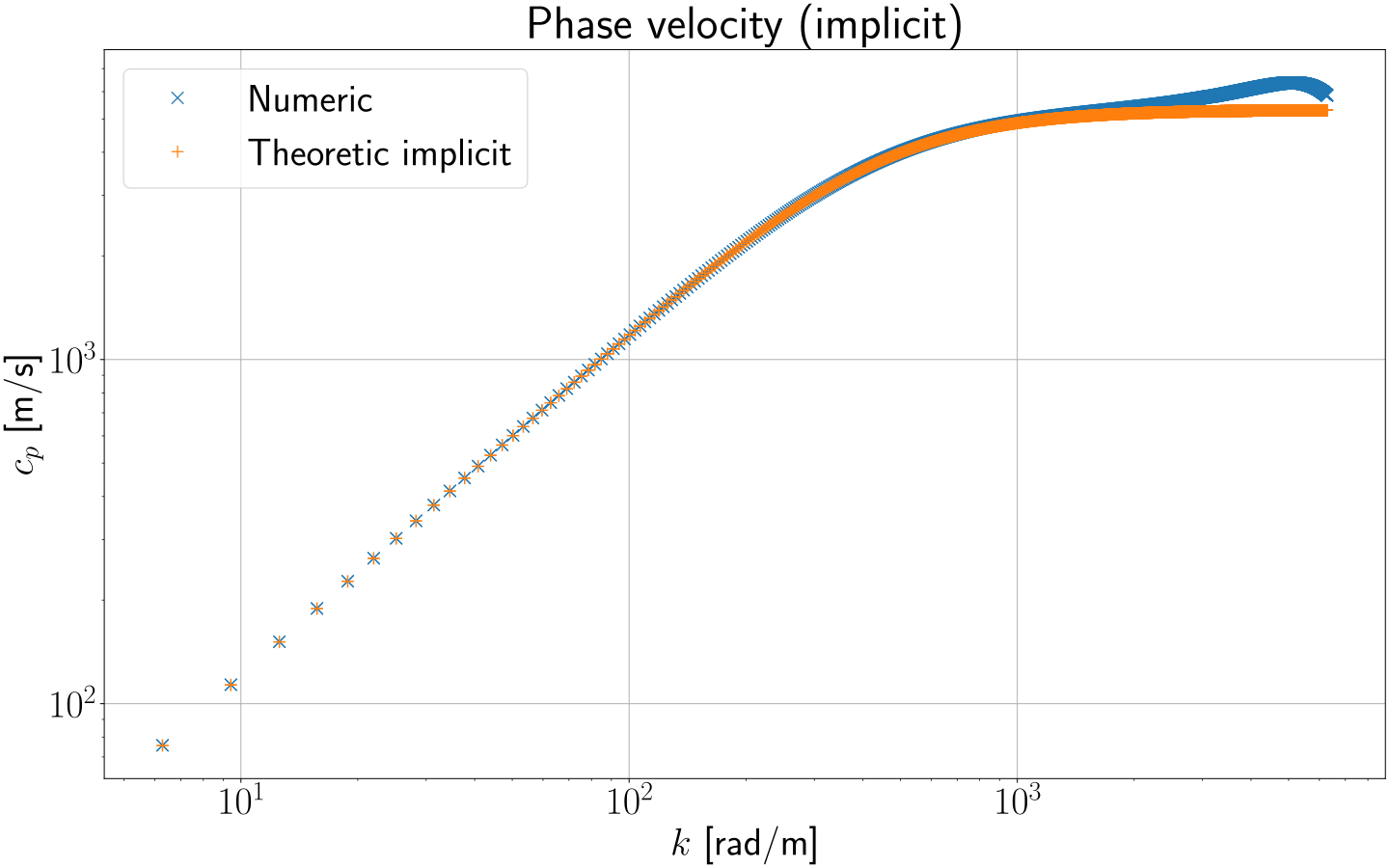} 
    \includegraphics[width=0.49\linewidth]{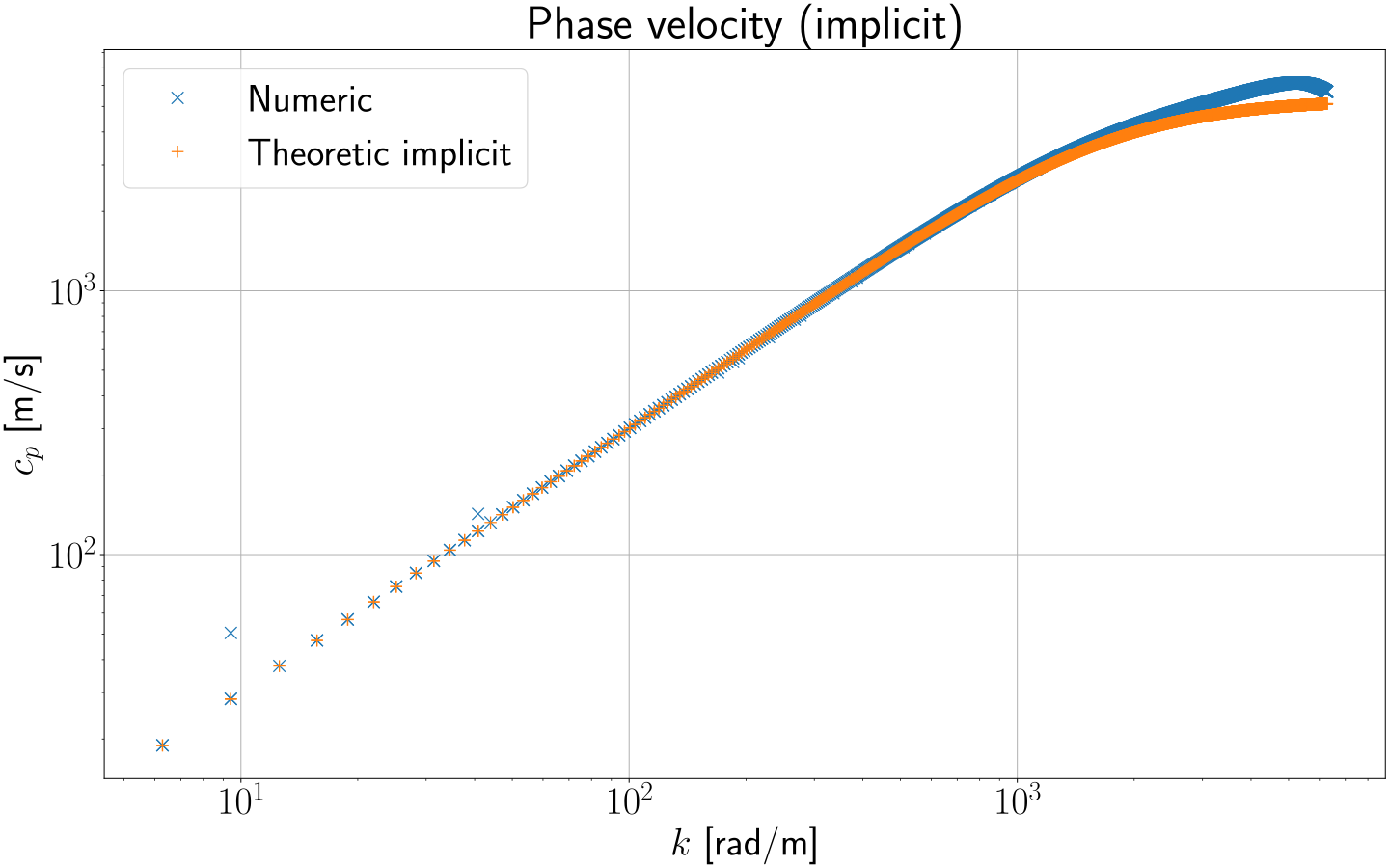} 
    \caption{Phase velocity for $r=5~$cm (left) and $r=2.5~$cm (right), and a mesh size parameter $\d x = 5.0\times10^{-4}$ (implicit Euler-Bernoulli beam).}
    \label{fig:phase-velocity}
\end{figure}

Once validated, system~\eqref{eq:beam-PFEM} may be used to simulate the evolution of a beam. Let us compare the evolution of the implicit Euler-Bernoulli model against the usual (explicit) Euler-Bernoulli model, always with homogeneous Dirichlet boundary conditions for $v$ and $\sigma$, i.e., a simply supported beam. The chosen time scheme is a Crank-Nicolson scheme with adaptative time step $\d t$. The final time is 10~ms, which is sufficient to observe several oscillations. The initial velocity is taken null, while the initial value for $\sigma$ is taken such that the initial position is an exponential bump centered at $x=0.5$. Figure~\ref{fig:evolution-beam} shows the evolution of a beam of radius $r=5$~cm, i.e., with $(h, D) = ( 7.85\times10^{-3}, 8.95\times10^{+3} )$, and a mesh size parameter $\d x = 5.0\times10^{-4}$.
\begin{figure}
    \centering
    \includegraphics[width=0.49\linewidth]{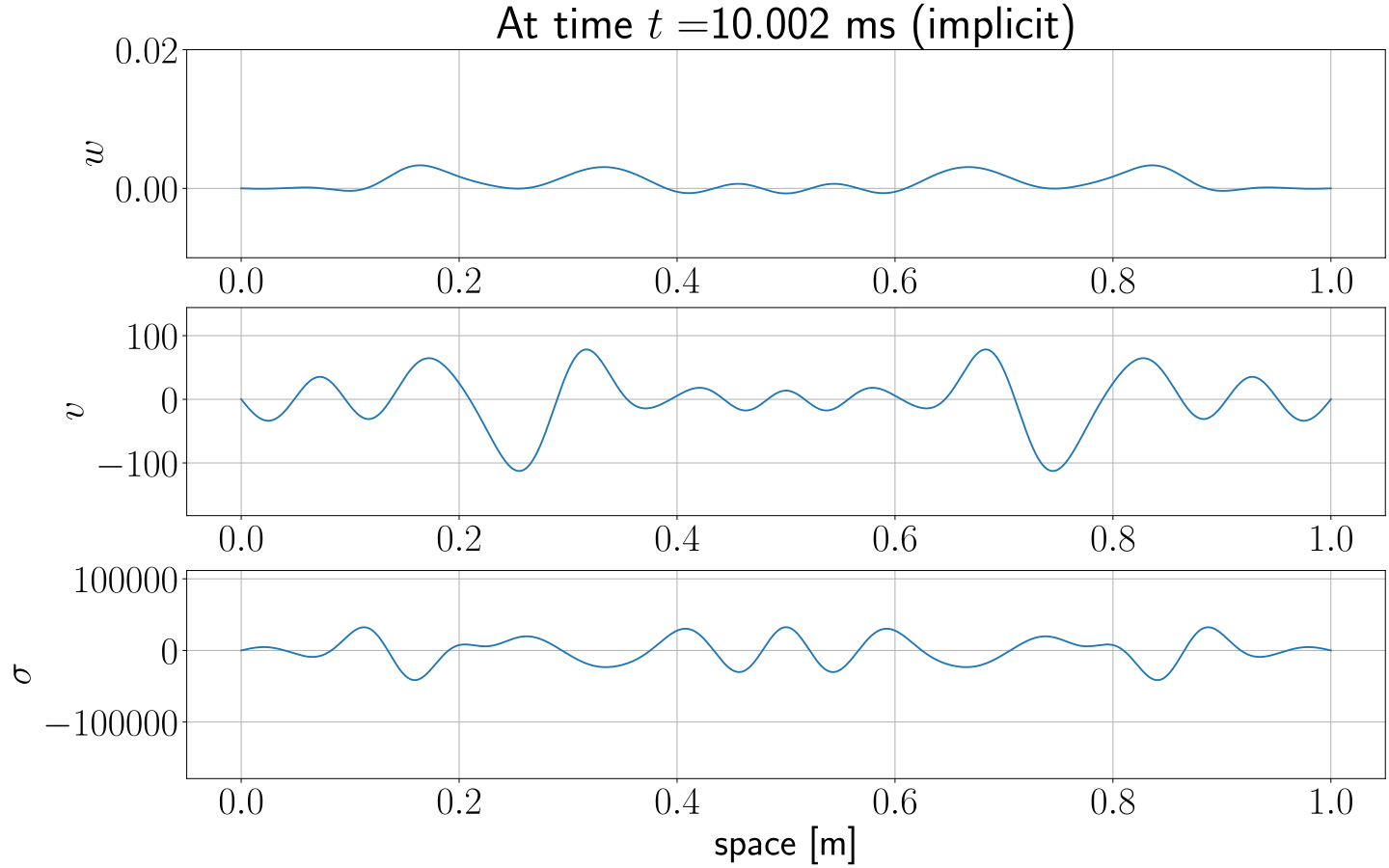}
    \includegraphics[width=0.49\linewidth]{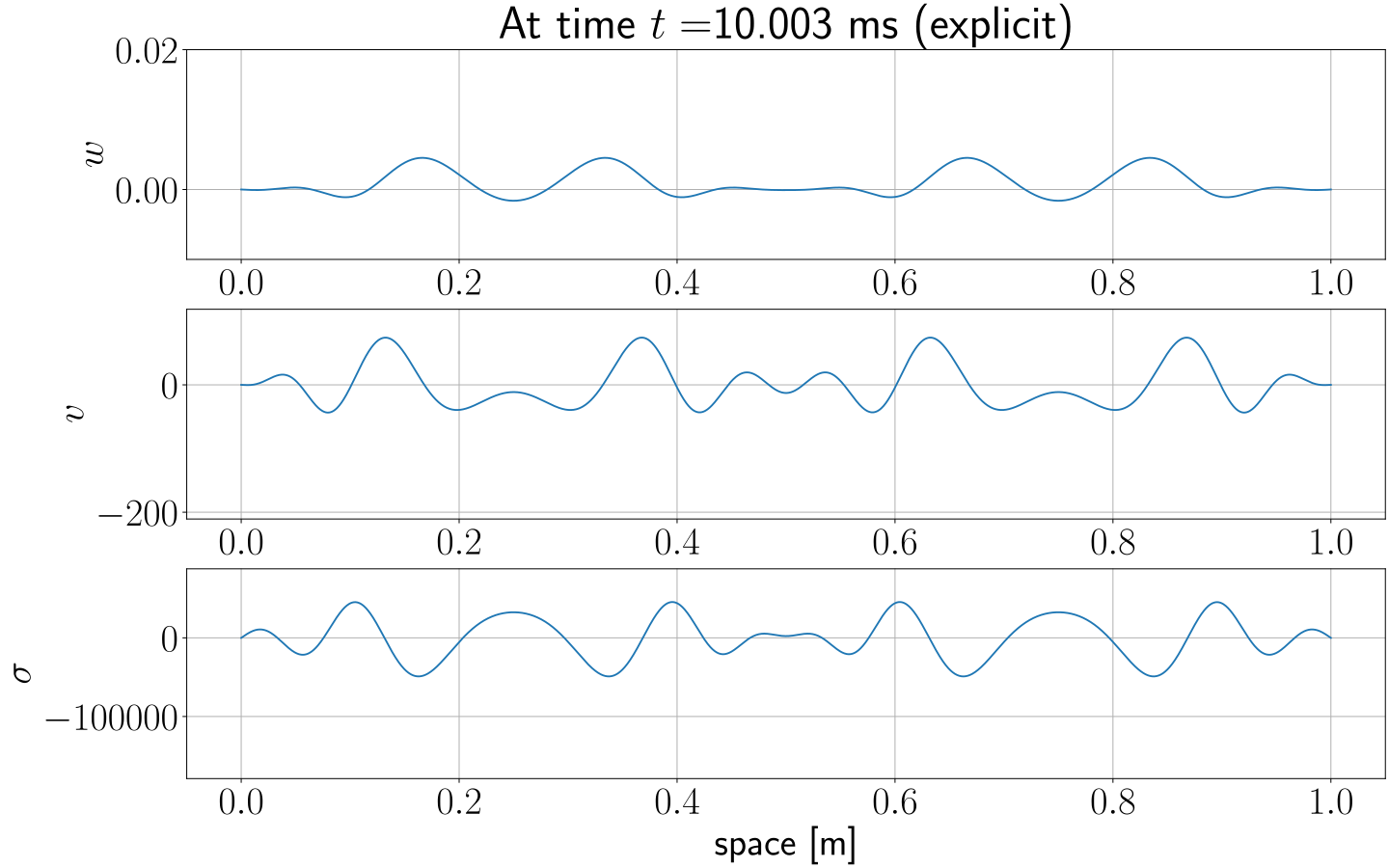}
    \caption{Snapshot of the implicit (left) and explicit (right) Euler-Bernoulli beam of radius $r=5$~cm with simply supported boundary conditions, at time $t=10$~ms. Parameters: $(h, D) = ( 7.85\times10^{-3}, 8.95\times10^{+3} )$, mesh size $\d x = 5.0\times10^{-4}$.}
    \label{fig:evolution-beam}
\end{figure}

While the difference between both models is not visually blatant in Figure~\ref{fig:evolution-beam}, it becomes clear in Figure~\ref{fig:difference}, where the $L^2$-norm of the difference between the position fields $w$ over time is represented.
\begin{figure}
    \centering
    \includegraphics[width=0.5\linewidth]{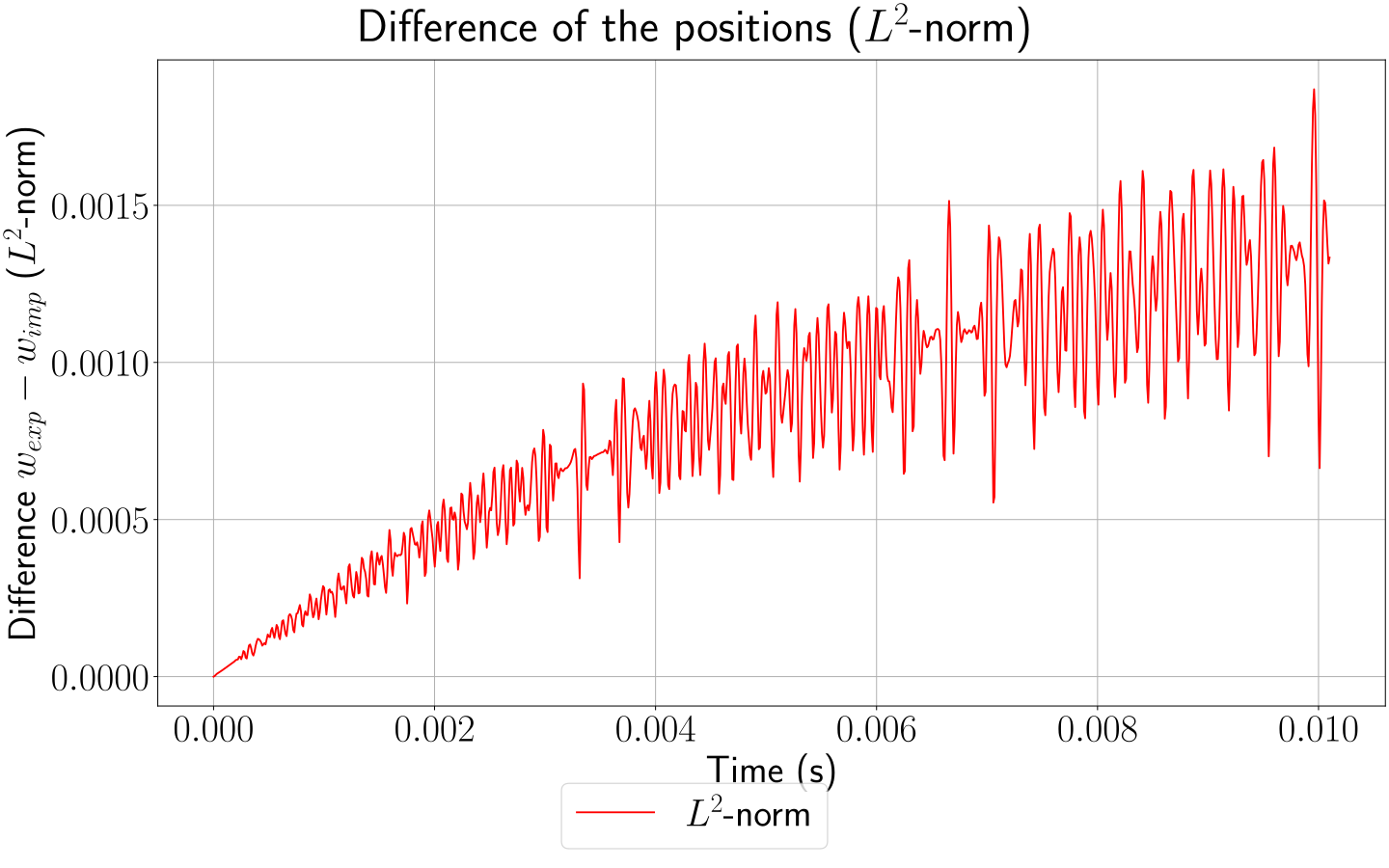}
    \caption{$L^2$ difference between the beam position computed with the explicit and implicit model over time.}
    \label{fig:difference}
\end{figure}

In Figure~\ref{fig:Hamiltonian}, the preservation of the underlying Dirac and Lagrange structures are enlightened: the variations of Hamiltonian are close to machine precision. One may furthermore appreciate that the implicit parameter, while improving the quality of the simulation for high frequencies, the requirement of CPU time is not significantly different. This latter is mainly due to the better condition number of the matrix $\rho h \bm{M} + \frac{\rho h^3}{12} \bm{K}$ on the left-hand side of~\eqref{eq:beam-PFEM}, when $h$ belongs to an appropriate range, because this results in a larger timestep, which counterbalances the time spent for the resolution of the linear system involving $\bm{K}$. Indeed, it has already been explained in~\cite{ducceschi2019conservative} that the implicit Euler-Bernoulli beam (the \emph{shear} model), is well-suited for beams of moderate thickness, while Timoshenko's model is better suited for thick beams and (explicit) Euler-Bernoulli's for thin beams.
\begin{figure}
    \centering
    \includegraphics[width=0.49\linewidth]{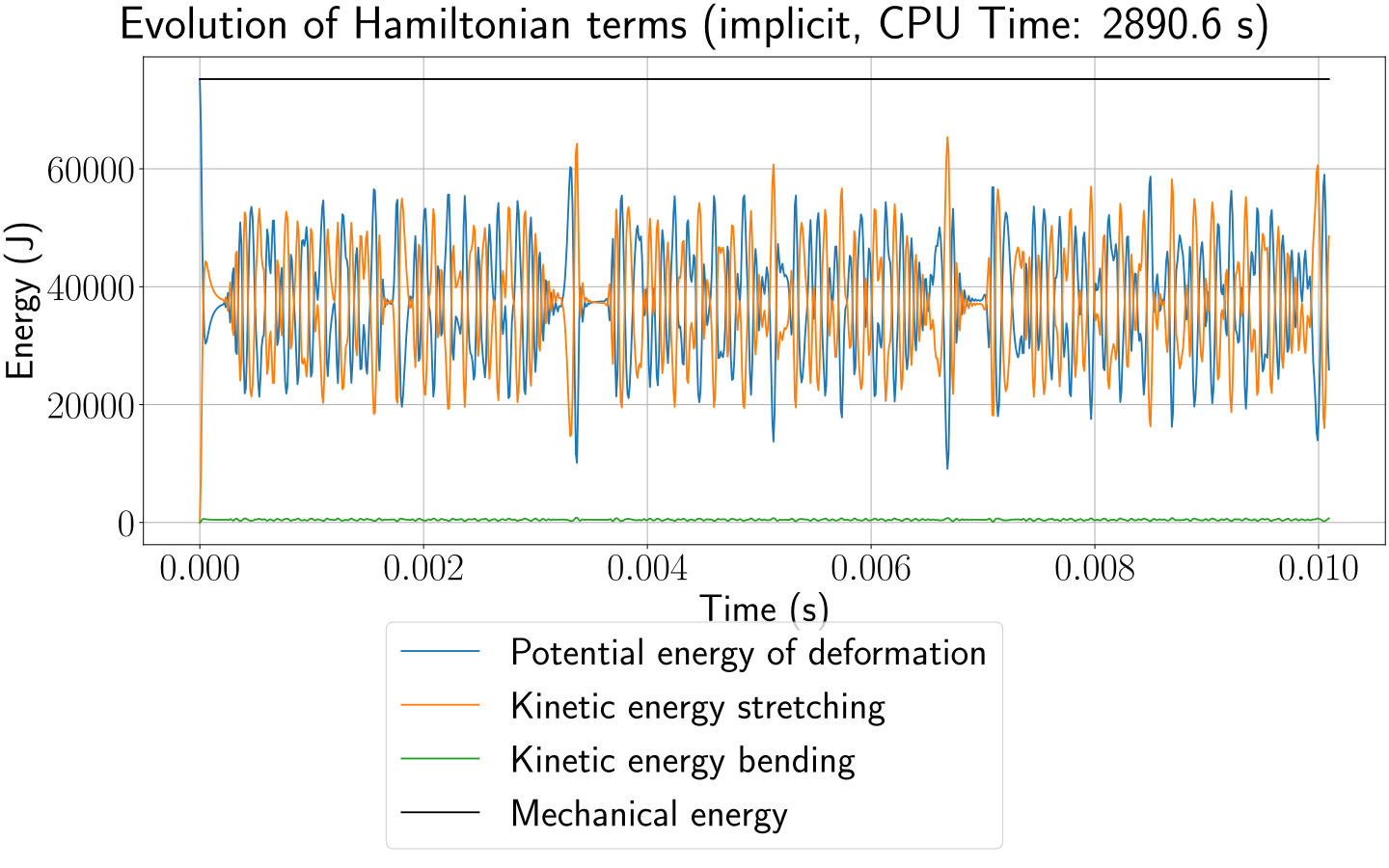} 
    \includegraphics[width=0.49\linewidth]{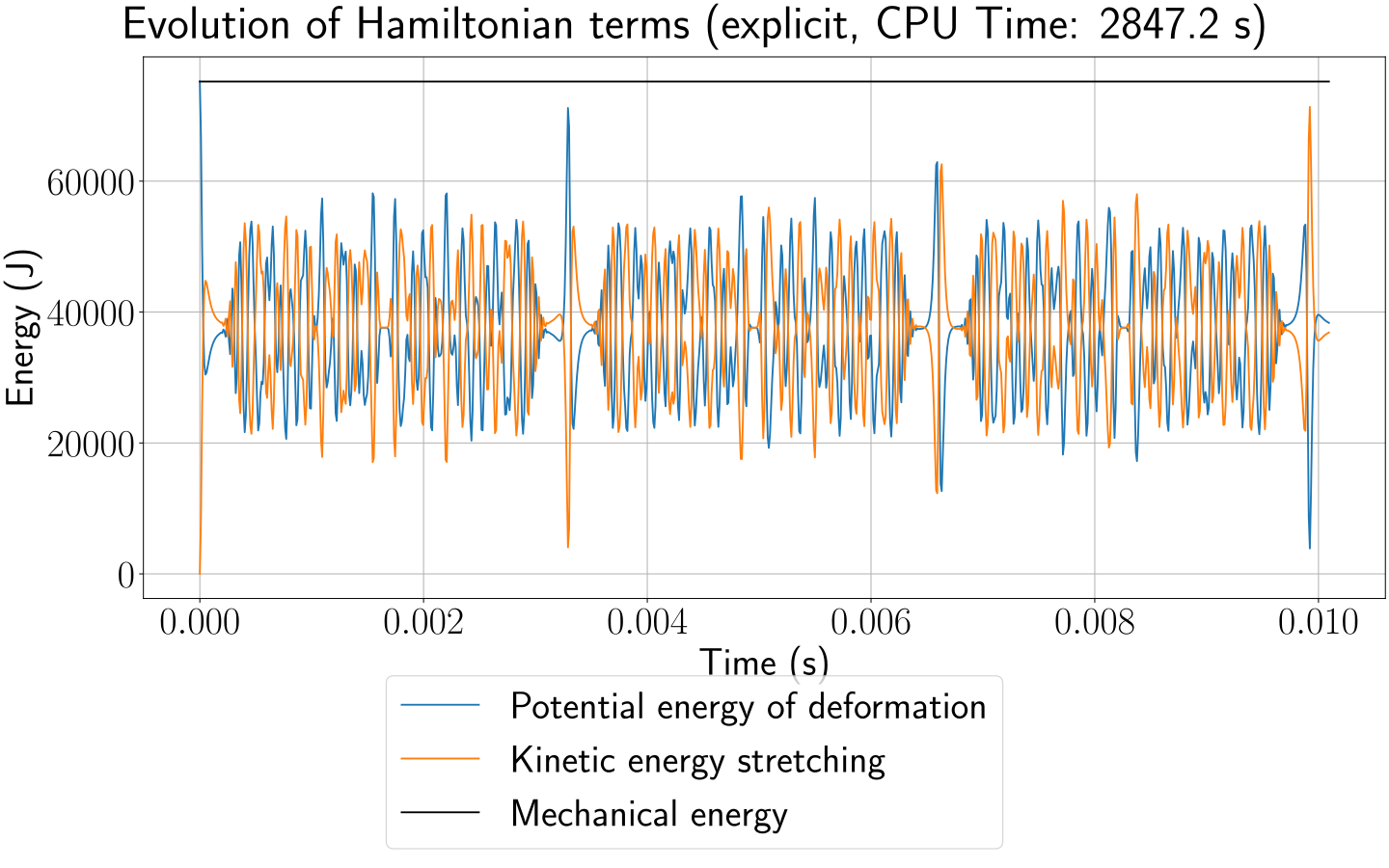} \\
    \includegraphics[width=0.49\linewidth]{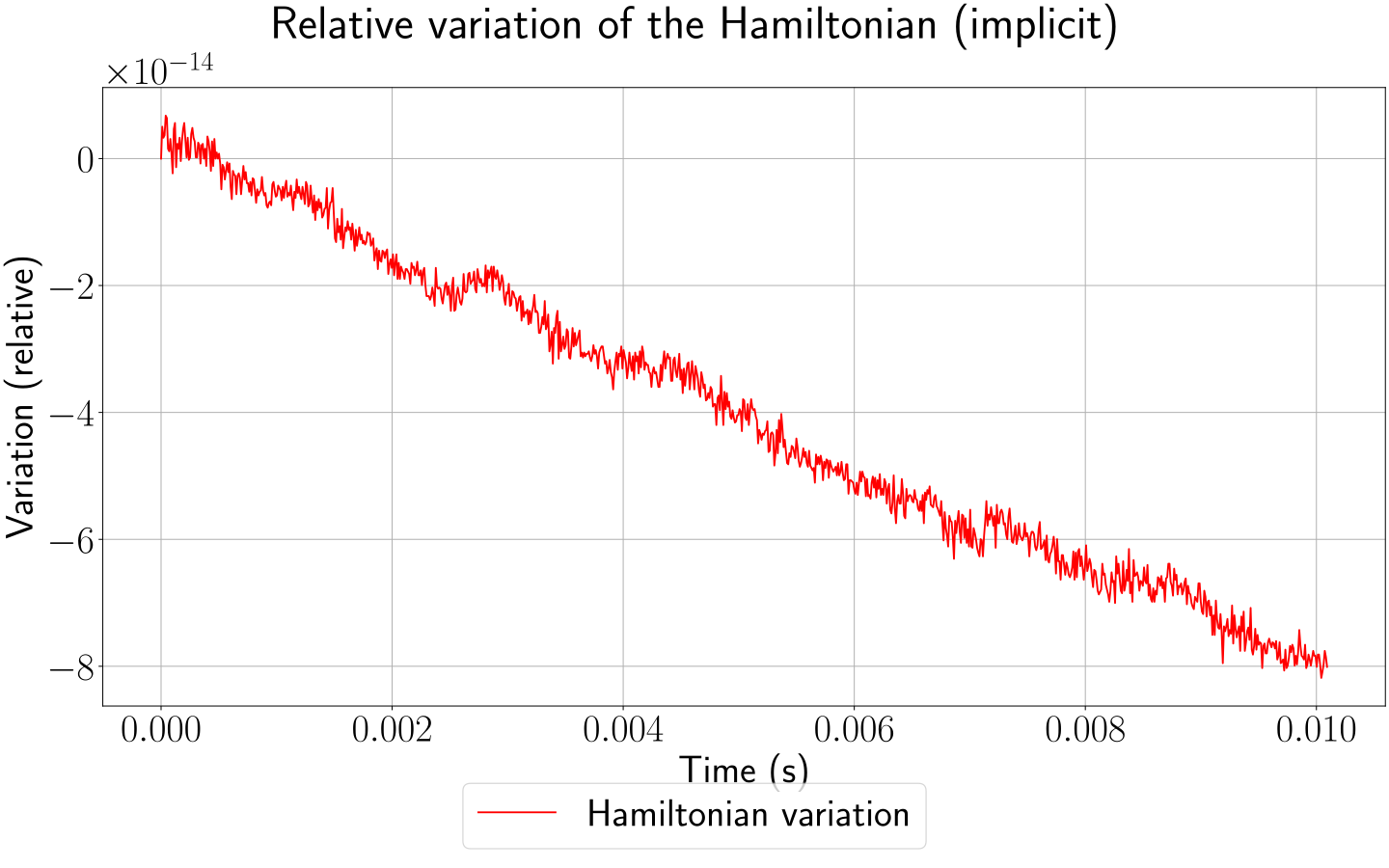}
    \includegraphics[width=0.49\linewidth]{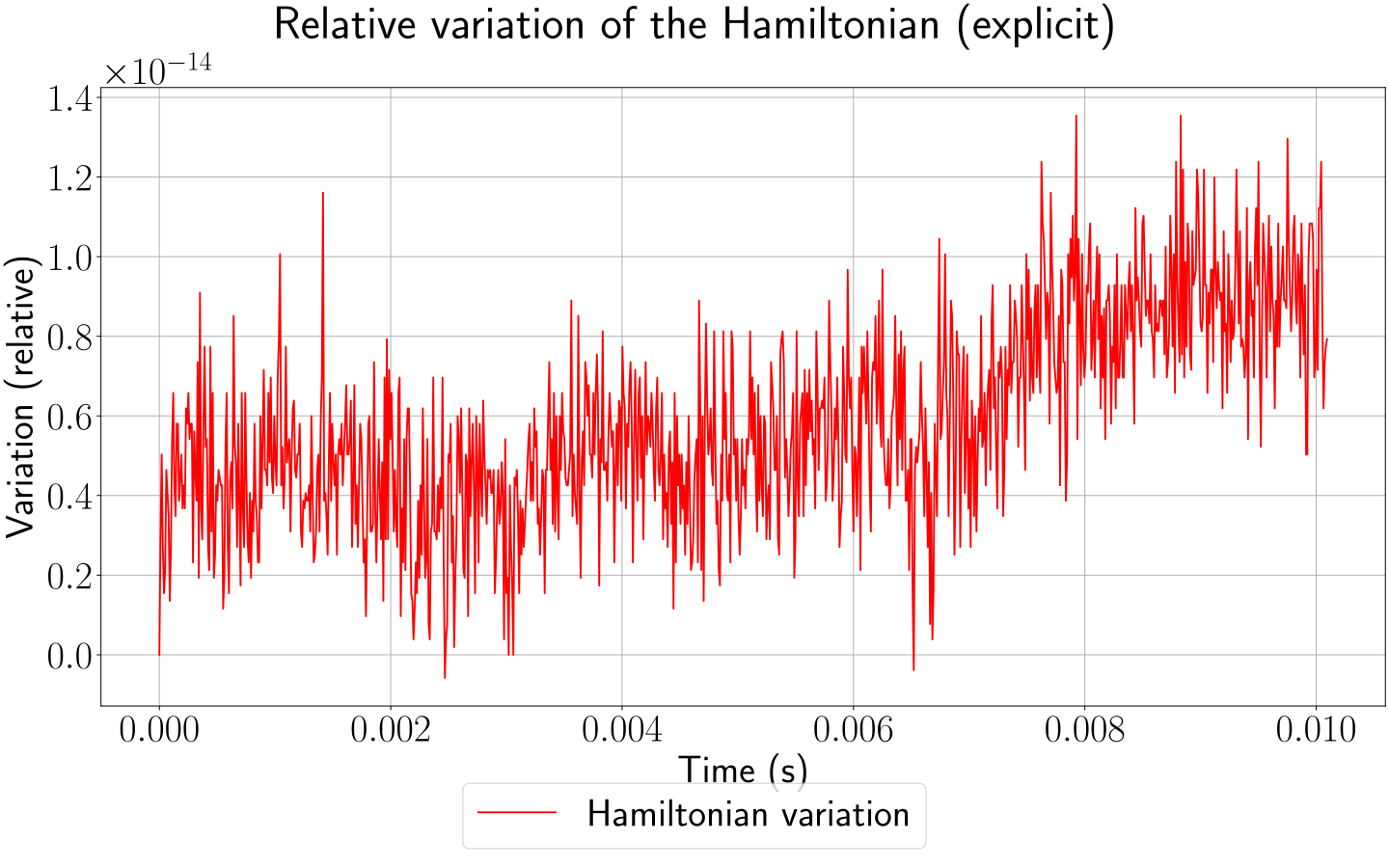} 
    \caption{Evolution of the Hamiltonian and its components (first line) for the implicit (left) and explicit (right) Euler-Bernoulli beam. The second line presents the relative variations of the Hamiltonian in both cases (left: implicit, right: explicit).}
    \label{fig:Hamiltonian}
\end{figure}

\subsubsection{Convergence analysis}

Let us study the convergence of the previously discretized model. Here, we will look at the phase velocity of the first 50 eigenvalues of the implicit Euler-Bernoulli beam. For a mode $i$ we denote by $c_{p,i}$ the corresponding phase velocity and $k_i$ the corresponding wavenumber. As such, we study the convergence of the relative error: 
$$\frac{1}{50}\sum_{i=1}^{50}\frac{|c_{p,i} - c_p^{th}(k_i)|}{c_p^{th}(k_i)}$$ 
with respect to the mesh step size $\d x$,
with $c_p^{th}(k) = k\sqrt{D}/\sqrt{\rho h (1 + (h^2/12)k^2} $.
The model parameter are chosen as~:
$D=5\times10^5,\rho=8\times10^3,h=6.28\times10^{-2}$. Figure~\ref{fig:eulerbernoulli-convergence-analysis} shows the convergence of the phase velocities with respect to the step size $\d x,$ in particular we obtain a convergence order of 1.106.
\begin{figure}
    \centering
    \includegraphics[width=0.5\linewidth]{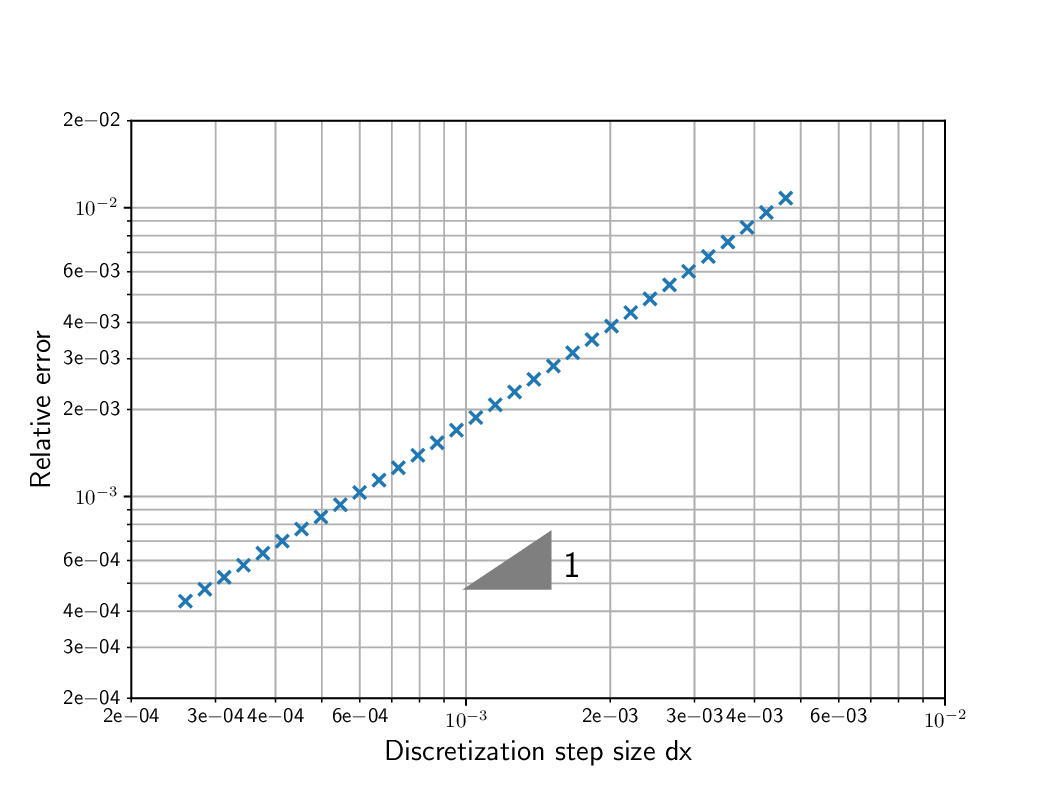}
    \caption{Convergence of the phase velocities for the implicit Euler-Bernoulli model}
    \label{fig:eulerbernoulli-convergence-analysis}
\end{figure}

\subsection{Incompressible Navier-Stokes equations}

\subsubsection{Dipole collision} Let us focus on the normal dipole collision with no-slip boundary conditions problem. Such a setting has been considered as a benchmark test case to study the properties of incompressible Navier-Stokes solvers~\cite{clercx2006normal}. The 2D domain is defined as $\Omega = [-1,1]^2$ and initial conditions for the vorticity as two monopoles~\cite{clercx2006normal}:
$$ \omega_0(x) = \omega_e \left( (1 - (r_1/r_0)^2 \exp( -(r_1/r_0)^2) -  (1 - (r_2/r_0)^2 \exp( -(r_2/r_0)^2)  \right), $$
with $r_1 = \left | x - c_1 \right |,r_1 = \left | x - c_2 \right |$ the  distances between $x$ and the center of each monopole $c_1,c_2 \in \Omega$, $r_0$ the radius of the monopoles and $\omega_e$ the extremum vorticity value.

Moreover, boundary conditions are chosen as no-slip and impermeable. As a consequence, the initial condition $\psi_0$ is chosen to be the solution of the Poisson equation $-\Delta \psi_0 = \omega_0$ with homogeneous Dirichlet boundary conditions to enforce the impermeability. As highlighted in \cite{clercx2006normal}, with such initial conditions, the tangential velocity corresponding to $\omega_0$ at the boundaries is very small, hence $\psi_0$ satisfies the no-slip boundary condition as well.

\paragraph{Normal dipole collision} Let us choose the initial condition parameters as $c_1 = (
    0 \quad 0.1
)^\top, c_2 = (
    0 \quad  - 0.1 )^\top, r_0 = 0.1$ and the extremum vorticity value is chosen $\omega_e = 300$ such that the initial kinetic energy is equal to two: $\mathcal{K}(\psi) = 2$. Such a value is slightly lower than the estimated $w_e=320$ value in~\cite{clercx2006normal} as in our case it fits the initial condition on the kinetic energy better.

The model parameters are $\rho_0=1$ and $\mu = 1/625$, the final time is chosen to be $T_f=2.5$s, timestep is $\d t= 1/600$~s and the total number of dofs of the model is 78 607.

Figure~\ref{fig:inse-screenshots} shows snapshots of the vorticity at times 0.4 s, 0.6 s and 1 s.
    Table~\ref{tab:normal-collision-table} lists the value of the enstrophy and kinetic energy at times 0.25 s, 0.5 s and 0.75 s.
Figure~\ref{fig:inse-kinetic-energy-power-balance} presents the time evolution of the power balance and kinetic energy. Figure~\ref{fig:inse-enstrophy} presents the time evolution of enstrophy along with the enstrophy balance. 

Figure~\ref{fig:inse-boundary-profile} presents the vorticity profile at the right boundary over the times 0.4~s, 0.6~s and 1.0~s. This shows how vorticity is being generated by the Lagrange multiplier $u_D^3 = \omega_{|\partial \Omega}$ to enforce the no-slip boundary condition.
Moreover, figure~\ref{fig:inse-boundary-profile}  presents a contour plot of the vorticty on the subdomain $[0.4,1]\times[0,0.6]$ at time $t=1$s showing how the upper part of the dipole interacts with the boundary after initial collision.

These results are compared to both \cite{clercx2006normal} and \cite{de2019inclusion}. Firstly, it is clear that the error on the enstrophy balance between each time step is close to machine precision whereas the power balance error is greater. However, in both cases, these functionals fit both references.
Additionally, both the contour plot and vorticity plot at the boundary are very close to reference plots as well.
\begin{table}[!ht]
    \centering
    \begin{tabular}{c|c|c}
     Time (s)& Kinetic energy (J)& Enstrophy \\
     \hline 
     \hline 
     0.25 & 1.50552 & 472.1750 \\
     0.5 & 1.01554 & 379.7911 \\
     0.75 & 0.76913 & 250.8609
\end{tabular}
    \caption{Kinetic energy and enstrophy over times $0.25$s, $0.5$s and $0.75$s.}
    \label{tab:normal-collision-table}
\end{table}

\begin{figure}
    \centering
    \includegraphics[width=0.32\textwidth]{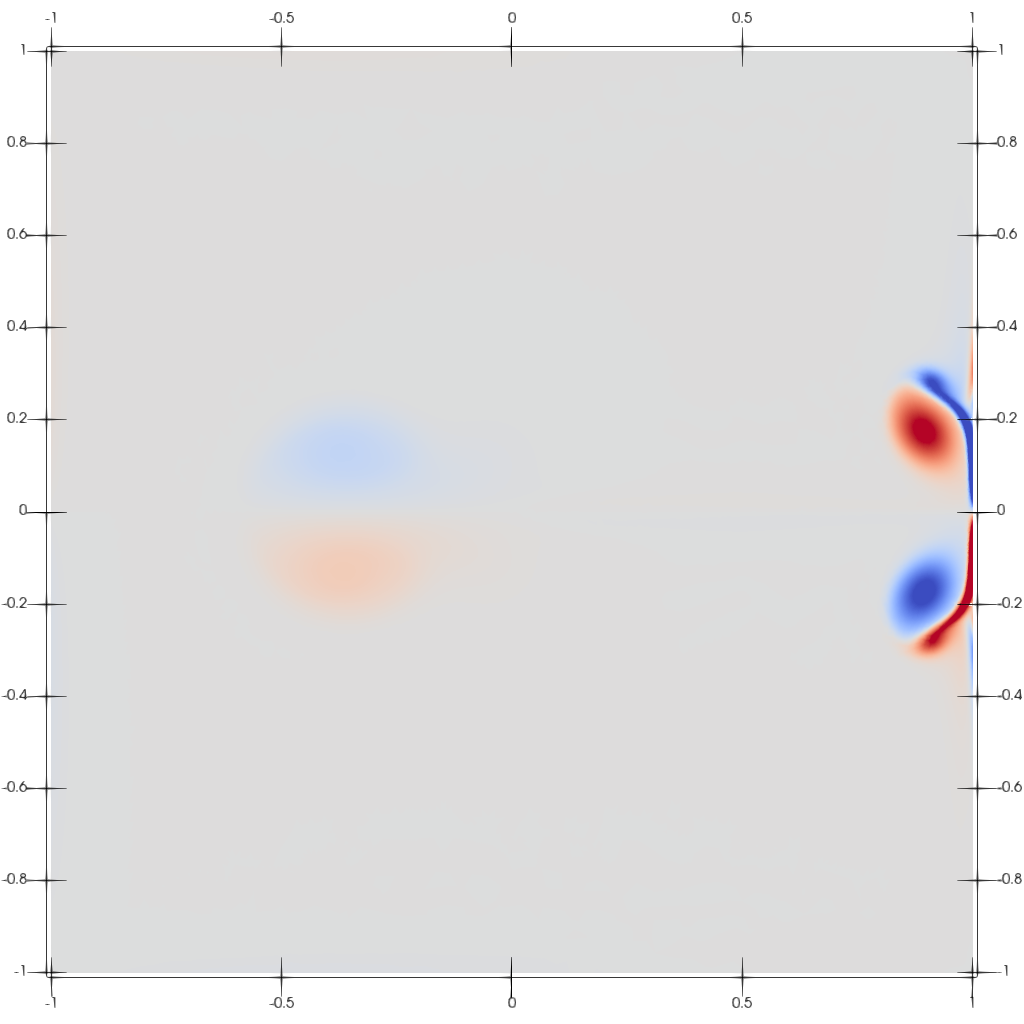}
    \includegraphics[width=0.32\textwidth]{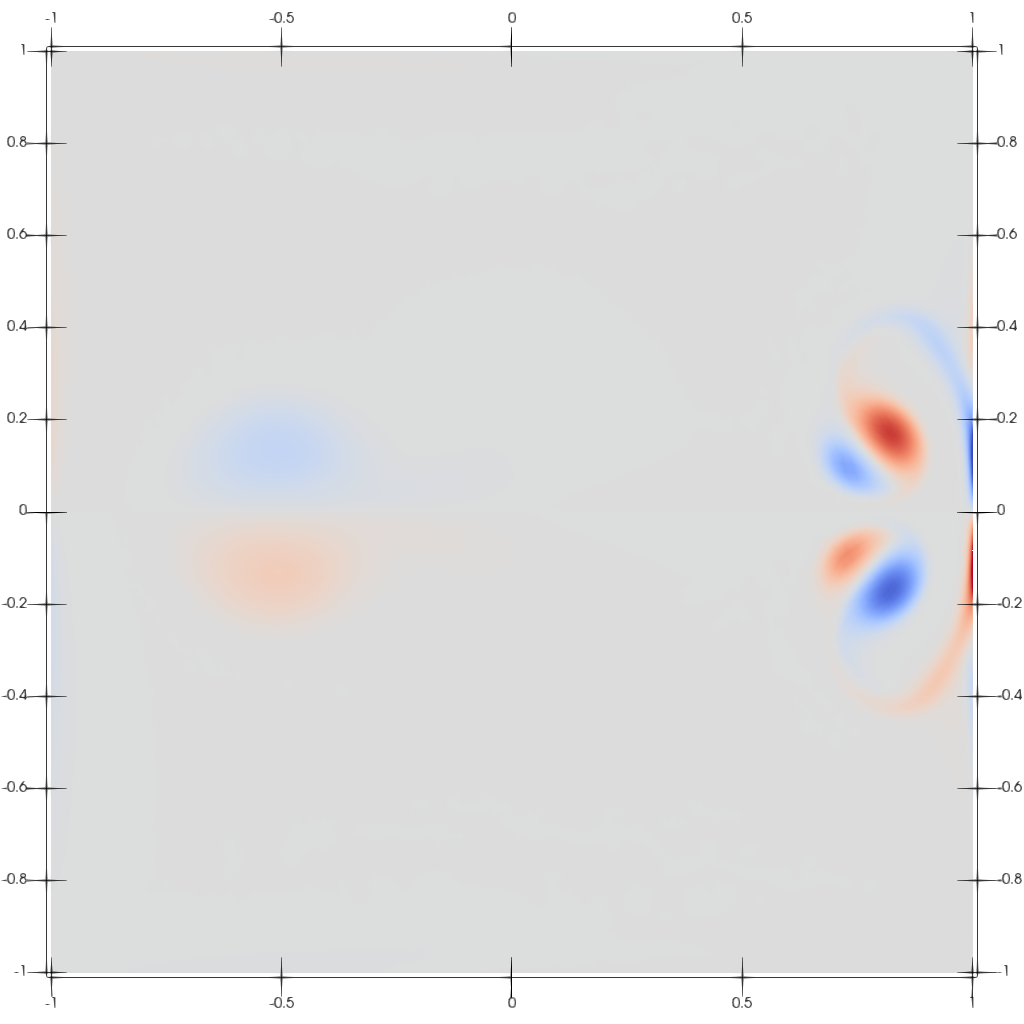}
    \includegraphics[width=0.32\textwidth]{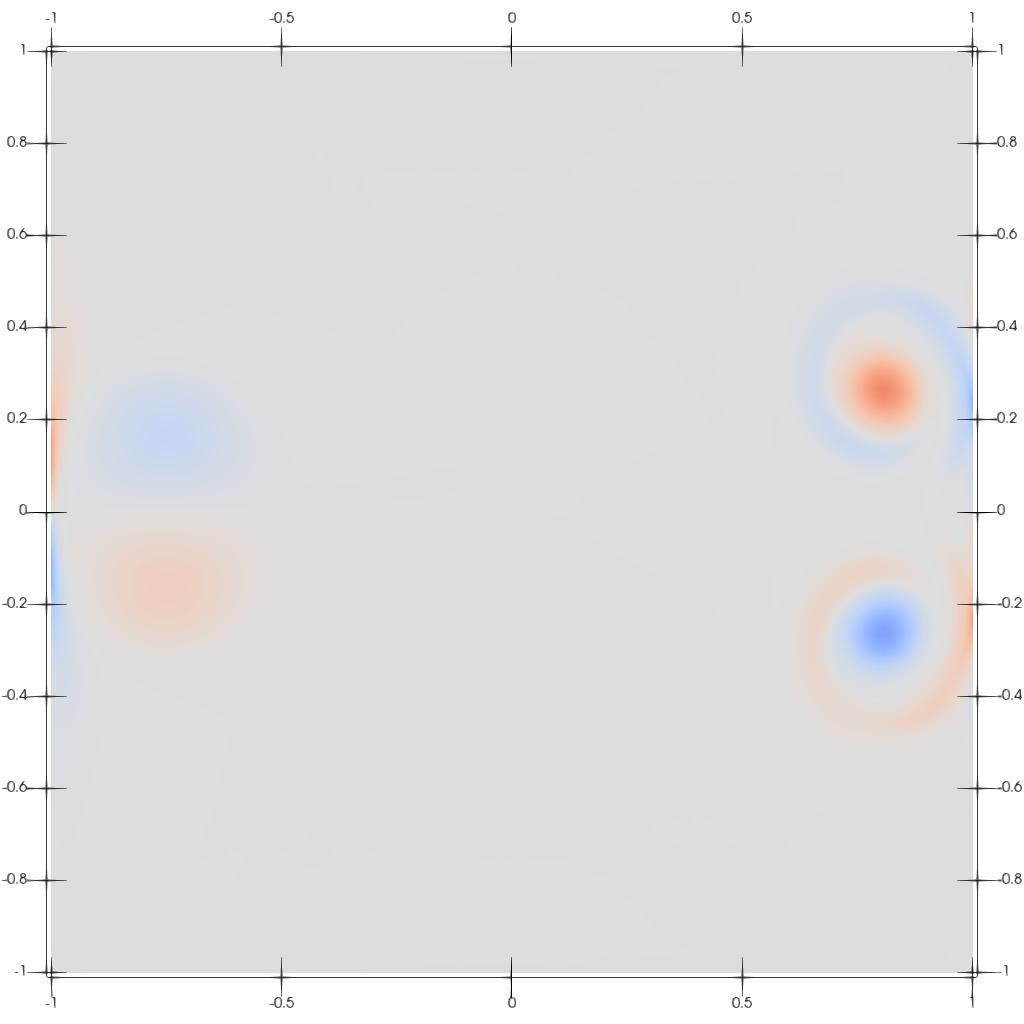}
    \caption{Vorticity at times 0.4 s, 0.6 s and 1.0 s.}
    \label{fig:inse-screenshots}
\end{figure}

\begin{figure}[ht]
    \centering
    \includegraphics[width=0.495\textwidth]{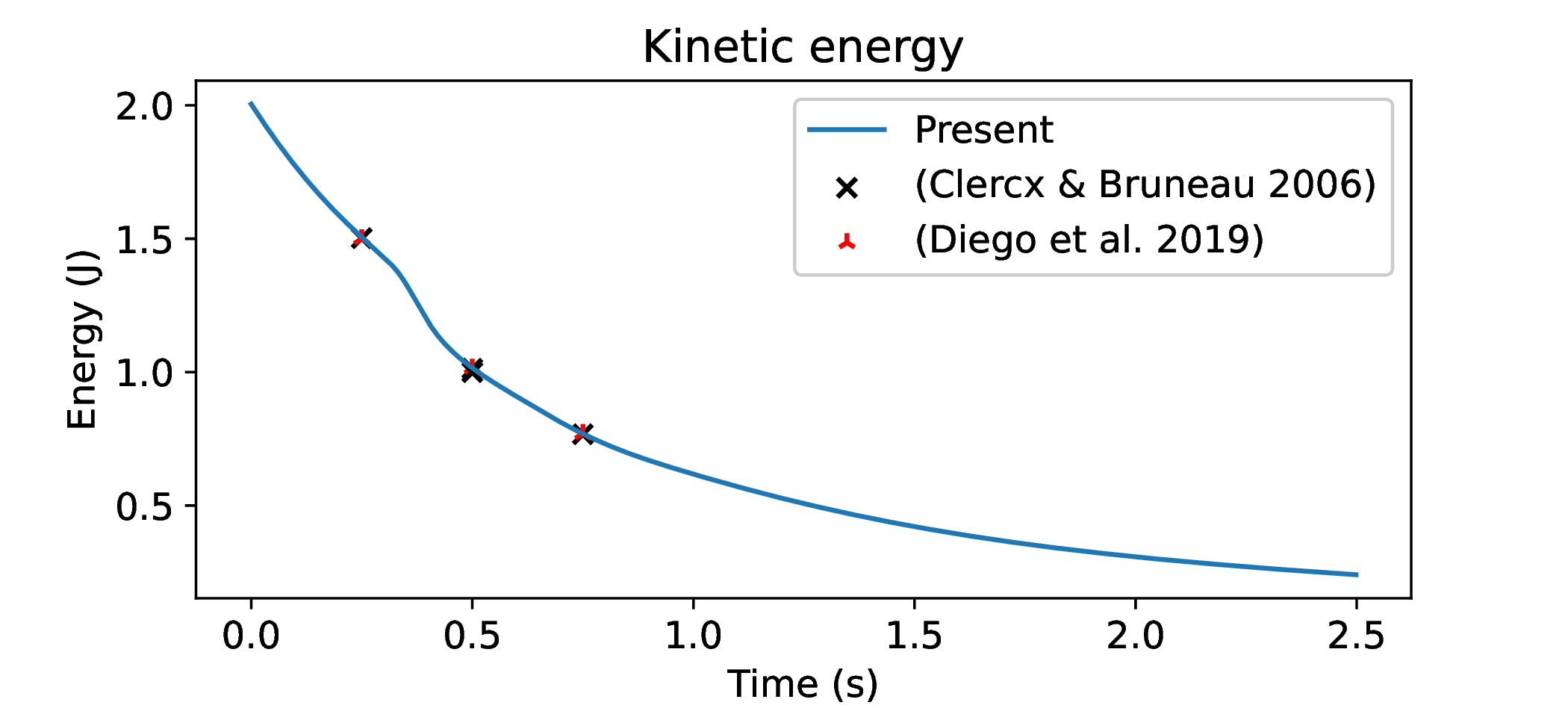}
    \includegraphics[width=0.495\textwidth]{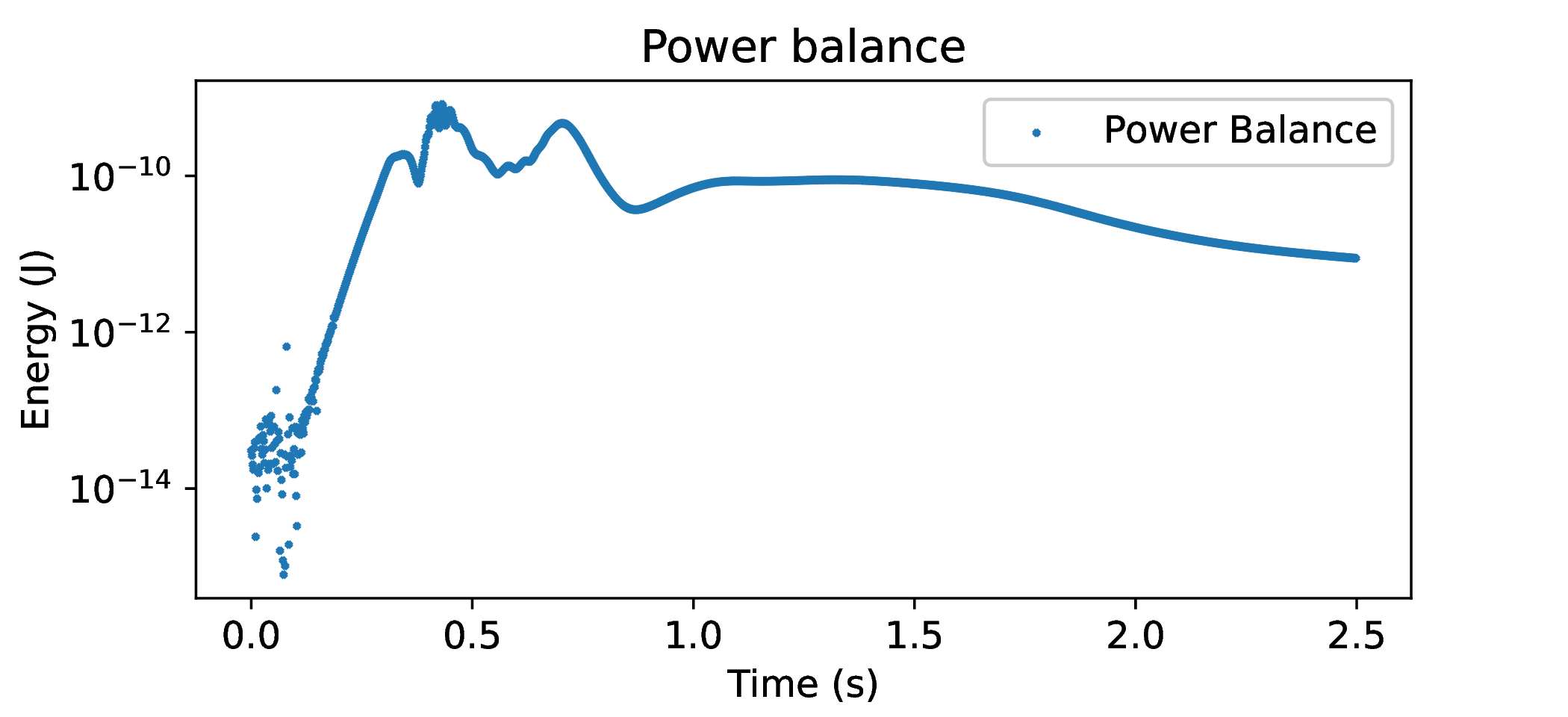}
    \caption{Time evolution of kinetic energy and power balance over time}
    \label{fig:inse-kinetic-energy-power-balance}
\end{figure}

\begin{figure}[ht]
    \centering
    \includegraphics[width=0.495\textwidth]{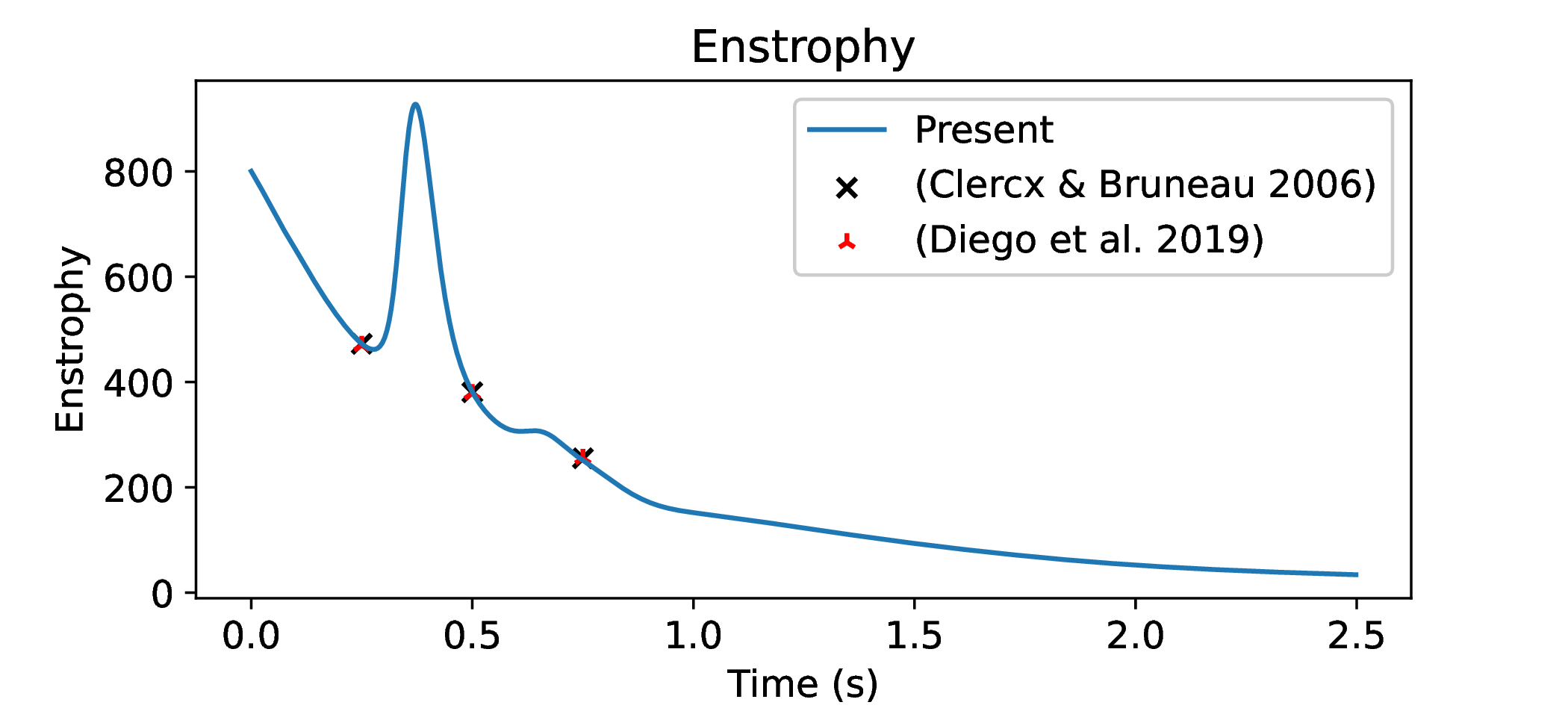}
    \includegraphics[width=0.495\textwidth]{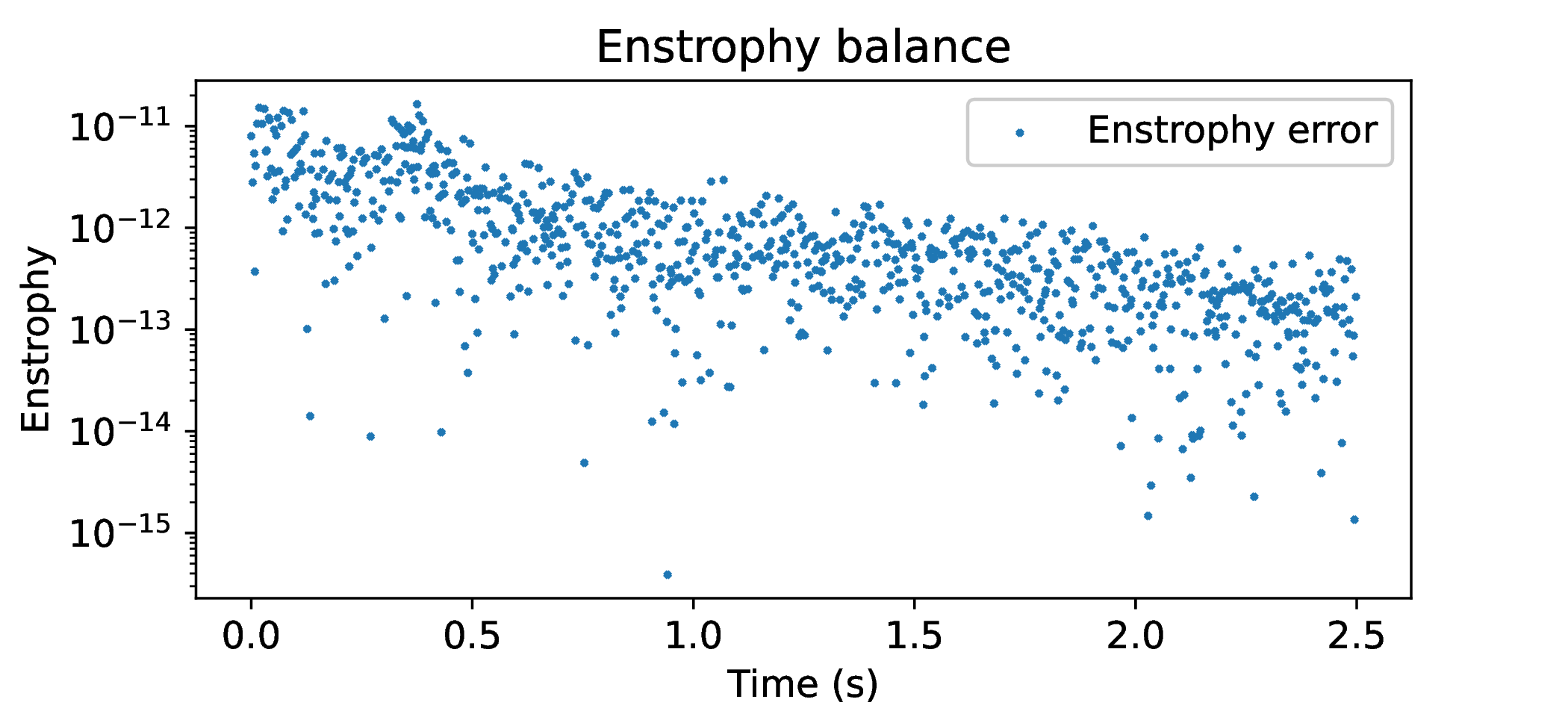}
    \caption{Time evolution of enstrophy and enstrophy balance over time}
    \label{fig:inse-enstrophy}
\end{figure}

\begin{figure}[ht]
    \centering
    \includegraphics[width=0.50\linewidth]{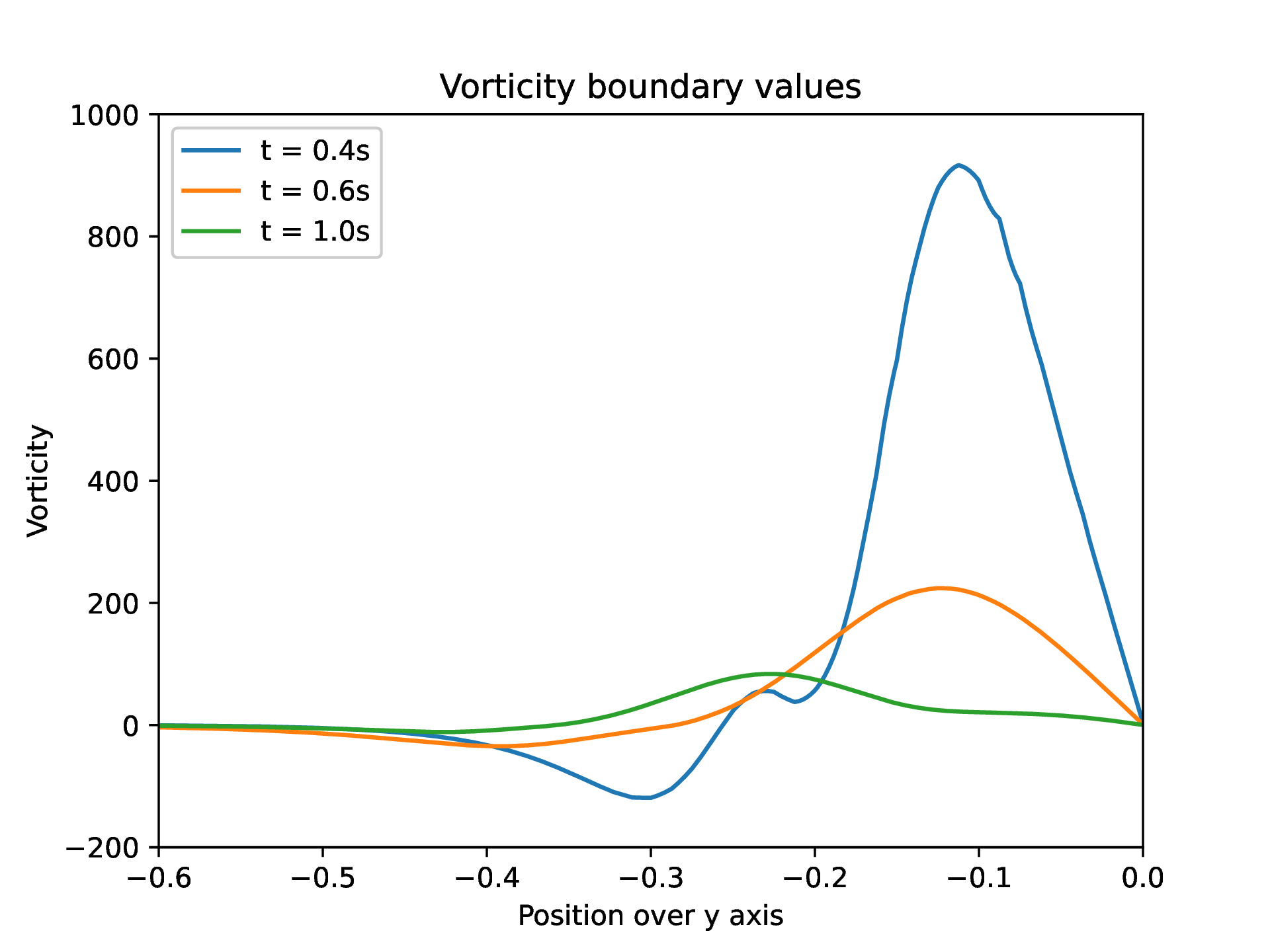} 
    \includegraphics[width=0.375\linewidth]{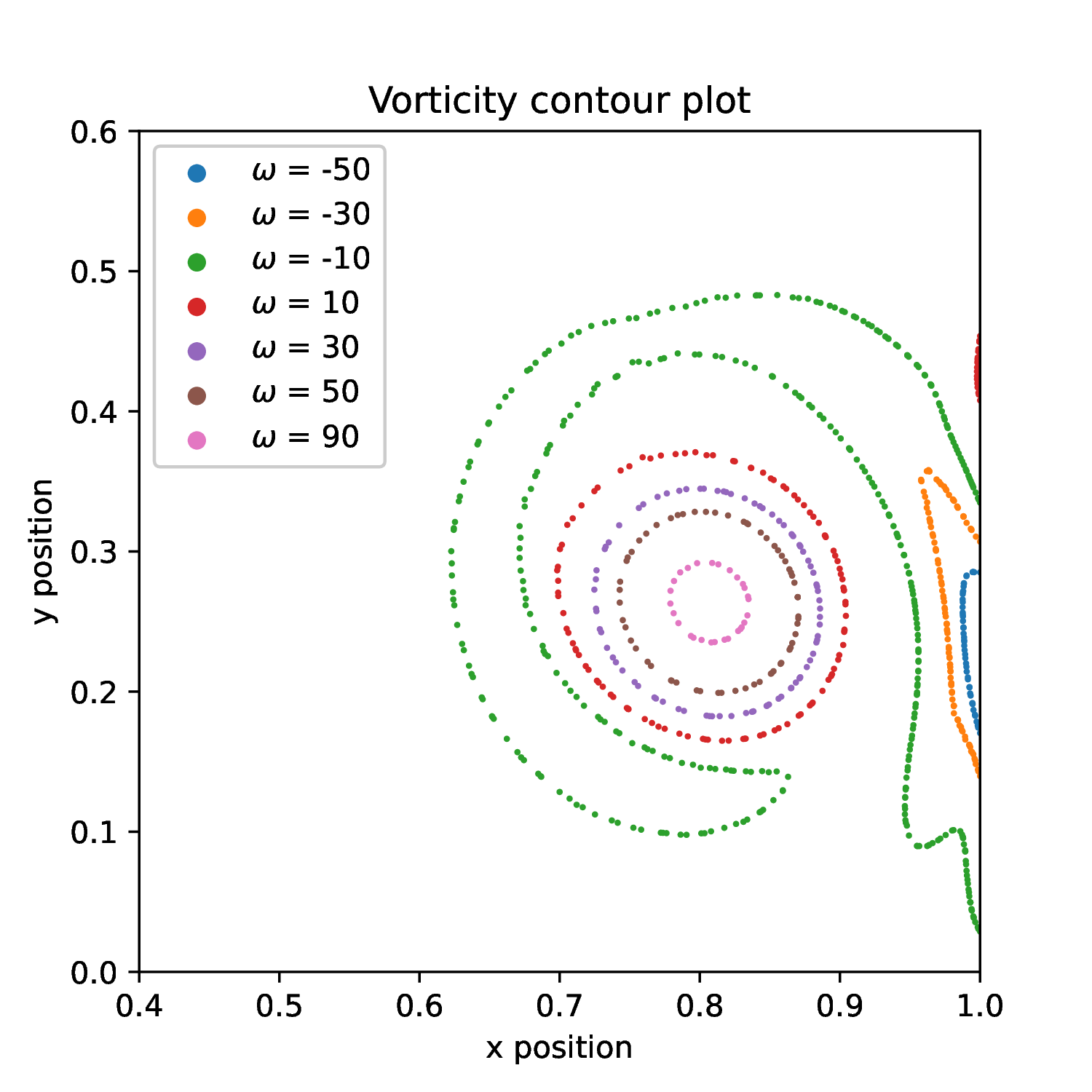} 
    \caption{Vorticity profile at the right boundary $x=1$ and along $y\in[-0.6,0]$ for the times $0.4$s, $0.6$s and $1.0$s (\emph{left}). Vorticity contour plot at $t=1$s in the subdomain $[0.4,1]\times[0,0.6]$ (\emph{right}).} 
    \label{fig:inse-boundary-profile}
\end{figure}

\subsubsection{Convergence analysis}
Let us study the convergence of the incompressible Navier-Stokes equations when refining the mesh. We first deal with Taylor-Green vortices, where the solution is mainly driven by the dissipative term and an analytical solution is known ; then we address the normal dipole collision,  where the convective term is no more negligible, and the finest mesh is used as a reference since no analytical solution is known.
    \paragraph{Taylor-Green vortices}
     First, let us look at an analytical solution, namely the Taylor-Green vortices case~\cite{de2019inclusion}:
    $$ \begin{aligned}
        \psi_{ref}(t,x,y) &= \frac{1}{\pi}\sin(\pi x) \sin(\pi y) \exp(- \frac{2\pi^2 t}{Re}), \\
        \omega_{ref}(t,x,y) &=- \Delta \psi_{ref}(t,x,y) =  2 \pi^2 \psi_{ref}(t,x,y), \\
    \end{aligned}$$
    for $(x,y)\in [0,1]^2, t\in\mathbb{R}^+,$
    with $Re = \rho_0/\mu,$ the Reynolds number~; in our case $\rho_0=1,$ and $\mu=1/100$. Moreover, the boundary conditions are homogeneous Dirichlet for both $\omega$ and $\psi$. The mesh is taken as a regular grid of simplexes (triangles) composed of $K$ rows and $K$ columns with $K\in\{5, 7, 9, 11, 13, 15, 17, 19, 21, 23, 25\}$. The time step is taken as $\d t = 1/1000 s$ and final time is $t=1s.$
    
    Figure~\ref{fig:taylor_green_convergence} shows the $H^1$ error $||\omega - \omega_{ref}||_{H^1}$ at time $t=1s$  and  $||\psi - \psi_{ref}||_{H^1}$ at time $t=0.9995s$ (since $\psi$ is evaluated on half time steps) respectively,  plotted against the average element size of the mesh (triangle circumradius).
    
    In such a case, we recover a convergence order of $2.565$ for $\omega$ and $4.760$ for $\psi$ with respect to the average element sizes.

    Note however that with such a choice of variables (vorticity/stream function) yields a solution mainly driven by the dissipative term and not the convective term. Indeed, the convective term is identically zero:

    $$ \diver(\omega_{ref} \, \grad^\perp(\psi_{ref})) = \grad(\omega_{ref}) \cdot \grad^\perp(\psi_{ref}) =2 \pi^2 \, \grad(\psi_{ref})\cdot\grad^\perp(\psi_{ref}) = 0.$$

    \begin{figure}[ht]
        \centering
        \includegraphics[width=0.47\linewidth]{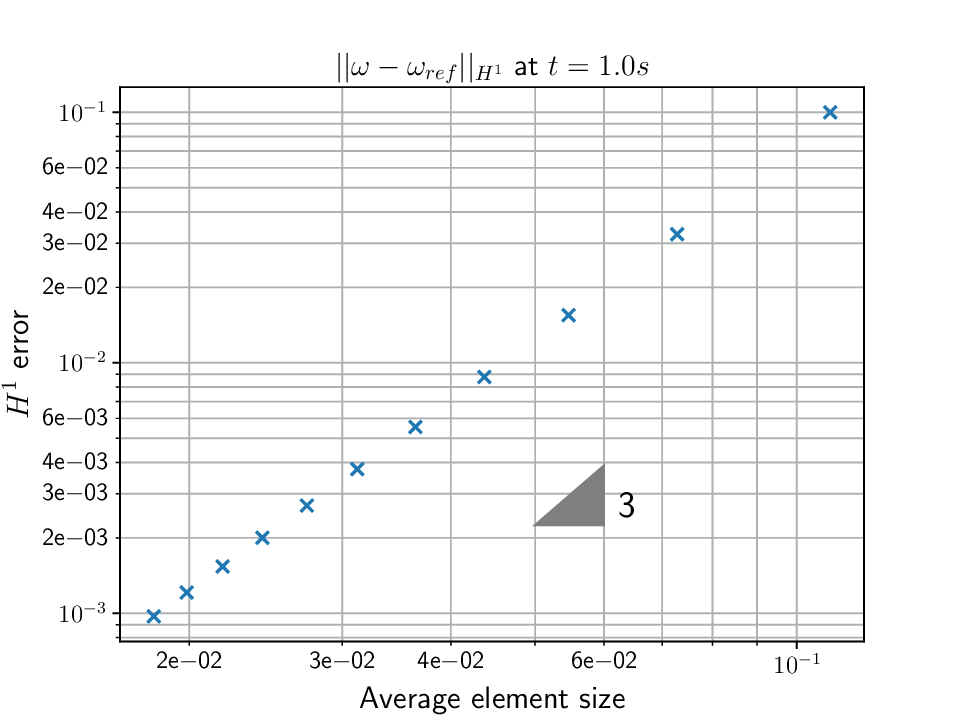}
        \includegraphics[width=0.47\linewidth]{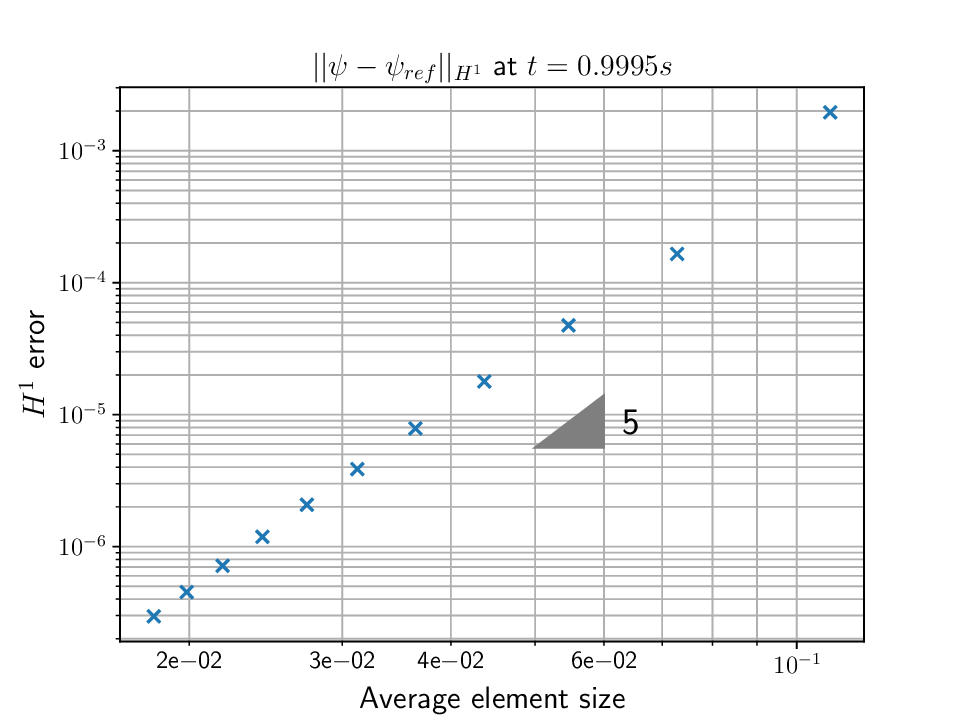}
        \caption{Vorticity and stream function $H^1$ error for the Taylor-Green vortices case.}
        \label{fig:taylor_green_convergence}
    \end{figure}
    
    \paragraph{Dipole collision}
    Let us now study the convergence of the normal dipole collision case when refining the mesh. Here, the convective term is not negligible. To do so, we ran the simulation with $\rho_0=1,\mu = 625, \d t=1/600$, a final time $t=1.5s$ and for 9 different meshes with a mesh size parameter $\ell$ ranging from $0.04$ to $0.12$ and an average triangle size ranging from $0.0055$ to $0.016$. The finest mesh ($\ell=0.04$) is used as a reference. The error is then defined as the average of the relative $H^1$ errors for  180 evenly distributed timesteps between $0$ and $1.5s$:

    $$ \varepsilon^\psi_i = \frac{1}{180}\sum_{k=1}^{180} \frac{|| \psi_i(t_k) - \psi_{ref}(t_k)||_{H^1}}{||\psi_{ref}(t_k)||_{H^1}}. $$
    $$ \varepsilon^\omega_i = \frac{1}{180}\sum_{k=1}^{180} \frac{|| \omega_i(t_k) - \omega_{ref}(t_k)||_{H^1}}{||\omega_{ref}(t_k)||_{H^1}}. $$

    We obtain a convergence order of 2.234 and 3.910 for $\omega$ and $\psi$ respectively.
    Figure~\ref{fig:dipole_convergence} shows the error for both $\psi$ and $\omega$ plotted against the average element size.

\begin{figure}
    \centering
    \includegraphics[width=0.49\linewidth]{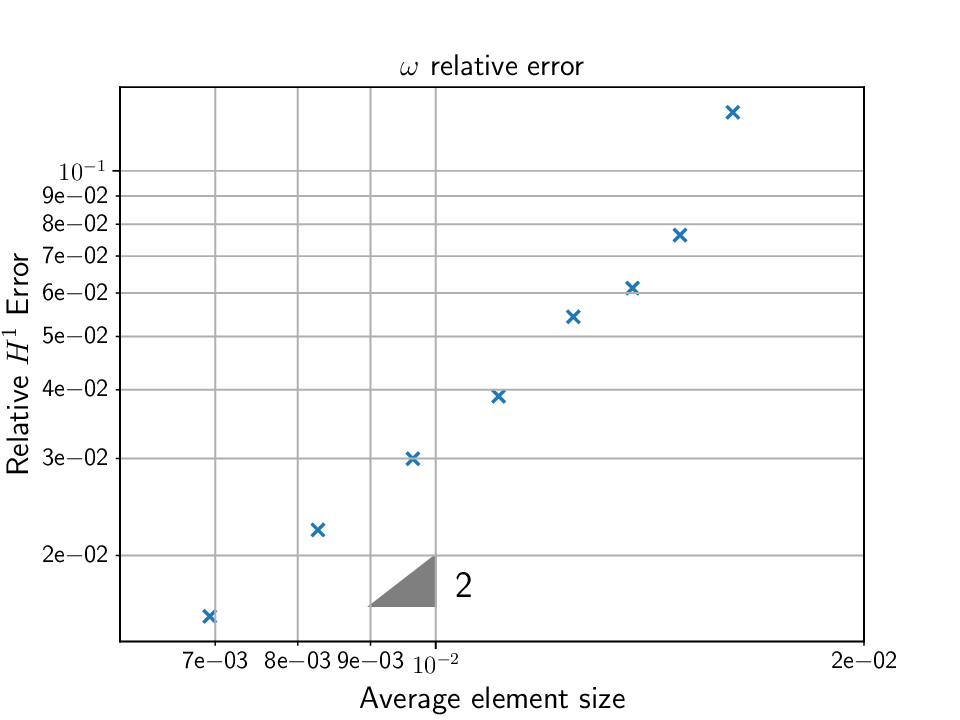}
    \includegraphics[width=0.49\linewidth]{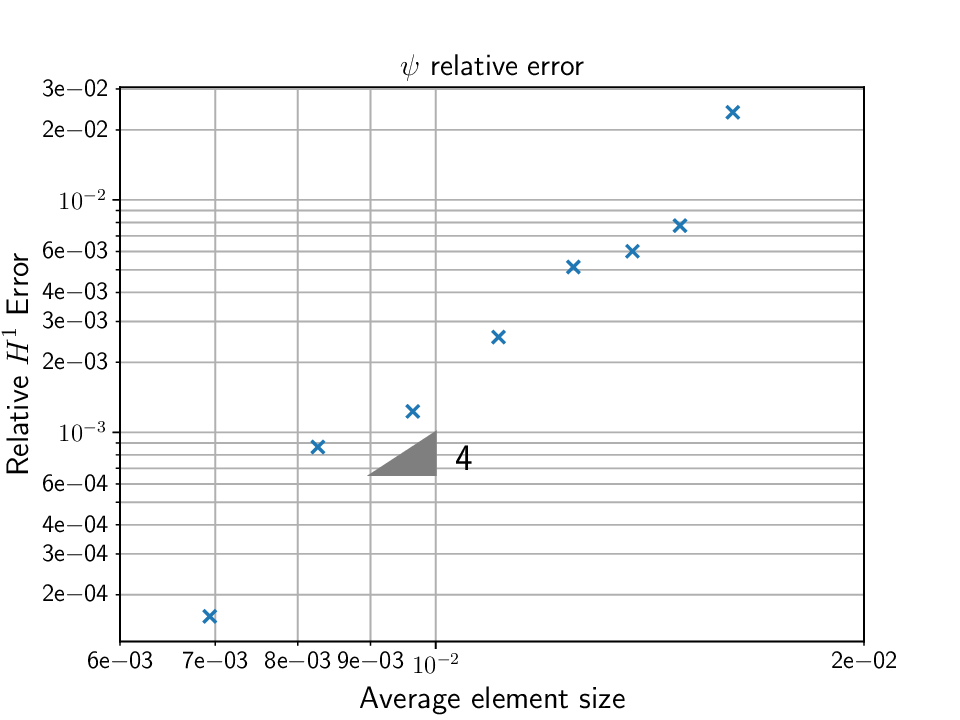}
    \caption{Vorticity and stream function $H^1$ relative error for the dipole collision case.}
    \label{fig:dipole_convergence}
\end{figure}

\section{Conclusion}

In this work, we have proposed a structure-preserving spatial discretization method for pH systems with implicit constitutive relations, formulated through Stokes-Lagrange structures. Building upon the Partitioned Finite Element Method (PFEM), which has been successfully employed for classical pH systems, we have extended this framework to accommodate systems including differential and nonlocal constitutive relations. This extension ensures that the essential geometric properties of the continuous system are faithfully retained in the finite-dimensional approximation.

The methodology has been applied to three representative cases: a 1D nanorod model capturing nonlocal effects, an implicit Euler-Bernoulli beam formulation relevant for high-frequency dynamics, and the 2D incompressible Navier-Stokes equations expressed in vorticity--stream function variables. For each of these examples, we have demonstrated that the proposed discretization not only preserves the power balances at the discrete level but also captures additional physical invariants such as enstrophy in the fluid flow case. The numerical results obtained confirm the consistency, robustness, and accuracy of the structure-preserving discretization strategy, highlighting its potential for reliable long-term simulations.

The work presented here opens several perspectives for future research. One important direction concerns the mathematical analysis of the proposed schemes, particularly regarding their convergence properties and stability guarantees when applied to systems with complex boundary conditions. Another natural extension lies in the treatment of other nonlinear port-Hamiltonian systems, especially those arising in multiphysics contexts where strong couplings and nonlinear constitutive relations are prevalent. The method could also be adapted to large-scale 3D applications, offering a powerful structure-preserving tool for the simulation of coupled thermo-mechanical, electromagnetic, or fluid-structure interaction problems.

Additionally, the discretization choices made in this work could be further refined by investigating optimal finite element spaces that enhance numerical accuracy while reducing computational complexity, without compromising the preservation of the underlying structure. Finally, extending the approach to space-time discretizations that preserve geometric properties simultaneously in space and time, as done in this work for the fluid case, would provide a unified framework capable of addressing both the spatial and temporal aspects of structure preservation.

Altogether, this study contributes to the ongoing development of structure-preserving numerical methods for distributed pH systems and lays the groundwork for future advances in the reliable and physically consistent simulation of complex dynamical systems.

\section*{Acknowledgements}

This work was supported by the IMPACTS project funded by the French National Research Agency (project ANR-21-CE48-0018), \url{https://impacts.ens2m.fr/}.

\appendix

\section{Proofs} \label{apx:proofs}

\subsection{Proof of Lemma~\ref{def:latent-state-ph}} \label{proof:latent-state-ph}

\begin{proof}
    Recall that $P$ and $S \in \mathcal{M}_{n+n_L}(\mathbb{R})$ are such that $P^\top S$ is symmetric and $\begin{bmatrix}
        P \\ S
    \end{bmatrix}$ is of full rank, and from~\eqref{eqn:discr-linear-phs:LAGRANGE}, that $\alpha, e \in \mathbb{R}^n$ and $u_L,y_L \in \mathbb{R}^{n_L}$ are such that:
    $$P^\top \begin{pmatrix}
        e \\ y_L
    \end{pmatrix}= S^\top\begin{pmatrix}
        \alpha \\ u_L
    \end{pmatrix}.$$
    It follows that $\begin{pmatrix}
        \alpha \\ u_L \\ e \\ y_L 
    \end{pmatrix} \in \ker( \begin{bmatrix}
        S^\top & - P^\top
    \end{bmatrix}.$
    
    Moreover, from $P^\top S - S^\top P=0$, it comes: $\text{Ran}\begin{bmatrix}
        P \\ S
    \end{bmatrix} \subset \ker( \begin{bmatrix}
        S^\top & - P^\top
    \end{bmatrix}).$
    
    Furthermore, we have from the size of $P$ and $S$: $\text{dim}(\ker( \begin{bmatrix}
        S^\top & - P^\top
    \end{bmatrix}) \leq n + n_L$, and the rank condition gives $\text{dim}(\text{Ran}\begin{bmatrix}
        P \\ S
    \end{bmatrix}) = n + n_L$ . Hence:
    \begin{equation} \label{eqn:ker-ran-equality-discrete}
        \ker(\begin{bmatrix}
        S^\top & - P^\top
    \end{bmatrix}) = \text{Ran}\begin{bmatrix}
        P \\ S
    \end{bmatrix},
    \end{equation} 
    which proves the existence of $(\lambda, \widetilde u_L ) \in \mathbb{R}^{n + n_L}$ such that \eqref{eqn:latent-state-control-def} holds.
    
    Finally, the rank condition ensures the injectivity of $\begin{bmatrix}
        P \\S
    \end{bmatrix}$, which proves the uniqueness.
\end{proof}
\subsection{Proof of Lemma~\ref{lemma:latent-state-hamiltonian-definition} } \label{proof:latent-state-hamiltonian-definition}
\begin{proof}
    Let us compute $\frac{\d}{\d t} \mathcal H(\lambda, \widetilde u)$:
    \begin{equation*}
        \begin{aligned}
            \frac{\d}{\d t} \mathcal H(\lambda, \widetilde u) & = \frac{\d}{\d t}\left(\frac{1}{2} \begin{pmatrix}
        \lambda \\ \widetilde u_L
    \end{pmatrix}^\top P^\top S \begin{pmatrix}
        \lambda \\ \widetilde u_L
    \end{pmatrix}\right), \\ & =  \left(\frac{\d }{\d t}\begin{pmatrix}
        \lambda \\ \widetilde u_L
    \end{pmatrix} ^\top P^\top \right) S \begin{pmatrix}
        \lambda \\ \widetilde u_L
    \end{pmatrix}, & (P^\top S = S^\top P) \\
    & = \frac{\d}{\d t} \begin{pmatrix}
        \alpha \\ u_L
    \end{pmatrix}^\top \begin{pmatrix}
        e \\y_L
    \end{pmatrix}, & \text{(Lemma~\ref{def:latent-state-ph})} \\
    & = e^\top \dot \alpha + y_L^\top \dot u_L. 
        \end{aligned} 
    \end{equation*}
    Hence, using solely the Lagrange structure, and in particular Eq.~\eqref{def:latent-state-ph}, leads to the first equality of \eqref{eqn:latent-state-ph-power-balance}.
    Now, using the Dirac structure, we have:
    \begin{equation*}
        \begin{aligned}
            0 & = \begin{pmatrix}
                e \\ e_R \\ u_D
            \end{pmatrix}^\top J \begin{pmatrix}
                e \\e_R \\ u_D
            \end{pmatrix}, & (J = - J^\top) \\
            &= \begin{pmatrix}
                e \\e_R \\ u_D
            \end{pmatrix}^\top \begin{pmatrix}
                \frac{\d}{\d t} (P_{1,1} \lambda + P_{1,2} \widetilde u_L) \\ f_R \\ -y_D
            \end{pmatrix}, \\
            &= e^\top \dot \alpha + e_R^\top f_R - y_D^\top u_D, \\
            &= \frac{\d}{\d t} \mathcal H(\lambda,\widetilde u_L) - y_L^\top \dot u_L + e_R^\top f_R - y_D^\top u_D.
        \end{aligned}
    \end{equation*}
    Rearranging the terms gives the desired result.
\end{proof}

\subsection{Proof of Lemma~\ref{lemma:nanorod-hamiltonian-definition}}\label{proof:nanorod-hamiltonian-definition}

\begin{proof}
    See \cite[Example 31]{maschke2023linear} for~\eqref{eqn:hamiltonian-nonlocal-nanorod}.
    
    Let us compute $\frac{\d}{\d t}\mathcal H(\lambda,p):$
    \begin{equation*}
        \begin{aligned}
            \frac{\d}{\d t} \mathcal H(\lambda,p) &= \frac{\d}{\d t} \left(\frac{1}{2} \int_a^b E\, (\lambda^2 + \ell^2 \, (\partial_x \lambda)^2) + \frac{1}{\rho}p^2 \, \,  \d x \right), \\
            &= \int_a^b E \, \lambda  \, \, (1 - \ell^2 \, \partial_{x^2}^2) \partial_t \lambda + \frac{p}{\rho} \, \partial_t p \, \d x  + [ \partial_t \lambda \, \,   E \ell^2 \,  \partial_x \lambda ]_a^b \,, \\
            &= \int_a^b \sigma \,  \partial_x v + v \, \partial_x \sigma \, \d x+ [ \partial_t \lambda \, \,   E \ell^2 \,  \partial_x \lambda ]_a^b \, , \\
            &= [ \sigma \, v ]_a^b + [ \partial_t \lambda \, \,   E \ell^2 \,  \partial_x \lambda ]_a^b.
        \end{aligned}
    \end{equation*}
    This gives us the first part of the result. Let us now consider the Robin boundary conditions~\eqref{eqn:compatible-boundary-conditions}. In particular, using $\sigma = E\, \lambda$, we get:
    $$ \begin{array}{rcl} \sigma(a) + n_a \, \ell \partial_x \sigma(a) = 0, &\qquad& \sigma(b) + n_b \,\ell \partial_x \sigma(b) = 0, \\ 
    \lambda(a) + n_a\, \ell \partial_x \lambda(a) = 0, &\qquad& \lambda(b) + n_b\, \ell \partial_x \lambda(b) = 0.\end{array}   $$
    Then:
    \begin{equation*}
        \begin{aligned}
             \frac{\d}{\d t} \mathcal H(\lambda,p) &= [ \sigma \, v ]_a^b + [ \partial_t \lambda \, \,   E \ell^2 \,  \partial_x \lambda ]_a^b, \\
             &= \sigma(a) \,v(a)\, n_a + \sigma(b)  \, v(b)\, n_b  + \partial_t \lambda(a) \, E \ell^2 \, \partial_x \lambda(a) \, n_a +  \partial_t \lambda(b) \, E \ell^2 \, \partial_x \lambda(b) \, n_b,\\
             &= - \ell \,  \partial_x\sigma(a) \, v(a) \, n_a^2 - \ell \,  \partial_x \sigma(b) \, v(b) \, n_b^2 
              - \ell \,  \partial_t \lambda(a) \, E \, \lambda(a) \, n_a^2 -  \ell \,  \partial_t \lambda(b) \, E \, \lambda(b) \, n_b^2. \\
        \end{aligned}
    \end{equation*}
    Now, using the Dirac structure~\eqref{eqn:nanorod-equation-implicit-latent:DIRAC} which imposes $\partial_x \sigma = \partial_t p$, and the Lagrange structure~\eqref{eqn:nanorod-equation-implicit-latent:LAGRANGE} which implies $v = p/\rho$, we obtain:
    \begin{equation*}
        \begin{aligned}
        \frac{\d}{\d t} \mathcal H(\lambda,p) &= - \ell \,  \partial_x\sigma(a) \, v(a) \,  - \ell \,  \partial_x \sigma(b) \, v(b) \,  - \ell \,  \partial_t \lambda(a) \, E \, \lambda(a) \, -  \ell \,  \partial_t \lambda(b) \, E \, \lambda(b) \, , \\
             &= -\ell \left[ \frac{1}{\rho} \left(\partial_t p(a)\, p(a) +   \partial_t p(b) \, p(b) \right) + E\, \left( \partial_t \lambda(a)\lambda(a) + \partial_t \lambda(b) \lambda(b)\right) \right].
        \end{aligned}
    \end{equation*}
\end{proof}

\subsection{Proof of Lemma~\ref{lemma:nanorod-hamiltonian-robin-definition}}\label{proof:nanorod-hamiltonian-robin-definition}

\begin{proof}
    Shifting the right-hand side of the power balance \eqref{eqn:nanorod-power-balance-with-robin-boundary-conditions} to the left yields:
    \begin{equation*} 
        \frac{\d}{\d t} \mathcal H(\lambda, p) +\ell \left[ \frac{1}{\rho} \left(\partial_t p(a)\, p(a) +   \partial_t p(b) \, p(b) \right) + E\, \left( \partial_t \lambda(a)\lambda(a) + \partial_t \lambda(b) \lambda(b)\right) \right] = 0.
    \end{equation*}
    Then, using the identity $ f'f = \frac{1}{2}(f^2)'$ gives us:
     \begin{equation*} 
       0 = \frac{\d}{\d t} \mathcal H(\lambda, p) + \frac{\ell}{2} \, \frac{\d}{\d t} \left[ \frac{1}{\rho}( p(a)^2 +   p(b)^2 ) + E\, (\lambda(a)^2 + \lambda(b)^2) \right] = \frac{\d}{\d t} \mathcal H_{\text{Rob}}(\lambda,p),
    \end{equation*}
    which proves the result.
\end{proof}

\subsection{Proof of Lemma~\ref{lemma:nanorod-discrete-robin}}\label{proof:nanorod-discrete-robin}

\begin{proof}
    Writing the Robin Boundary conditions \eqref{eqn:compatible-boundary-conditions} using the energy control port $u_L$ gives us:
    $$\sigma^d(a) + \ell\, (u_L)_1 = 0 =\sigma^d(b) + \ell\, (u_L)_2.$$
    Which, written using the previously defined matrices gives \eqref{eqn:discr-robin-boundary-conditions-lagrange}.

    Similarly, writing \eqref{eqn:compatible-boundary-conditions} using the power control port $u_D$ yields:
    $$ (u_D)_1  + \ell \,  \partial_x \sigma(a) =  0 = (u_D)_2 + \ell \,  \partial_x \sigma(b), $$
    where using the fact that $(u_D)_1 = \sigma(a)n_a$ and $(u_D)_2 = \sigma(b)n_b$ allowed us to remove the outer normals $n_a,n_b$ from the expression.
 Now, using Dirac structure \eqref{eqn:nanorod-equation-implicit-coenergy}, we have that $\partial_x \sigma = \partial_t p = \rho \, \partial_t v,$ which gives us:
    $$ (u_D)_1  + \ell \, \rho(a) \, \partial_t v^d(a) =  0 = (u_D)_2 + \ell \,  \rho(b) \, \partial_t v^d(b). $$
Writing the obtained equation using the previously defined matrices gives \eqref{eqn:discr-robin-boundary-conditions-dirac}.
\end{proof} 

\subsection{Proof of Lemma~\ref{lemma:navier-stokes-skew-symmetry}} \label{proof:navier-stokes-skew-symmetry}

\begin{proof} 
    Let $\psi, \omega, \phi_1,\phi_2 \in H^1$ then we have:
    \begin{equation*}
        \begin{aligned}
            \int_\Omega \phi_1 \, \diver(\omega \, \grad^\perp(\phi_2)) \, \d x& = - \int_\Omega \omega \, \grad(\phi_1) \cdot \grad^\perp(\phi_2) \, \d x + \int_{\partial \Omega} \omega \, \phi_1 \,  \bm{n} \cdot \grad^\perp(\phi_2) \, \d s, \\
            &= \int_\Omega \omega \, \grad^\perp(\phi_1) \cdot \grad(\phi_2) \, \d x + \int_{\partial \Omega} \omega \, \phi_1 \,  \bm{n} \cdot \grad^\perp(\phi_2) \, \d s,\\
            &= - \int_\Omega \phi_2 \, \diver(\omega \, \grad^\perp(\phi_1)) \, \d x + \int_{\partial \Omega} \omega \, \phi_1 \,  \bm{n} \cdot \grad^\perp(\phi_2) \, \d s \\ & \quad  \quad + \int_{\partial \Omega} \omega \, \phi_2 \,  \bm{n} \cdot \grad^\perp(\phi_1) \, \d s.
        \end{aligned}
    \end{equation*}
Which proves the first equality, then:
    \begin{equation*}
        \begin{aligned}
            \int_\Omega \phi_1 \, \diver(\grad^\perp(\psi) \, \, \phi_2) \, \d x& = - \int_{\Omega} \phi_2 \, \grad(\phi_1) \cdot \grad^\perp(\psi) \, \d x + \int_{\partial \Omega} \phi_1 \, \phi_2 \,  \bm{n} \cdot \grad^\perp(\psi) \, \d s, \\
            &= - \int_\Omega \phi_2 \, \diver( \grad^\perp(\psi) \, \phi_1 )\, \d x+ \int_{\partial \Omega} \phi_1 \, \phi_2 \,  \bm{n} \cdot \grad^\perp(\psi) \, \d s.
        \end{aligned}
    \end{equation*}
Where the last equality is obtained from the fact that $\diver(\grad^\perp(\psi)) = 0.$
\end{proof}

\section{Useful identities} \label{apx:vector-identities}

\begin{lemma}
Given $\bm{\phi} \in H^\curl(\Omega)$ and $\psi \in H^1(\Omega)$, we have:
\begin{equation} \label{ipp:gradperp-curl}
	\int_\Omega  \bm{\phi} \cdot \grad^\perp(\psi) \, \d x = \int_\Omega \curldeuxD(\bm{\phi}) \,  \psi \, \d x + \int_{\partial \Omega} \psi \, \, \bm{\phi} \cdot \bm{t} \,\d s.
\end{equation}
\begin{proof}
    See e.g. \cite[Chapter 11]{BofBreFor15}.
\end{proof}    
\end{lemma}

\begin{lemma}
\begin{equation} \label{eqn:gradperpcurldeuxd-laplacian}
    \curldeuxD \,\grad^\perp = - \Delta.
\end{equation}
\end{lemma}
\begin{proof}
    $\curldeuxD \,\grad^\perp = \begin{bmatrix}
        - \partial_y & \partial_x
    \end{bmatrix} \begin{bmatrix}
         \partial_y \\ -\partial_x
    \end{bmatrix} = - \partial_{y^2}^2 - \partial_{x^2}^2 = - \Delta.$
\end{proof}

\section{Nanorod equation}

\subsection{Condition number} \label{apx:nanorod-condition-number}

Figure \ref{fig:nanorod-condition-number} presents the condition number of $\bm{M} + \ell^2 \, \bm{K} + \ell \, \bm{B}\bm{B}^\top$ as a function of the parameter $\ell$, in particular for the number of discretization $N=100$, the condition number worsens for $\ell$  bigger than $10^{-2}$.
\begin{figure}[ht]
    \centering
    \includegraphics[width=0.55\linewidth]{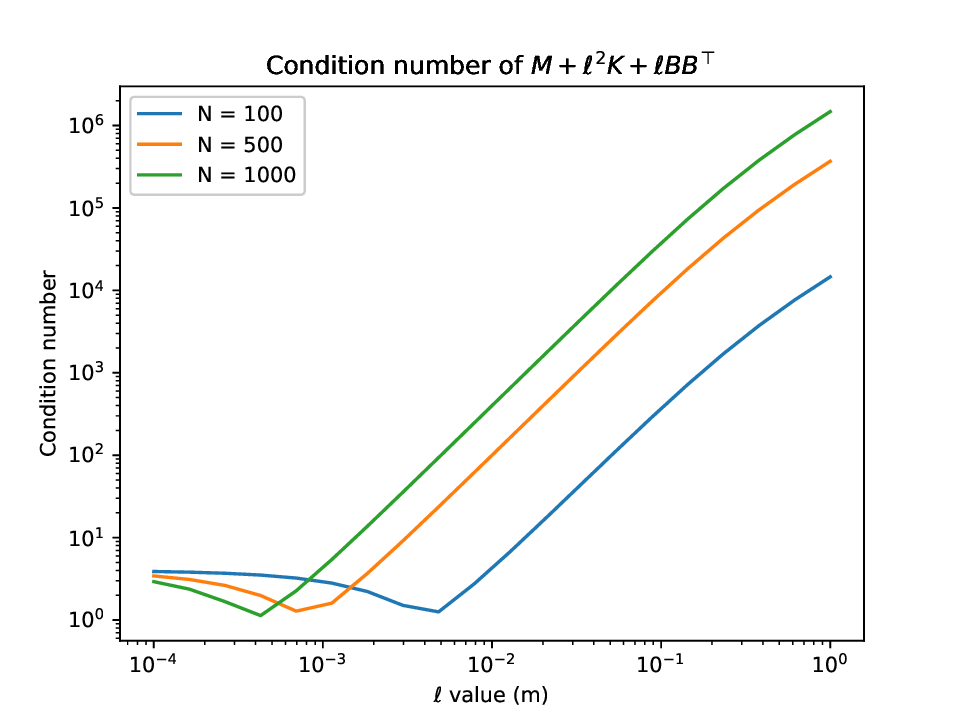}
    \caption{Condition number of matrix $\bm{M} + \ell^2 \, \bm{K} + \ell \, \bm{B}\bm{B}^\top$ for number of discretization points $N=$100, 500 and 1000.}
    \label{fig:nanorod-condition-number}
\end{figure}

\section{Implicit Euler-Bernoulli beam}
\label{apx:EB}

Let us derive the model following~\cite[Chapter~2]{lagnese1989boundary}, where thin plate models are tackled. We neglect the $y$ component to model the beam case of this work.

Let us denote by $U(x,z,t)$ the displacement at a time $t$ of the particle that occupies the position $(x,z)$ when the beam is at equilibrium. Furthermore, we denote $u(x,t) = U(x,0,t)$, the displacement of the middle line. 

Assume that the strain is proportional to the displacement (i.e., a linear strain displacement relation), and that the linear filaments of the beam initially perpendicular to the middle line remain straight and perpendicular to the deformed middle line, and undeformable: neither contraction nor extension (transverse shear effects are neglected).

The second assumption imposes a nonlinear relationship between $U(x,z,t)$ and $u(x,t)$. If this relation is linearized, we obtain: $U = u - z \partial_x u_z$, where $u_z$ denotes the $z$ component of the vector field $u$.

Denoting $w:=u_z$, the vertical displacement, the \emph{strain energy} $\mathcal{P}$ of the beam is defined by:
$$
\mathcal{P}(w):=\frac{D}{2} \int_a^b \left( \frac{\partial^2}{\partial x^2}w \right)^2,
$$
where $L:=b-a$ is the length of the beam and $D := E \frac{h^3}{12(1-\nu^2)}$ is the modulus of flexural rigidity, with $E$ the Young's modulus, $\nu$ the Poisson's ratio ($0<\nu<\frac{1}{2}$ in physical situations), and $h$ the cross-sectional area ($h:=\pi r^2$ for cylindric beam of radius $r$).

Note that to include a tension $T_0$, as done in~\cite[Eq.~(10)]{bendimerad2024implicit}, and model pre-stressed beams, it requires adding more potential energy terms, following~\cite[Eq.~(2)]{ducceschi2019conservative}. As this is not the purpose of this work, we restrict ourselves to the simpler models without tension.

The kinetic energy reads:
$$ \mathcal{K}(w) = \frac{1}{2} \rho h\int_a^b \underbrace{\left(\frac{\partial}{\partial t} w \right)^2}_{\text{streching part}} + \underbrace{\frac{h^2}{12} \left( \frac{\partial}{\partial t}  \left (\frac{\partial}{\partial x} w \right ) \right )^2}_{\text{bending part}} \, \d x. $$

\paragraph{Principle of virtual work} The equations of motion for $w = u_z$ are obtained by setting to zero the first variation of the Lagrangian:
$$
\mathcal{L} = \int_0^T \mathcal{K} + \mathcal{W} - \mathcal{P} \, \d t,
$$
where $\mathcal{W}$ is the work done on the beam by external forces that contributes to the bending.

More precisely, let $\eta = \delta w$ be any admissible variation of $w$ in $\mathcal{C}([0,T]\times[a,b])$ with compact support, then:
$$
\delta \mathcal{P} = D \int_a^b \frac{\partial^4}{\partial x^4} w \, \eta \, \d x,
$$
$$
\int_0^T \delta \mathcal{K} \, \d t = \rho h \int_0^T \int_a^b \left( - \frac{\partial^2}{\partial t^2} w + \frac{h^2}{12} \frac{\partial^4}{\partial x^2 \partial t^2} w \right) \eta \, \d x \, \d t,
$$
and
$$
\delta \mathcal{W} = \int_a^b f \, \eta \, \d x,
$$
where $f$ is the external force.

Since $\eta$ is arbitrary, the equation of motion for the Euler-Bernoulli beam is obtained by setting $\delta \mathcal{L} = \int_0^T \delta \mathcal{K} + \delta \mathcal{W} - \delta \mathcal{P} \, \d t= 0$, which then reads:
\begin{equation}\label{eq:EB-derivation}
\rho h \left(1- \frac{h^2}{12} \frac{\partial^2}{\partial x^2}\right) \frac{\partial^2}{\partial t^2} w + D \frac{\partial^4}{\partial x^4} w = f.
\end{equation}

\section{Incompressible Navier-Stokes equation}

\subsection{Time discretization}

Similarly to \cite{de2019inclusion}, we will consider a staggered approach where $\overline{\psi}$,  $u_D^3$ and $\widetilde u_D$ are evaluated at half time steps $t_{k+1/2}$ and where $\overline{\omega}$ and $u_D^5$ are evaluated at whole timesteps $t_k$. The modulated Dirac structure is then evaluated for each variable at time $t$ using the value of the other variable at the previous time $t - \d t /2.$ 

Then, the implicit midpoint rule is used on the state variable of the linearized equation, and constraints are imposed at the end of the timestep interval. Finally, initialization is carried out by integrating $\overline{\psi}(t_0)$ to the first half time step $t_{1/2}$. Denoting by $\alpha_{t_{k}}$ the variable $\alpha(t)$ discretized at time $t_k$, we have:
\begin{lemma} The discrete power balance and enstrophy balance read:
    $$\mathcal{K}^d(\overline{\psi}_{t_{k+1/2}}) -  \mathcal{K}^d(\overline{\psi}_{t_{k-1/2}}) = - \d t \, \frac{(\overline{\psi}_{t_{k+1/2}} + \overline{\psi}_{t_{k-1/2}})}{2}^\top \bm{R}^1 \frac{(\overline{\psi}_{t_{k+1/2}} + \overline{\psi}_{t_{k-1/2}})}{2},$$
    $$ \mathcal{E}^d(\overline{\omega}_{t_{k+1}}) - \mathcal{E}^d(\overline{\omega}_{t_{k}}) = - \d t \, \frac{(\overline{\omega}_{t_{k+1}} + \overline{\omega}_{t_{k-1}})}{2}^\top \bm{R}^2  \frac{(\overline{\omega}_{t_{k+1}} + \overline{\omega}_{t_{k-1}})}{2} + \d t \, u_{D t_{k+1}}^\top\bm{B}^{5\top} \frac{\overline{\omega}_{t_{k+1}} + \overline{\omega}_{t_{k}}}{2}.  $$
\end{lemma}

\subsection{Mesh}

Figure \ref{fig:dipole-collision-mesh} presents the mesh used for the simulation of the normal dipole collision experiment. The mesh is finer along the right and left boundary and along the path of the dipole, this allows for faster computation compared to a homogeneous mesh. It is composed of 4448 nodes.
\begin{figure}[!ht]
    \centering
    \includegraphics[width=0.45\linewidth]{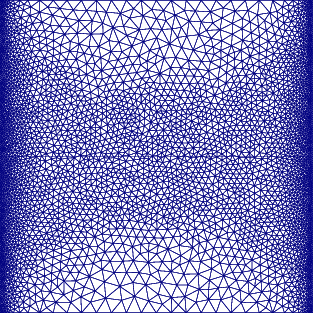}
    \caption{Mesh used for the dipole collision experiment}
    \label{fig:dipole-collision-mesh}
\end{figure}

\bibliographystyle{elsarticle-num}
\bibliography{biblio}

\end{document}